\documentclass{conm-p-l}
\usepackage{amssymb,latexsym,amsmath}
\usepackage{bbm}
\usepackage[all]{xy}
\usepackage{graphicx}
\usepackage{enumerate}
\usepackage{amscd}

\DeclareMathAlphabet{\Ma}{U}{msa}{m}{n}
\DeclareMathAlphabet{\Mb}{U}{msb}{m}{n}
\DeclareMathAlphabet{\Meuf}{U}{euf}{m}{n}
\DeclareSymbolFont{ASMa}{U}{msa}{m}{n}
\DeclareSymbolFont{ASMb}{U}{msb}{m}{n}

\def\mr #1.{\mathrm{#1\,}}
\def\mrt #1.{\mathrm{\mbox{\tiny #1\,}}}
\def\mt #1.{{\mbox{\tiny $#1$}}}
\def\ms #1.{{\mbox{\small $#1$}}}
\def\ol #1.{\overline{#1}}

\def\N{\mathbb{N}}
\def\ol{\overline}

\def\1{\mathbbm 1}

\def\1{\mathbbm 1}

\def\ot #1.{{\got{#1}}}
\def\got#1{\Meuf{#1}}
\def\al #1.{{\mathcal{#1}}}

\theoremstyle{plain}            
\newtheorem{theorem}{Theorem}[section]
\newtheorem*{maintheorem*}{Main Theorem}
\newtheorem*{maintheorem.}{Main Theorem~1'}

\newtheorem{proposition}[theorem]{Proposition}
\newtheorem{lemma}[theorem]{Lemma}
\newtheorem{corollary}[theorem]{Corollary}

\newtheorem{blank}[theorem]{}

\theoremstyle{definition}       
\newtheorem{definition}[theorem]{Definition}

\theoremstyle{remark}
\newtheorem{remark}[theorem]{Remark}
\newtheorem{example}[theorem]{Example}

\newcommand{\mL}{\mathcal L}               

\newcommand{\WA}{\mathrm{W}(A)}
\newcommand{\W}{\mathrm{W}}
\newcommand{\V}{\mathrm{V}}

\newcommand{\dt}{d_{\tau}}
\newcommand{\ups}{\upsilon}

\def\mr #1.{\mathrm{#1\,}}
\def\mrt #1.{\mathrm{\mbox{\tiny #1\,}}}
\def\mt #1.{{\mbox{\tiny $#1$}}}
\def\ms #1.{{\mbox{\small $#1$}}}
\def\ol #1.{\overline{#1}}

\def\N{\mathbb{N}}

\newfont{\Kcal}{cmsy6 scaled 1000}
\newfont{\Kgot}{eufm6 scaled 1000}

\def\Kbegin{\begin{equation} \left. \begin{array}{rcl}}
\def\Kend{\end{array} \right\} \end{equation}}

\DeclareMathSymbol{\hsemi}{\mathord}{ASMb}{"6E}
\newcommand{\semi}[2]{\mbox{$#1\kern.1em\hsemi\kern.1em#2$}}

\def\vplatz#1{{\rule{0mm}{#1}}}

\def\LA{\left\langle\bgroup}
\def\LE{\left[\bgroup}
\def\LG{\left\{\bgroup}
\def\LR{\left(\bgroup}
\def\RA{\egroup^{\rule{0mm}{0mm}}\right\rangle}
\def\RE{\egroup^{\rule{0mm}{2mm}}\right]}
\def\RG{\egroup^{\rule{0mm}{2mm}}\right\}}
\def\RR{\egroup^{\rule{0mm}{2mm}}\right)}
\def\Ldummy{\left.\bgroup}
\def\Rdummy{\egroup^{\rule{0mm}{2mm}}\right.}

\def\ccr#1{\mbox{{\rm CCR$\left({#1}^{\vplatz{1.5mm}}\right)$}}}

\def\ccr #1,#2.{\overline{\Delta(#1,\,#2)}}

\def\b #1.{{\bf #1}}
\def\cross#1.{\mathrel{\mathop{\times}\limits_{#1}}}

\def\N{\Mb{N}}

\def\wwh #1.{\widehat{#1}}
\def\wt #1.{\widetilde{#1}}

\def\cross #1.{\mathrel{\raise 3pt\hbox{$\mathop\times\limits_{#1}$}}}
\def\set #1,#2.{\left\{\,#1\;\bigm|\;#2\,\right\}}
\def\b #1.{{\bf #1}}
\def\ker{{\rm Ker}\,}

\def\ol #1.{\overline{#1}}
\def\rn#1.{\romannumeral{#1}}

\def\s #1.{_{\smash{\lower2pt\hbox{\mathsurround=0pt $\scriptstyle #1$}}\mathsurround=3pt}}
\def\bra #1,#2.{{\left\langle #1,\,#2\right\rangle_{\al A.}}}

\def\XP#1!{\renewcommand{\baselinestretch}{.7}\marginpar{{\footnotesize #1}\hfil}
\renewcommand{\baselinestretch}{1.5}}
\def\XB{\marginpar{
{\footnotesize\bf Change~starts----}\lower 11pt\hbox{\mathsurround=0pt$
\!\!\displaystyle{
\Bigg\downarrow}$\mathsurround=3pt}}}
\def\XE{\marginpar{{\footnotesize\bf Change~ends-----}\raise 10pt\hbox{\mathsurround=0pt$
\!\!\displaystyle{
\Bigg\downarrow}$\mathsurround=3pt}}}

\DeclareMathSymbol{\hsemi}{\mathord}{ASMb}{"6E}

\def\LA{\left\langle\bgroup}
\def\LE{\left[\bgroup}
\def\LG{\left\{\bgroup}
\def\LR{\left(\bgroup}
\def\RA{\egroup^{\rule{0mm}{0mm}}\right\rangle}
\def\RE{\egroup^{\rule{0mm}{2mm}}\right]}
\def\RG{\egroup^{\rule{0mm}{2mm}}\right\}}
\def\RR{\egroup^{\rule{0mm}{2mm}}\right)}
\def\Ldummy{\left.\bgroup}
\def\Rdummy{\egroup^{\rule{0mm}{2mm}}\right.}

\def\ccr#1{\mbox{{\rm CCR$\left({#1}^{\vplatz{1.5mm}}\right)$}}}

\newcommand{\Cu}{\mathrm{Cu}(A)}
\newcommand{\mK}{\mathcal{K}_A}
\newcommand{\M}{\mathcal{M}}
\newcommand{\CC}{\subset\subset}

\newcommand{\aff}{\mathrm{Aff}}
\newcommand{\laff}{\mathrm{LAff}}

\def\ccr #1,#2.{\overline{\Delta(#1,\,#2)}}

\def\b #1.{{\bf #1}}
\def\cross#1.{\mathrel{\mathop{\times}\limits_{#1}}}

\def\N{\Mb{N}}

\def\wwh #1.{\widehat{#1}}
\def\wt #1.{\widetilde{#1}}

\def\cross #1.{\mathrel{\raise 3pt\hbox{$\mathop\times\limits_{#1}$}}}
\def\set #1,#2.{\left\{\,#1\;\bigm|\;#2\,\right\}}
\def\b #1.{{\bf #1}}
\def\ker{{\rm Ker}\,}

\def\ol #1.{\overline{#1}}
\def\rn#1.{\romannumeral{#1}}

\def\s #1.{_{\smash{\lower2pt\hbox{\mathsurround=0pt $\scriptstyle #1$}}\mathsurround=3pt}}
\def\bra #1,#2.{{\left\langle #1,\,#2\right\rangle_{\al A.}}}

\def\XP#1!{\renewcommand{\baselinestretch}{.7}\marginpar{{\footnotesize #1}\hfil}
\renewcommand{\baselinestretch}{1.5}}
\def\XB{\marginpar{
{\footnotesize\bf Change~starts----}\lower
11pt\hbox{\mathsurround=0pt$ \!\!\displaystyle{
\Bigg\downarrow}$\mathsurround=3pt}}}
\def\XE{\marginpar{{\footnotesize\bf Change~ends-----}\raise 10pt\hbox{\mathsurround=0pt$
\!\!\displaystyle{ \Bigg\downarrow}$\mathsurround=3pt}}}

\newcommand{\T}{\mathrm{T}}
\newcommand{\Z}{\mathcal{Z}}

\title[$K$-Theory]{$K$-theory for operator algebras.
Classification of C$^*$-algebras.}
\author{Pere Ara}
\address{Department of Mathematics, Universitat Aut\`onoma de Barcelona, 08193 Bellaterra (Barcelona), Spain}
\email{para@mat.uab.cat}
\author{Francesc Perera}
\address{Department of Mathematics, Universitat Aut\`onoma de Barcelona, 08193 Bellaterra (Barcelona), Spain}
\email{perera@mat.uab.cat}
\author{Andrew S. Toms}
\address{Department of Mathematics, University of York, Toronto, Canada}
\email{atoms@mathstat.yorku.ca}
\date{\today}

\begin{document}
\maketitle
\tableofcontents
\begin{abstract} 
In this article we survey some of the recent goings-on in the classification programme of C$^*$-algebras, following the interesting link found between the Cuntz semigroup and the classical Elliott invariant and the fact that the Elliott conjecture does not hold at its boldest. We review the construction of this object both by means of positive elements and via its recent interpretation using countably generated Hilbert modules (due to Coward, Elliott and Ivanescu). The passage from one picture to another is presented with full, concise, proofs. We indicate the potential role of the Cuntz semigroup in future classification results, particularly for non-simple algebras.
\end{abstract}

\section{Introduction}

\vskip1cm

Classification is a central and recurring theme in Mathematics. It
has appeared throughout history in various guises, and the common
feature is usually the replacement of a complicated object with a
less complicated (or a computable) one, which is invariant under
isomorphism. In an ideal setting, we seek for a complete invariant,
i.e. one that completely captures the nature of isomorphism of our
original objects. Another part of our understanding asks for the
range this invariant has. These lecture notes aim at introducing the
classification problem for (nuclear) C$^*$-algebras, that has been
around for the last couple of decades, and has attracted attention
(and efforts) of many mathematicians.

The fact that the Elliott Conjecture, in its original format, has
dramatic counterexamples has had, as a consequence, a resurgence of
the programme in a different (i.e. new) light. This recasting of the
Conjecture is very intimately connected to the Cuntz semigroup. We
hope to convince the reader that this can be a useful tool in future
classification results, or that at least gives new conceptual
insight into the ones already existing. At some stage it might seem
that the new point of view diminishes the value of some of the
pioneering earlier work of ``classificators'' (i.e. those who, with
their bare hands, pushed their way to classify big classes of
algebras). This, however, is not the case as some of the results
(for example, to prove $\Z$-stability of certain classes of
algebras) necessitate the full force of previous classification
theorems.

In broad outline, the purpose of these notes is twofold. On the one hand, we give a self-contained
account on the basics of the Cuntz semigroup, a powerful device that strives to capture the nature of equivalence classes of positive elements. Having a technical essence, we find it useful to collect some of its various features under one roof, with proofs or sketches of proofs. This is covered in the first three sections below, of which the first one introduces the Cuntz semigroup following a more `classical' approach (see, e.g. \cite{Cu}, \cite{Rfunct}). 

In our journey we shall dwell on an important recent development of the theory, namely the introduction by Coward, Elliott and Ivanescu \cite{cei} of a stable version of the Cuntz semigroup, denoted $\Cu$, that has interesting connections with the theory of Hilbert modules (of which the basic prerequisites are given in the third section). As a result we may view the Cuntz semigroup as a continuous functor from C$^*$-algebras to a new category $\mathsf{Cu}$ of ordered semigroups. We provide full details.

On the other hand, we shall survey some of the most prominent applications of this invariant to the classification programme, with emphasis on its description in terms of K-Theory and traces for a wide class of algebras.

Portions of these notes are shamelessly based on \cite{bpt} and
\cite{pt} (and also \cite{ET} in part, which constitutes a very nice
survey of the whole story).
\vskip2cm


\section{The Cuntz semigroup: Definitions,
historical origin and technical devices}

\vskip1cm

\subsection{Introduction}

Our main aim here is to define the Cuntz semigroup, as it was
introduced historically. We also offer proofs of technical facts
which come as very useful when dealing with this object. These
appear in a rather scattered way through the literature, both in
space and time. Our approach is not linear, but tries to benefit
from more recent results that allow to give somewhat different
proofs, in a self-contained manner.

We start by giving the definition of Cuntz comparison and the Cuntz
semigroup $\WA$ for a C$^*$-algebra $A$, and we establish a few,
fundamental basic facts, that will be used over and over. We then
proceed to make an explicit connection with the usual Murray-von
Neumann comparison of projections $\V(A)$, and examine cases where
one can think of the Cuntz semigroup as a natural extension of
$\V(A)$. The relationship between projections and general positive
elements is also analysed, and in particular we prove a cancellation
theorem (due to R\o rdam and Winter).

The enveloping group of $\WA$, also known as $\mathrm{K}_0^*(A)$ is
related to the enveloping group of a certain subsemigroup of $\WA$,
that made of the so-called purely positive elements.

States on the semigroup, known as dimension functions, are of great
importance, in particular those that are lower semicontinuous. We
gather together some of the deep results that go back to the seminal
paper of Blackadar and Handelman, connecting quasi-traces and lower
semicontinuous dimension functions. Some details of the proofs are
also included. Partly for expository reasons, partly for its
relevance, and partly for its beauty, we shall also recover a
classical result due to Cuntz and Handelman, which relates stable
finiteness to the existence of dimension functions.

We close this section with the computation of $\WA$ in some cases,
notably those C$^*$-algebras that have the property of real rank
zero. In this situation, the computation can be done in terms of the
underlying projection semigroup in a rather neat way.

\subsection{C$^*$-algebras}\label{cstar}

\vspace{2mm}
\begin{definition}  A {\bf C$^*$-algebra} is a Banach algebra $A$ equipped with an involution satisfying the C$^*$-algebra identity:
\[
\Vert x^*x \Vert = \Vert x \Vert^2, \ \forall x \in A.
\]
\end{definition}
The norm on a C$^*$-algebra is completely determined by its
algebraic structure, and is unique.  {\bf Homomorphisms are assumed
to be $*$-preserving}, and are automatically contractive.  A theorem
of Gelfand, Naimark, and Segal (cf. \cite[Theorem 2.12]{ALP}) states
that every C$^*$-algebra is isometrically $*$-isomorphic to a closed
sub-$*$-algebra of the algebra of bounded linear operators on a
Hilbert space;  if the C$^*$-algebra is separable, then the Hilbert
space may be chosen to be separable.

\subsubsection{Units and approximate units} C$^*$-algebras need not
be unital.  If $A$ is not unital, then we may adjoin a unit to $A$
as follows: set $\tilde{A} = \mathbb{C}
\oplus A$, and equip $\tilde{A}$ with co-ordinatewise addition, multiplication given by
\[
(\lambda,a) \cdot (\gamma,b) = (\lambda \gamma, ab + \lambda b +
\gamma a),
\]
adjoint given by $(\lambda,a)^* = (\bar{\lambda},a^*)$, and norm
given by
\[
\Vert (\lambda,a) \Vert = \sup \{ \Vert \lambda b + ab \Vert \  | \
\Vert b \Vert = 1 \}.
\]

There is a notion of positivity for elements in a C$^*$-algebra,
giving rise to an order structure--see \cite{ALP}.

\begin{definition}
Let $A$ be a C$^*$-algebra.  An {\bf approximate unit} for $A$ is a
net $(h_\lambda)_{\lambda \in \Lambda}$ of elements of $A$ with the
following properties:
\[
0\leq  h_\lambda  \leq 1, \ \forall \lambda \in \Lambda; \ \ \ \Vert
h_\lambda x - x \Vert \to 0, \ \forall x \in A.
\]
If $h_\lambda
\leq h_\mu$ whenever $\lambda \leq \mu$, then we say that
$(h_\lambda)$ is {\bf increasing}.
\end{definition}

It turns out that every C$^*$-algebra has an increasing approximate
unit.

\subsection{Examples}\label{examples}

\subsubsection{Commutative C$^*$-algebras}  A theorem of Gelfand and
Naimark (cf. \cite[Theorem 2.13]{ALP}) asserts that a commutative
C$^*$-algebra $A$ is isomorphic to the algebra of continuous
functions on a locally compact Hausdorff space $X$ which vanish at
infinity, denoted by $\mathrm{C}_0(X)$;  the involution is given by
pointwise complex conjugation.  If $A$ is unital, then $X$ is
compact.  The correspondence of Gelfand and Naimark is in fact a
natural equivalence of categories.  In the unital case, for
instance, a continuous map $f:X \to Y$ between compact Hausdorff
spaces $X$ and $Y$ induces a homomorphism $\bar{f}:\mathrm{C}(Y) \to
\mathrm{C}(X)$ via the formula
\[
\bar{f}(a)(x) = a(f(x)), \ \forall x \in X,
\]
and every homomorphism from $\mathrm{C}(Y)$ to $\mathrm{C}(X)$
arises in this way.

\subsubsection{Finite-dimensional C$^*$-algebras}  The algebra
$B(\mathcal{H})$ of bounded linear operators on a Hilbert space
$\mathcal{H}$ is a C$^*$-algebra under the operator norm and the
usual adjoint.  If $\mathcal{H}$ is finite-dimensional, then
$B(\mathcal{H})$ is isomorphic to the algebra of $n \times n$
matrices over $\mathbb{C}$, with the adjoint operation being complex
conjugate transposition.  The latter object will be denoted by
$\mathrm{M}_n(\mathbb{C})$.  If a C$^*$-algebra $A$ is
finite-dimensional (as a vector space), then it is necessarily
isomorphic to a C$^*$-algebra of the form
\[
\mathrm{M}_{n_1}(\mathbb{C}) \oplus \mathrm{M}_{n_1}(\mathbb{C})
\oplus \cdots \oplus \mathrm{M}_{n_k}(\mathbb{C}).
\]

\subsubsection{Inductive limits}  For each $i \in \mathbb{N}$, let
$A_i$ be a C$^*$-algebra, and let $\phi_i:A_i \to A_{i+1}$ be a
homomorphism.  We refer to the sequence $(A_i,\phi_i)_{i \in
\mathbb{N}}$ as an inductive sequence.  There is a C$^*$-algebra $A$
and a sequence of maps $\gamma_i:A_i \to A$ satisfying the following
universal property:
\begin{enumerate}
\item[(i)] the diagram
\[
\xymatrix{
& A & \\
{A_i}\ar[ur]^{\gamma_i}\ar[rr]^{\phi_i} &&
{A_{i+1}}\ar[ul]_{\gamma_{i+1}} }
\]
commutes for each $i \in \mathbb{N}$;
\item[(ii)] if $B$ is a C$^*$-algebra and $(\eta_i)_{i \in \mathbb{N}}$ a sequence of homomorphisms such that $\eta_i:A_i \to B$ and
\[
\xymatrix{
& B & \\
{A_i}\ar[ur]^{\eta_i}\ar[rr]^{\phi_i} &&
{A_{i+1}}\ar[ul]_{\eta_{i+1}} }
\]
commutes for each $i \in \mathbb{N}$, then there is a homomorphism
$\eta:A \to B$ such that $\eta_i = \eta \circ \gamma_i$ for each $i
\in \mathbb{N}$.
\end{enumerate}
$A$ is called the {\bf inductive limit} of the inductive sequence
$(A_i,\phi_i)_{i \in \mathbb{N}}$, and we write $A = \lim_{i \to
\infty} (A_i, \phi_i)$.  It will be convenient to have the
abbreviation
\[
\phi_{i,j} := \phi_{j-1} \circ \phi_{j-2} \circ \cdots \circ \phi_i
\]
for any $j > i$.

The limit C$^*$-algebra $A$ can be described concretely.  Let
$A^\circ$ denote the set of all sequences
\[
\mathbf{a} = (\underbrace{0,\ldots,0}_{i-1 \ \mathrm{times}},a,
\phi_i(a), \phi_{i,i+1}(a), \ldots ),
\]
where $a \in A_i$ and $i$ ranges over $\mathbb{N}$.  Equip $A^\circ$
with the co-ordinatewise operations and the norm
\[
\Vert \mathbf{a} \Vert = \lim_{j \to \infty} \Vert \phi_{i,j}(a)
\Vert.
\]
$A$ is the completion of the pre-C$^*$-algebra $A^\circ$ in this
norm.

\subsection{The Cuntz semigroup}\label{SubSecCuntz}

\subsubsection{Cuntz comparison for general elements}

From here on we make the blanket assumption that all C$^*$-algebras
are separable unless otherwise stated or obviously false.

Let $A$ be a C$^*$-algebra, and let $\mathrm{M}_n(A)$ denote the $n
\times n$ matrices whose entries are elements of $A$.  If $A =
\mathbb{C}$, then we simply write $\mathrm{M}_n$. Let
$\mathrm{M}_{\infty}(A)$ denote the algebraic limit of the direct
system $(\mathrm{M}_n(A),\phi_n)$, where $\phi_n:\mathrm{M}_n(A) \to
\mathrm{M}_{n+1}(A)$ is given by
\[
a \mapsto \left( \begin{array}{cc} a & 0 \\ 0 & 0 \end{array}
\right).
\]
Let $\mathrm{M}_{\infty}(A)_+$ (resp. $\mathrm{M}_n(A)_+$) denote
the positive elements in $\mathrm{M}_{\infty}(A)$ (resp.
$\mathrm{M}_n(A)$). For positive elements $a$ and $b$ in
$\mathrm{M}_{\infty}(A)$, write $a\oplus b$ to denote the element
$\left( \begin{array}{cc} a & 0 \\ 0 & b \end{array} \right)$, which
is also positive in $\mathrm{M}_{\infty}(A)$. We shall as customary
use $\widetilde{A}$ to refer to the minimal unitization of $A$, and
$\mathcal{M}(A)$ will stand for the multiplier algebra of $A$.

\begin{definition}
Given $a,b \in \mathrm{M}_{\infty}(A)_+$, we say that $a$ is {\it
Cuntz subequivalent} to $b$ (written $a \precsim b$) if there is a
sequence $(v_n)_{n=1}^{\infty}$ of elements of
$\mathrm{M}_{\infty}(A)$ such that
\[
\Vert v_nbv_n^*-a\Vert \stackrel{n \to \infty}{\longrightarrow} 0.
\]
We say that $a$ and $b$ are {\it Cuntz equivalent} (written $a \sim
b$) if $a \precsim b$ and $b \precsim a$.
\end{definition}

That Cuntz equivalence is an equivalence relation is an exercise
that we leave to the reader. We write $\langle a \rangle$ for the
equivalence class of $a$.
\begin{definition}
The object
\[
\WA := \mathrm{M}_{\infty}(A)_+/ \sim
\]
will be called the {\it Cuntz semigroup} of $A$.
\end{definition}

Observe that $\WA$ becomes a positively ordered Abelian monoid (i.e.
$\langle 0\rangle\leq x$ for any $x\in\WA$) when equipped with the
operation
\[
\langle a \rangle + \langle b \rangle = \langle a \oplus b \rangle
\]
and the partial order
\[
\langle a \rangle \leq \langle b \rangle \Leftrightarrow a \precsim
b.
\]

Every Abelian semigroup $M$ can be ordered using the so-called
algebraic ordering, so that $x\leq y$ in $M$ if, by definition,
$x+z=y$ for some $z$ in $M$. The fact that $\mathrm{W}(A)$ is
positively ordered that the order given extends the algebraic
ordering, that is, if $\langle a\rangle +\langle c\rangle=\langle
b\rangle$, then $\langle a\rangle\leq \langle b\rangle$.

Given $a$ in $\mathrm{M}_{\infty}(A)_+$ and $\epsilon > 0$, we
denote by $(a-\epsilon)_+$ the element of $C^*(a)$ corresponding
(via the functional calculus) to the function
\[
f(t) = \mathrm{max}\{0,t-\epsilon\}, \ t \in \sigma(a).
\]
(Recall that $\sigma(a)$ denotes the spectrum of $a$.) By the
functional calculus, it follows in a straightforward manner that
$((a-\epsilon_1)_+-\epsilon_2)_+=(a-(\epsilon_1+\epsilon_2))_+$.

In order to gain some insight into the meaning of Cuntz
subequivalence we offer the following:

\begin{proposition}
\label{prop:commutative} Let $X$ be a compact Hausdorff space, and
let $f$, $g\in C(X)_+$. Then $f\precsim g$ if and only if
$\mathrm{supp}(f)\subseteq\mathrm{supp}(g)$.
\end{proposition}
\begin{proof}
The ``only if'' direction is trivial. Assume then that
$\mathrm{supp}(f)\subseteq\mathrm{supp}(g)$, and we are to show that
$f\precsim g$. Given $\epsilon>0$, set $K=\{x\in X\mid
f(x)\geq\epsilon\}$, a compact subset of $X$. Necessarily then
$K\subset\mathrm{supp}(g)$. Since $g$ is continuous on $K$ and $K$
is compact, there is a positive $\delta$ such that $g>\delta$ on
$K$. Put $U=\{x\in X\mid g(x)>\delta\}$, which is open and contains
$K$. Use Urysohn's lemma to find a function $h$ such that $h_{|K}=1$
and $h_{|X\setminus U}=0$, and then consider the function $e$
defined as $e_{|U}=(\frac{h}{g})_{|U}$ and $e_{|X\setminus U}=0$.
Then one may check that $e$ is continuous and
\[
\Vert f-egf\Vert<\epsilon
\]
\end{proof}

The following corollary tells us that we are on the right track.

\begin{corollary}
For any {\rm C}$^*$-algebra $A$ and any $a\in A_+$, we have $a\sim
a^n$ ($n\in\mathbb{N}$). For any $a\in A$, we have $aa^*\sim a^*a$.
\end{corollary}
\begin{proof}
The first part of the statement follows from Proposition
\ref{prop:commutative} above, by noticing that $a\in C^*(a)\cong
C^*(\sigma(a))$, and it is clear that $a$ and $a^n$ have the same
support.

Next, for any $a$ we have $aa^*\sim (aa^*)^2=aa^*aa^*\precsim a^*a$,
and by symmetry, $aa^*\sim a^*a$.
\end{proof}

We note below that the natural order on $\WA$ does not agree with
the algebraic order, except in trivial cases. Rather than being a
drawback, this makes it the more interesting.

\begin{example}
Consider $A=C(X)$, where $X$ is a compact Hausdorff space such that
$[0,1]\subseteq X\subset \mathbb R$, and take functions $f,g\in A$
defined as follows: $g$ is linear (increasing) in $[0,
\frac{1}{3}]$, $1$ in $[\frac{1}{3}, \frac{2}{3}]$, linear
(decreasing) in $[\frac{2}{3}, 1]$, and zero elsewhere; $f$ is
linear (increasing) in $[\frac{1}{3}, \frac{1}{2}]$, linear
(decreasing) in $[\frac{1}{2}, \frac{2}{3}]$, and zero elsewhere.

It is clear from Proposition \ref{prop:commutative} that $f\precsim
g$. If there is $y=\langle h\rangle \in \W(A)$, with $h\in
M_n(A)_+$, such that $\langle f\rangle +y=\langle g\rangle$, then we
have:
\[
\left(\begin{array}{cccc} f & 0 & \ldots & 0 \\0&h_{11}&\ldots&h_{1n}\\
\vdots&\vdots&\ddots&\vdots\\0&h_{n1}&\ldots&h_{nn}
\end{array}\right) \sim \left(\begin{array}{cc}g&0\\0&0
\end{array}\right)\,.
\]
Thus there exist functions ${\alpha}_0^k, {\alpha}_1^k,\ldots
,{\alpha}_n^k, \,\,k\in \mathbb N$ with
\[
\left(\begin{array}{ccc}{\alpha}_0^k&\ldots&0\\\vdots&\ddots&\vdots\\{\alpha}_n^k&\ldots&0
\end{array}\right) \left(\begin{array}{cc}g&0\\0&0 \end{array}\right) \left(\begin{array}{ccc}{\overline{{\alpha}_0^k}}&\ldots&{\overline{{\alpha}_n^k}}\\\vdots&\ddots&\vdots\\0&\ldots&0 \end{array}\right) \to
\left(\begin{array}{cccc}f&0&\ldots&0\\0&h_{11}&\ldots&h_{1n}\\
\vdots&\vdots&\ddots&\vdots\\0&h_{n1}&\ldots&h_{nn}
\end{array}\right)\,.
\]
This means that ${\alpha}_0^k\overline{{\alpha}_i^k}\to 0, \,\,
{\vert {\alpha}_0^k\vert }^2g\to f, \,\,
{\alpha}_i^k\overline{{\alpha}_j^k}g\to h_{ij}$ when $k\to \infty $.
Thus
\[
f=\lim {\vert {{\alpha}_0^k}\vert }^2g, \,\, \vert h_{ij} \vert
=\lim \vert {\alpha}_i^k{\alpha}_j^k\vert g
\]
and therefore
\[
f\vert h_{ij}\vert =\lim\limits_{k} g{\vert {{\alpha}_0^k}\vert
}^2\vert {\alpha}_i^k{\alpha}_j^k\vert g =\lim\limits_{k} (g\vert
{\alpha}_0^k{\alpha}_i^k\vert )(g\vert {\alpha}_0^k{\alpha}_j^k\vert
)=0\,.
\]
Thus $f\vert h_{ij}\vert =0$ for all $i,j$, and
$\mathrm{supp}(h_{ij})=\mathrm{supp}(\vert h_{ij}\vert )$, so that
\[
\mathrm{supp}(f)\sqcup \bigl( {\cup}_{ij}
\mathrm{supp}(h_{ij})\bigr) \subseteq \mathrm{supp}(g)\,.
\]

On the other hand, for $0\leq i,j\leq n$, there exist functions
$a_{ij}^k$ such that
\[
(a_{ij}^k)\left(\begin{array}{cc}f&0\\0&h\end{array}\right)(\overline{a_{ji}^k})
\to \left(\begin{array}{cc}g&0\\0&0\end{array}\right)\,.
\]
In particular, we have that
\[
{\vert a_{00}^k\vert}^2f+ \bigl( {\vert
a_{01}^k\vert}^2h_{11}+\ldots +\overline{a_{01}^k}
a_{0n}^kh_{n1}\bigr)+\ldots +\bigl( \overline{a_{0n}^k}
a_{01}^kh_{1n}+\ldots +{\vert a_{0n}^k\vert }^2h_{nn}\bigr) \to g
\]
as $k\to \infty$, whence $\mathrm{supp}(g)\subseteq\mathrm{supp}(f)
\sqcup \bigl( {\cup}_{ij} \mathrm{supp}(h_{ij})\bigr) $.

Now, as $\mathrm{supp}(g)$ is connected and $\mathrm{supp}(f)\neq
\emptyset $, we see that $\mathrm{supp}(f)=\mathrm{supp}(g)$, but by
construction $g\not\precsim f$.
\end{example}

\begin{lemma}
\label{lem:leq} If $0\leq a\leq b$, then we also have $a\precsim b$.
\end{lemma}
\begin{proof}
This follows, for example, quoting a lemma by Handelman
(see~\cite{han}), where it is shown that if $a^*a\leq b^*b$, there
is then a sequence $(z_n)$ of elements with $\Vert z_n\Vert \leq 1$
such that $a=\lim\limits_n z_nb$.

We apply this to write $a^{\frac{1}{2}}=\lim\limits
z_nb^{\frac{1}{2}}$, whence $a=\lim\limits_n z_nbz_n^*$.
\end{proof}

\begin{corollary}
If $a\in A_+$ and $\epsilon>0$, we have $(a-\epsilon)_+\precsim a$.
\end{corollary}

The behaviour of Cuntz subequivalence under addition is good under
orthogonality conditions.
\begin{lemma}
\label{Cuntz} If $a$, $b\in A_+$, then $a+b\precsim a\oplus b$, and
if further $a\perp b$, then $a+b\sim a\oplus b$.
\end{lemma}
\begin{proof}
If $x=\left( \begin{array}{cc} a^{\frac{1}{2}} &
b^{\frac{1}{2}}\end{array}\right)$, we then have
\[
a+b=xx^*\sim x^*x=\left( \begin{array}{cc}  a &
b^{\frac{1}{2}}a^{\frac{1}{2}} \\ a^{\frac{1}{2}}b^{\frac{1}{2}} &
b\end{array}\right)\precsim a\oplus b\,,
\]
and if $a\perp b$, then the latter $\precsim$ is in fact an
equality.
\end{proof}

One of the main technical advantages of Cuntz comparison is that it
allows to decompose elements up to arbitrary approximations. The
theorem below (proved in~\cite{KR2}) follows this spirit. For the
proof we need some preliminary results.

\begin{lemma}
\label{Pedersen} Let $A$ be a {\rm C}$^*$-algebra, let $x$, $y\in A$
and $a\in A_+$. If $x^*x\leq a^{\alpha}$ and $yy^*\leq a^{\beta}$,
where $\alpha$, $\beta>0$ and $\alpha+\beta>1$, then there is an
element $u\in A$ such that
\[
u_n:=x\left(\frac{1}{n}+a\right)^{-\frac{1}{2}}y\stackrel{n \to
\infty}{\longrightarrow} u\,,
\]
and $\Vert u\Vert\leq \Vert a^{\frac{\alpha+\beta-1}{2}}\Vert$.
\end{lemma}
\begin{proof} (Outline.)
Put
$d_{nm}=\left(\frac{1}{n}+a\right)^{-\frac{1}{2}}-\left(\frac{1}{m}+a\right)^{-\frac{1}{2}}$,
which is a self-adjoint element.

A calculation using the C$^*$-equation and the inequalities
involving $x^*x$ and $yy^*$ implies that $\Vert u_n-u_m\Vert\leq
\Vert a^{\frac{\alpha+\beta}{2}}d_{nm}\Vert$.

Now let $f_n(t):=\left(\frac{1}{n}+t\right)^{-\frac{1}{2}}$, a
continuous real-valued function defined on $\sigma(a)$. We have that
$(f_n)$ is an increasing sequence of functions that converges
pointwise to $f(t)=t^{\frac{\alpha+\beta-1}{2}}$, also continuous.
It follows from Dini's Theorem that this convergence is in fact
uniform, whence $d_{nm}a^{\frac{\alpha+\beta}{2}}\to 0$. This
implies that $(u_n)$ is a Cauchy sequence, so it has a limit $u$.
\end{proof}

\begin{proposition}
\label{Pedersen2} Let $A$ be a {\rm C}$^*$-algebra, and let $x$,
$a\in A$ with $a\geq 0$. If $x^*x\leq a$ and $0<\alpha<1/2$, there
exists $u$ in $A$ with $\Vert u\Vert\leq \Vert
a^{\frac{1}{2}-\alpha}\Vert$, and
\[
x=ua^{\alpha}\,.
\]
\end{proposition}
\begin{proof} (Outline.)
Put $y=a^{\frac{1}{2}-\alpha}$. Then $yy^*=a^{1-2\alpha}$. We may
now apply Lemma~\ref{Pedersen} to conclude that there is $u\in A$
with
\[
u_n:=x\left(\frac{1}{n}+a\right)^{-\frac{1}{2}}a^{\frac{1}{2}-\alpha}\stackrel{n
\to \infty}{\longrightarrow} u\,,
\]
and $\Vert u\Vert\leq \Vert a^{\frac{1}{2}-\alpha}\Vert$.

Next, put
$b=1-\left(\frac{1}{n}+a\right)^{-\frac{1}{2}}a^{\frac{1}{2}}$ (in
$\widetilde{A}$). One checks that $\Vert
x-u_na^{\alpha}\Vert^2=\Vert xb\Vert^2\leq \Vert b^*ab\Vert =\Vert
a^{\frac{1}{2}}\Vert^2=\Vert
a^{\frac{1}{2}}-a\left(\frac{1}{n}+a\right)^{-\frac{1}{2}}\Vert^2\stackrel{n
\to \infty}{\longrightarrow}0$, using again Dini's Theorem.
\end{proof}

Notice that, if we put $\vert x\vert =(x^*x)^{\frac{1}{2}}$, the
previous proposition tells us that, given $0<\beta<1$, there is $u$
with $x=u\vert x\vert^{\beta}$, but of course we cannot always take
$\beta=1$. That requires, in general, the passage to the double
dual.

\begin{theorem}
\label{Kirchberg-Rordam}
{\rm (\cite[Lemma 2.2]{KR2})}
Let $A$ be a {\rm C}$^*$-algebra, and $a$, $b\in A_+$. Let
$\epsilon>0$, and suppose that
\[
\Vert a-b\Vert<\epsilon\,.
\]
Then there is a contraction $d$ in $A$ with $(a-\epsilon)_+=dbd^*$.
\end{theorem}
\begin{proof}
Let $r>1$. Define $g_r\colon\mathbb{R}_+\to\mathbb{R}_+$ as
$g_r(t)=\min(t,t^r)$, so that $g_r(b)\to b$ as $r\to 1$.

Since $\Vert a-g_r(b)\Vert\leq \Vert a-b\Vert+\Vert b-g_r(b)\Vert$,
we may choose $r>1$ such that $\Vert
a-g_r(b)\Vert=\epsilon_1<\epsilon$ and also $\Vert
b-g_r(b)\Vert<\epsilon$.

Let $b_0=g_r(b)$. Then, since $a-b_0\leq \epsilon_1\cdot 1$, we have
$a-\epsilon_1\leq b_0$. Note that also $b_0\leq b^r$.

Define $e\in C^*(a)$ as $e(a)$, where $e$ is the continuous function
defined on $\sigma(a)$ as
$e(t)=\left(\frac{t-\epsilon}{t-\epsilon_1}\right)^{\frac{1}{2}}$ if
$t\geq\epsilon$, and $e(t)=0$ otherwise. Note that
\[
e(a-\epsilon_1)e=(a-\epsilon)_+\,,\text{ and }\Vert e\Vert\leq 1.
\]
From this it follows that $(a-\epsilon)_+=e(a-\epsilon_1)e\leq
eb_0e$.

Next, let $x=b_0^{\frac{1}{2}}e$, and let $x=v(x^*x)^{\frac{1}{2}}$
be its polar decomposition (with $v\in A^{**}$, so
$v^*x=(x^*x)^{\frac{1}{2}}$ and
$(v^*v)(x^*x)^{\frac{1}{2}}=(x^*x)^{\frac{1}{2}}$). Then
$x^*x=eb_0e$, and we claim that, since $(a-\epsilon)_+\leq x^*x$, we
have
\[
y:=v(a-\epsilon)_+^{\frac{1}{2}}\in A\,.
\]
Indeed, using Proposition~\ref{Pedersen2}, we may write
$(a-\epsilon)_+^{\frac{1}{2}}=(x^*x)^{\frac{1}{4}}u$ for some
element $u\in A$. Since we may also write
$(x^*x)^{\frac{1}{4}}=\lim\limits (x^*x)^{\frac{1}{2}}t_n$, for a
certain sequence $(t_n)$ (with $\Vert t_n\Vert\leq 1$), we have
\[
y=v(a-\epsilon)_+^{\frac{1}{2}}=v(x^*x)^{\frac{1}{4}}=\lim\limits
v(x^*x)^{\frac{1}{2}}t_n\in A\,.
\]
Now note that
\[
y^*y=(a-\epsilon)_+^{\frac{1}{2}}v^*v(a-\epsilon)_+^{\frac{1}{2}}=(a-\epsilon)_+^{\frac{1}{2}}\,,\text{
and}
\]
\[
yy^*=v(a-\epsilon)_+v^*\leq
vx^*xv^*=v(x^*x)^{\frac{1}{2}}(x^*x)^{\frac{1}{2}}v^*=xx^*=b_0^{\frac{1}{2}}e^2b_0^{\frac{1}{2}}\leq
b_0^{\frac{1}{2}}\,,
\]
as $\Vert e\Vert \leq 1$.

Put
$d_n:=y^*\left(\frac{1}{n}+b^r\right)^{-\frac{1}{2}}b^{\frac{r-1}{2}}$.
Since $yy^*\leq b_0\leq b^r$, we may apply Lemma~\ref{Pedersen}
(with $\alpha=1$ and $\beta=\frac{r-1}{r}$) to obtain that $(d_n)$
is a Cauchy sequence. Let $d$ be its limit. As in the proof of
Proposition~\ref{Pedersen2}, we have $db^{\frac{1}{2}}=y^*$.

Therefore $dbd^*=y^*y=(a-\epsilon)_+$.

Finally
\[
d_n^*d_n\leq
b^{\frac{r-1}{2}}\left(\frac{1}{n}+b^r\right)^{-\frac{1}{2}}yy^*
\left(\frac{1}{n}+b^r\right)^{-\frac{1}{2}}b^{\frac{r-1}{2}}\leq
1\,,
\]
so $\Vert d_n\Vert\leq 1$, which implies that $d$ is a contraction.
\end{proof}

\begin{lemma}
\label{polardec} Let $a=v\vert a\vert$ be the polar decomposition of
an element $a$ in a {\rm C}$^*$-algebra $A$. If $f\in C(\sigma(\vert
a\vert))_+$ and $f$ vanishes on a neighbourhood of zero, then we
have $vf(\vert a\vert)v^*=f(\vert a^*\vert)$.
\end{lemma}
\begin{proof}
This follows from the functional calculus, noticing that
$v(a^*a)v^*=v\vert a\vert\vert a\vert v^*=aa^*$, whence
$v(a^*a)^nv^*=(aa^*)^n$ for any $n$. Therefore $v\vert a\vert
v^*=\vert a^*\vert$.
\end{proof}

Some of our results have the assumption that $A$ has moreover stable
rank one. The notion of stable rank was introduced, in the
topological setting, by Rieffel in \cite{Ri}, as a non-commutative
dimension theory, as in the commutative case is related to the
dimension of the underlying topological space. It was later
connected to the classical stable rank for general rings introduced
by Bass (see \cite{hv}). We shall be exclusively concerned here with
algebras that have \emph{stable rank one}, which by definition means
that the set of invertible elements is dense in $A$. We write
$\mathrm{sr}(A)=1$, as is customary, to mean that the stable rank of
$A$ is one. If $A$ is non-unital, then $A$ has stable rank one if,
by definition, its minimal unitization $\widetilde{A}$ has stable
rank one.

It is a theorem that the stable rank one condition implies that
projections cancel from direct sums. A weaker form of cancellation
is that of stable finiteness. We say that a C$^*$-algebra $A$ is
\emph{directly finite} if isometries are unitaries. In other words,
\[
xx^*=1 \implies x^*x=1\,.
\]
A C$^*$-algebra $A$ such that $M_n(A)$ is directly finite for all
$n\geq 1$ will be called \emph{stably finite}.  Since stable rank
one implies cancellation, it also implies stable finiteness.

We shall also need the result quoted below. Its proof will be
omitted, but can be found in~\cite[Corollary 8]{pedunitary}.

\begin{theorem}
\label{Pedunitary} Let $A$ be a {\rm C}$^*$-algebra with stable rank
one. If $a$ has polar decomposition $a=v\vert a\vert$, then given
$f\in C(\sigma(\vert a\vert))_+$ that vanishes on a neighbourhood of
zero, there is a unitary $u\in U(\widetilde{A})$ such that $vf(\vert
a\vert)=uf(\vert a\vert)$, and therefore $uf(\vert
a\vert)u^*=vf(\vert a\vert)v^*=f(\vert a^*\vert)$.
\end{theorem}

\begin{lemma}
\label{Blackaddar} Let $A$ be a {\rm C}$^*$-algebra. Given $b\in
A_+$, the set $\{a\in A_+\mid a\precsim b\}$ is norm-closed.
\end{lemma}
\begin{proof}
Suppose $a=\lim a_n$ with $a_n\precsim b$ for all $n$. Given
$n\in\mathbb{N}$, choose $a_m\in A$ with $m$ depending on $n$ such
that $\Vert a-a_m\Vert<\frac{1}{2n}$ and $x_n\in A$ such that $\Vert
a_m-x_nbx_n^*\Vert<\frac{1}{2n}$. Therefore
\[
\Vert a-x_nbx_n^*\Vert\leq\Vert a-a_m\Vert+\Vert
a_m-x_nbx_n^*\Vert<\frac{1}{n}\,.
\]
\end{proof}

In order to ease the notation, we will use in the sequel $A_a$ to
denote the hereditary C$^*$-algebra generated by a positive element
$a$ in $A$, that is, $A_a=\overline{aAa}$. Recall that, if $A$ is a
separable C$^*$-algebra, then all hereditary algebras are of this
form.

\begin{proposition}{\rm (\cite{KR},~\cite{Rfunct})}\label{basics}
Let $A$ be a C$^*$-algebra, and $a,b \in A_+$. The following
conditions are equivalent:
\begin{enumerate}[{\rm (i)}]
\item $a \precsim b$; \item for all $\epsilon > 0$,
$(a-\epsilon)_+ \precsim b$; \item for all $\epsilon > 0$, there
exists $\delta > 0$ such that $(a-\epsilon)_+ \precsim
(b-\delta)_+$.
\end{enumerate}
If moreover $\mathrm{sr}(A)=1$, then
\[
a\precsim b \text{ if and only if for every }\epsilon>0\,,\text{
there is }u\text{ in }U(A)\text{ such that }u^*(a-\epsilon)_+u\in
A_b\,.
\]
\end{proposition}
\begin{proof}
(i) $\implies$ (ii). There is by assumption a sequence $(x_n)$ such
that $a=\lim\limits_n x_n^*bx_n$. Given $\epsilon>0$, we find $n$
such that $\Vert a-x_n^*bx_n\Vert<\epsilon$. Thus
Theorem~\ref{Kirchberg-Rordam} implies that
$(a-\epsilon)_+=dx_n^*bx_nd^*$, for some $d$. Therefore
$(a-\epsilon)_+\precsim b$.

(ii) $\implies$ (iii). Given $\epsilon>0$, there is by (ii) an
element $x$ such that $\Vert
(a-\frac{\epsilon}{2})_+-xbx^*\Vert=\epsilon_1<\frac{\epsilon}{2}$.

Since $(b-\delta)_+$ is monotone increasing and converges to $b$ (in
norm) as $\delta\to 0$, we may choose
$\delta<\frac{\frac{\epsilon}{2}-\epsilon_1}{\Vert x\Vert^2}$.
Therefore
\[
\Vert (a-\frac{\epsilon}{2})_+-x(b-\delta)_+x^*\Vert\leq\Vert
(a-\frac{\epsilon}{2})_+-xbx^*\Vert+\Vert
xbx^*-x(b-\delta)_+x^*\Vert\leq \epsilon_1+\Vert x\Vert^2\delta
<\frac{\epsilon}{2}\,,
\]
so by the Theorem above $(a-\epsilon)_+=y(b-\delta)_+y^*\precsim
(b-\delta)_+$.

(iii) $\implies$ (i). By assumption we have that
$(a-\epsilon)_+\precsim b$ for all $\epsilon>0$, so
Lemma~\ref{Blackaddar} implies $a\precsim b$.

The ``if'' direction in the last part of our statement holds without
any stable rank conditions. Namely, assume that $\epsilon >0$ is
given, and that we can find a unitary $u$ such that
$u^*(a-\epsilon)_+u\in \overline{bAb}$. This implies that
$u^*(a-\epsilon)_+u\precsim b$, and so $(a-\epsilon)_+\precsim b$,
and condition (ii) is verified.

For the converse, assume that $A$ has stable rank one and that
$a\precsim b$. Given $\epsilon>0$, find an element $x$ such that
$(a-\frac{\epsilon}{2})_+=xbx^*=zz^*$, where $z=xb^{\frac{1}{2}}$.
If the polar decomposition of $z^*$ is $z^*=v\vert z^*\vert$, we
have by Theorem~\ref{Pedunitary} that there is a unitary $u$ with
\[
u\left(zz^*-\frac{\epsilon}{2}\right)_+u^*=v\left(zz^*-\frac{\epsilon}{2}\right)_+v^*=\left(z^*z-\frac{\epsilon}{2}\right)_+\,.
\]
Then,
\[
u(a-\epsilon)_+u^*=
u\left(\left(a-\frac{\epsilon}{2}\right)_+-\frac{\epsilon}{2}\right)_+u^*=u\left(zz^*-\frac{\epsilon}{2}\right)_+u^*=\left(z^*z-\frac{\epsilon}{2}\right)_+\,,
\]
and clearly the latter belongs to $A_b$ as
$z^*z=b^{\frac{1}{2}}xx^*b^{\frac{1}{2}}$ does.
\end{proof}

Note that, if $A_a\subseteq A_b$ for positive elements $a$ and $b$
in $A$, we have that $a\precsim b$ (by Proposition~\ref{basics}).

\subsubsection{Comparison and Projections}
\label{projs}

When one has projections (i.e. self-adjoint idempotents) rather than
general positive elements, one recovers the usual comparison by
Murray and von Neumann. Recall that two projections $p$ and $q$ in a
C$^*$-algebra $A$ are equivalent if $p=vv^*$ and $q=v^*v$ for some
element $v\in A$ (necessarily a partial isometry). Classically, the
notation $p\sim q$ has been used to refer to this equivalence, and
this is not to be mistaken with the above $\sim$ defined for general
positive elements. Both notions will, however, agree for a
significant class of algebras as we shall see below.

%

Given projections $p$, $q$, we have $p\leq q$ (as positive elements)
if and only if $p=pq$. Recall that $p$ is subequivalent to $q$ if
$p$ is equivalent to $p'\leq q$, that is, $p=vv^*$, with $v^*v\leq
q$.
\begin{lemma}
For projections $p$ and $q$, we have that $p$ is subequivalent to
$q$ if and only if $p\precsim q$.
\end{lemma}
\begin{proof}
It is clear that, if $p$ is subequivalent to $q$, then $p\precsim
q$. For the converse, if $p\precsim q$, then given $0<\epsilon<1$,
we have $(p-\epsilon)_+=xqx^*$, and $(p-\epsilon)_+=\lambda p$ for
some positive $\lambda$.

Therefore, changing notation we have $p=xqx^*$, which implies that
$qx^*xq\leq q$ is a projection equivalent to $p$.
\end{proof}
The argument in the lemma above is more general. Given $\epsilon>0$,
denote by $f_{\epsilon}(t)$ the real valued function defined as $0$
for $t\leq \epsilon/2$, as $1$ for $t\geq \epsilon$, and linear
elsewhere.

\begin{lemma}
\label{tonteria} If $p$ is a projection and $a$ a positive element,
and $p\precsim a$, then there is $\delta>0$ and a projection $q\leq
\lambda a$ ($\lambda$ a positive real) with $p\sim q$ and
$f_{\delta}(a)q=q$.
\end{lemma}
\begin{proof}
Pick $\epsilon>1$, so $(p-\epsilon)_+=\lambda'p$ for a positive
number $\lambda'$. Then there is by Proposition~\ref{basics} a
$\delta'>0$ and an element $x\in pA$ with
$\lambda'p=x(a-\delta')_+x^*$. Changing notation, we have
$p=x(a-\delta')_+x^*$. Thus
$q:=(a-\delta')_+^{\frac{1}{2}}x^*x(a-\delta')_+^{\frac{1}{2}} $ is
a projection equivalent to $p$ and $q\leq \Vert x\Vert^2a$. On the
other hand, it is clear by the definition of $q$ that we can choose
$\delta<\delta'$ (e.g. $\delta=\frac{\delta'}{2}$) such that
$f_{\delta}(a)q=q$.
\end{proof}

Recall that $\V(A)$ is used to denote the semigroup of Murray-von
Neumann equivalence classes of projections coming from
$M_{\infty}(A)$. If we use $[p]$ to denote the class of a projection
$p$,  we then have a natural map $\V(A)\to \W(A)$, given by
$[p]\mapsto\langle p\rangle$. This is easily seen to be a semigroup
morphism, and it will be injective in interesting cases. Note that
injectivity amounts exactly to the fact that Murray-von Neumann
equivalence and Cuntz equivalence agree on projections.

\begin{lemma}
If $A$ is stably finite (and, in particular, if it has stable rank
one), then the natural map $\V(A)\to \W(A)$ is injective.
\end{lemma}
\begin{proof}
It is enough to show this in the stably finite case. Suppose that we
are given projections $p$ and $q$ in $M_{\infty}(A)$ such that
$p\sim q$ (in $\W(A)$). Then $p\precsim q$ and $q\precsim p$. This
means there are projections $p'$ and $q'$ such that $p\oplus p'\sim
q$ and $q\oplus q'\sim p$ (in $\V(A)$). Thus
\[
p\oplus p'\oplus q'\sim q\oplus q'\sim p\,,
\]
and it follows from stable finiteness that $p'\oplus q'=0$, i.e.
$p'=q'=0$, so that $p$ and $q$ are Murray von Neumann equivalent.
\end{proof}

In the sequel, we will identify $\V(A)$ with its image inside
$\W(A)$ whenever $A$ is stably finite without further comment.

\begin{lemma}
\label{krordam2.8} Given a positive element $a$ and a projection
$p\in\mathcal{M}(A)$, we have
\[
a\precsim pap+(1-p)a(1-p)\,.
\]
\end{lemma}
\begin{proof}
Let $s=p-(1-p)$. Since
\[
a\leq a+sas=2(pap+(1-p)a(1-p))\sim pap+(1-p)a(1-p)\,,
\]
the claim follows from Lemma \ref{lem:leq}.
\end{proof}

We have already noticed that the order in $\W(A)$ is not algebraic.
If we restrict to projections, since this is the Murray von Neumann
subequivalence, it \emph{is} algebraic. In fact, projections behave
well with respect to every other element as the following shows.

\begin{proposition}
{\rm (\cite[Proposition 2.2]{pt})}
\label{projcomplement} Let $A$ be a {\rm C}$^*$-algebra, and let
$a$, $p$ be positive elements in $M_{\infty}(A)$ with $p$ a
projection. If $p\precsim a$, then there is $b$ in $M_{\infty}(A)$
such that $p\oplus b\sim a$.
\end{proposition}
\begin{proof}
By Lemma~\ref{tonteria}, $p\sim q$ with $q\leq \lambda a$ for a
positive number $\lambda$. Trading $p$ with $q$ and ignoring
$\lambda$ (which is in fact irrelevant to Cuntz comparison), we may
assume that $p\leq a$.

Using Lemma~\ref{krordam2.8}, we have $a\precsim pap\oplus
(1-p)a(1-p)$. As we have $pap\leq \Vert a\Vert^2p\sim p$, we obtain
$a\precsim p\oplus (1-p)a(1-p)$. For the converse subequivalence,
note that both $p$ and $(1-p)a(1-p)$ belong to the hereditary
algebra generated by $a$. Indeed $p$ falls in as $p\leq a$, and
$(1-p)a^{\frac{1}{2}}=a^{\frac{1}{2}}-pa^{\frac{1}{2}}$ also falls
in, whence so does $(1-p)a(1-p)$.
\end{proof}

The following will be used a number of times:

\begin{proposition}
\label{prop:equivproj}
{\rm (\cite[Proposition 3.12]{perijm})}
Let $A$ be a unital {\em C}$^*$-algebra with stable rank one. Then,
for $a\in M_{\infty} (A)$, the following are equivalent:
\begin{enumerate}[{\rm (i)}]
\item $\langle a\rangle=\langle p\rangle$, for a projection $p$,
\item $0$ is an isolated point of $\sigma (a)$, or
$0\notin\sigma(a)$.
\end{enumerate}
\end{proposition}
\begin{proof}
(ii) $\implies $ (i) is clear.

(i) $\implies$ (ii). Suppose $a\sim p$, and that $0$ is a
non-isolated point of $\sigma (a)$. Using Lemma~\ref{tonteria}, find
a projection $q\sim p$ and $\delta>0$ with $f_{\delta}(a)q=q$.

Since $0$ is not isolated in $\sigma(a)$, we know $f_{\delta}(a)$ is
not a projection, so in particular $f_{\delta}(a)\neq q$. This tells
that
\[
q=f_{\delta}(a)^{\frac{1}{2}}qf_{\delta}(a)^{\frac{1}{2}}<
f_{\delta}(a)\,.
\]
Choose $0<\delta'<\frac{\delta}{2}$, so that $f_{\delta}(a)\leq
(a-\delta')_+$. Next, use that $a\precsim q$, so there is by
Proposition~\ref{basics}, a unitary $u$ with $u(a-\delta')_+u^*\in
qAq$, and so $uf_{\delta}(a)u^*\in qAq$. In particular, we have
$uf_{\delta}(a)u^*\leq uu^*\leq 1$, so
$uf_{\delta}(a)u^*=quf_{\delta}(a)u^*q\leq q$. But now
\[
uqu^*+u(f_{\delta}(a)-q)u^*\leq q\,,
\]
and $u(f_{\delta}(a)-q)u^*>0$, whence $uqu^*<q$. But this
contradicts the fact that $A$ has stable rank one, and in particular
is stably finite.
\end{proof}

The previous result motivates the following. Let $\W(A)_+$ denote
the subset of $\W(A)$ consisting of those classes which are
\emph{not} the classes of projections. If $a\in A_+$ and $\langle
a\rangle\in \W(A)_+$, then we will say that $a$ is \emph{purely
positive} and denote the set of such elements by $A_{++}$.

\begin{corollary}\label{varecov}
Let $A$ be a unital {\rm C}$^*$-algebra which is either simple or of
stable rank one. Then
\begin{enumerate}[{\rm (i)}]
\item $\W(A)_+$ is a semigroup, and is absorbing in the sense that
if one has $a \in W(A)$ and $b \in W(A)_+$, then $a + b \in W(A)_+$;
\item
\[
\V(A)=\{x\in \W(A)\mid \text{ if }x\leq y \text{ for }y\in
\W(A)\,,\,\text{then }x+z=y\text{ for some }z\in \W(A)\}
\]
\end{enumerate}
\end{corollary}

\begin{proof}
(i). Take $\langle a \rangle$, $\langle b \rangle \in \W(A)_+$ and
notice that the spectrum of $a \oplus b$ contains the union of the
spectra of $a$ and $b$, and then apply
Proposition~\ref{prop:equivproj}.

(ii). Set $X=\{x\in \W(A)\mid \text{ if }x\leq y \text{ for }y\in
\W(A)\,,\,\text{then }x+z=y\text{ for some }z\in \W(A)\}$. By
Proposition~\ref{projcomplement}, we already know that
$\V(A)\subseteq X$.

Conversely, if $\langle a\rangle\in X$, then we may find a
projection $p$ (in $M_{\infty}(A)$) such that $\langle
a\rangle\leq\langle p\rangle$. But then there is $b$ in
$M_{\infty}(A)$ for which $a\oplus b\sim p$. Since either
$0\notin\sigma (p)$ or $0$ is an isolated point in $\sigma(p)$, the
same will be true of $\sigma (a)$. Invoking
Proposition~\ref{prop:equivproj}, we find a projection $q$ such that
$q\sim a$, and so $\langle a\rangle\in \V(A)$.
\end{proof}

We close by proving a recent result due to R\o rdam and Winter,
which says that, if we have stable rank one, then we can cancel
projections
\begin{proposition}
\label{prop:rorwin} Let $A$ be a {\rm C}$^*$-algebra of stable rank
one and let $a$, $b$ be positive elements and $p$ a projection (all
in $M_{\infty}(A)$). If
\[
a\oplus p\precsim b\oplus p
\]
then $a\precsim b$.
\end{proposition}
\begin{proof}
We may assume that all elements belong to $A$ and that $a$ and $b$
are both orthogonal to $p$. Given $0<\epsilon<1$, there is by
Proposition~\ref{basics} (and taking into account that $p$ is a
projection) a unitary $u$ such that
\[
u((a-\epsilon)_++p)u^*\in A_{(b+p)}\,.
\]
Since also $A_{b+p}$ has stable rank one (being a hereditary
subalgebra of $A$) and $p$, $upu^*\in A_{b+p}$, we can find a
unitary $v$ (in its unitization) with $upu^*=vpv^*$.

Now $v^*u(a-\epsilon)_+u^*v\in A_{p+b}$ and is orthogonal to $p$,
whence $v^*u(a-\epsilon)_+u^*v\in (1-p)A_{p+b}(1-p)=A_b$.

This shows that $(a-\epsilon)_+\precsim b$ for any $\epsilon<1$, so
$a\precsim b$.
\end{proof}

\begin{theorem}
{\rm (\cite[Theorem 4.3]{rorwin})}
Let $A$ be a {\rm C}$^*$-algebra with stable rank one, and let $x$,
$y\in \W(A)$ be such that
\[
x+\langle c\rangle\leq y+\langle (c-\epsilon)_+\rangle
\]
for some $c\in M_{\infty}(A)_+$ and some $\epsilon>0$. Then $x\leq
y$.
\end{theorem}
\begin{proof}
As above we may assume that $x$, $y$ are represented by elements
$a$, $b$ in $A$ and that $c$ also belongs to $A$. We may also assume
$A$ unital.

Given $0<\epsilon$, define a function $h_{\epsilon}$ to be zero for
$t\geq \epsilon$, one at zero, and linear in $[0,\epsilon]$. By
construction $h_{\epsilon}(c)$ is orthogonal to $(c-\epsilon)_+$,
and $c+h_{\epsilon}(c)\geq 1$, hence it is invertible. Further, by
Lemma~\ref{Cuntz}, $c+h_{\epsilon}(c)\precsim c\oplus
h_{\epsilon}(c)$, and $(c-\epsilon)_+\oplus h_{\epsilon}(c)\sim
(c-\epsilon)_++ h_{\epsilon}(c)\leq 1$. Thus
\[
a\oplus 1\precsim a\oplus c\oplus h_{\epsilon}(c)\precsim b\oplus
(c-\epsilon)_+\oplus h_{\epsilon}(c)\sim b\oplus((c-\epsilon)_++
h_{\epsilon}(c)) \precsim b\oplus 1\,,
\]
and so $a\precsim b$ using Proposition~\ref{prop:rorwin}.
\end{proof}

As remarked in~\cite{rorwin}, it is not true that the Cuntz
semigroup is cancellative, even in the stable rank one case. We will
see later that $\W(A)$ is never cancellative for a wide class of
algebras of stable rank one.

\subsection{The group $\mathbf{\mathrm{K}_0^*}$}

The Grothendieck enveloping group of $\W(A)$ is denoted
$\mathrm{K}_0^*(A)$. Its structure has been previously analysed in
\cite{bh}, \cite{Cu}, \cite{han}, and~\cite{perijm}. As $\W(A)$
carries its own order coming from the Cuntz comparison relation, we
may in principle equip $\mathrm{K}_0^*(A)$ with two natural
(partial) orderings.

If $M$ is an Abelian semigroup with a partial order $\leq$ that
extends the algebraic order (as is the case with $\W(A)$), we use
$G(M)$ to denote its Grothendieck enveloping group. For convenience,
we recall its construction and the two orderings it may be given.

\begin{definition}
Define a congruence on $M$ by declaring $x\sim y$ if there is $z$ in
$M$ with $x+z=y+z$. Set $M_c=M/\sim$, and denote $[x]$ the
congruence class of $x$. The set $M_c$ becomes an abelian semigroup
under addition $[x]+[y]=[x+y]$. Adjoining formal inverses to the
elements of $M_c$ we obtain a group $G(M)$.
\end{definition}

There is a natural semigroup homomorphism $\gamma\colon M\to G(M)$,
referred as to the Grothendieck map. With this notation,
\[
G(M)=\{\gamma(a)-\gamma (b)\mid a, b\in M\}\,.
\]
We define the following cones:
\[
G(M)^+=\gamma (M)\,,
\]
and
\[
G(M)^{++}=\{\gamma(a)-\gamma(b)\mid a,\,b\in M\text{ and }b\leq
a\}\,.
\]
We shall also use $[a]=\gamma (a)$, for $a$ in $M$. Note that then
$[a]-[b]\leq [c]-[d]$ in $(G(M), G(M)^{++})$ if and only if
$a+d+e\leq b+c+e$ for some $e$ in $M$.

\begin{lemma}
If $M$ is partially ordered, so is $(G(M), G(M)^{++})$. Since
$G(M)^+\subseteq G(M)^{++}$ (and the inclusion may be strict), we
conclude that also $(G(M),G(M)^+)$ is partially ordered.
\end{lemma}
\begin{proof}
Assume that $\gamma(a)-\gamma(b)\in G(M)^{++}\cap(-G(M)^{++})$. Then
there are elements $s$, $t$, $u$, $v$ in $M$ such that
\[
a+z\leq b+z\,,\,\, t+v\leq s+v\,,\,\, a+s+u=b+t+u\,,
\]
so that $\gamma(b)-\gamma(a)=\gamma(s)-\gamma(t)\in G(M)^{++}$. Set
$w=u+v+z+t$ and check that $a+w=b+w$, whence $\gamma(a)=\gamma(b)$.
\end{proof}

\begin{proposition}
\label{k*struc} Let $A$ be a {\rm C}$^*$-algebra with stable rank
one and such that the semigroup $\W(A)_+$ of purely positive
elements is non-empty. Then there exists an ordered group
isomorphism
\[
\alpha\colon (\mathrm{K}_0^*(A), \mathrm{K}_0^*(A)^{++})\to
(G(\W(A)_+),G(\W(A)_+)^+)\,.
\]
\end{proposition}

\begin{proof}
By Corollary~\ref{varecov} we have that, since $A$ has stable rank
one, $\W(A)=V(A)\sqcup \W(A)_+$. Denote by $\gamma\colon \W(A)_+\to
G(\W(A)_+)$ the Grothendieck map, and choose any element $c\in
\W(A)_+$. Then, define
\[
\alpha\colon \W(A)\to G(\W(A)_+)
\]
by $\alpha(\langle a\rangle)=\gamma(\langle a\rangle)$ if $\langle
a\rangle\in \W(A)_+$, and by $\alpha(\langle
p\rangle)=\gamma(\langle p\rangle+c)-\gamma(c)$ for any projection
in $M_{\infty}(A)$.

Note that $\alpha$ is a well defined semigroup homomorphism. Indeed,
since $A$ has stable rank one, $\langle p\rangle +c\in \W(A)_+$
whenever $c\in \W(A)_+$, and if $c'\in \W(A)_+$ is any other
element, then one has that $\gamma(\langle
p\rangle+c)-\gamma(c)=\gamma(\langle p\rangle+c')-\gamma(c')$.

In order to check that $\alpha$ is a homomorphism, let $p$, $q$ and
$a$ be elements in $M_{\infty}(A)_+$ with $p$ and $q$ projections
and $a$ purely positive. Then,
\begin{eqnarray*}
\alpha(\langle p\rangle+\langle q\rangle) & = & \gamma(\langle
p\oplus
q\rangle+2c)-\gamma(2c) \\
& = & \gamma(\langle p\rangle+c)-\gamma(c)+\gamma(\langle
q\rangle+c)-\gamma(c) \\
& = & \alpha(\langle p\rangle)+\alpha(\langle q\rangle).
\end{eqnarray*}

\noindent Also
\begin{eqnarray*}
\alpha(\langle p\rangle+\langle a\rangle)& = & \gamma(\langle p\oplus a\rangle) \\
& = & \gamma(\langle p\oplus a\rangle+c)-\gamma(c) \\
& = & \gamma(\langle p\rangle+c)-\gamma(c)+\gamma(\langle a\rangle) \\
& = & \alpha(\langle p\rangle)+\alpha(\langle a\rangle).
\end{eqnarray*}

It is easy to check that $\alpha (\W(A))\subseteq G(\W(A)_+)^+$, and
so $\alpha$ extends to an ordered group homomorphism
\[
\alpha\colon \mathrm{K}_0^*(A)=G(\W(A))\to G(\W(A)_+)\,,
\]
given by the rule $\alpha([a]-[b])=\alpha(\langle
a\rangle)-\alpha(\langle b\rangle)$. Evidently, $\alpha$ is
surjective and satisfies
\[\
\alpha (\mathrm{K}_0^*(A)^{++})\subseteq G(\W(A)_+)^+\,
\]

To prove injectivity, assume that $\alpha (\langle
a\rangle)=\alpha(\langle p\rangle)$ for $\langle a\rangle\in
\W(A)_+$ and $p$ a projection. This means that $\gamma(\langle
a\rangle)=\gamma(\langle p\rangle+c)-\gamma(c)$, and hence $\langle
a\rangle+c+c'=\langle p\rangle+c+c'$ for some $c'\in \W(A)$. Thus
$[a]=[b]$ in $\mathrm{K}_0^*(A)$. If, for projections $p$ and $q$,
we have that $\alpha(\langle p\rangle)=\alpha(\langle q\rangle)$,
then $\gamma(\langle p\rangle+c)-\gamma(c)=\gamma(\langle
q\rangle+c)-\gamma(c)$, from which $[p]=[q]$ in $\mathrm{K}_0^*(A)$.
%
\end{proof}

\subsection{The Cuntz semigroup and dimension functions}

\begin{definition} {\rm Let $(M,leq)$ be a preordered Abelian semigroup. Recall that a non-zero element $u$ in
$M$ is said to be an \emph{order-unit} provided that for any $x$ in
$M$ there is a natural number $n$ such that $x\leq nu$. A
\emph{state} on a preordered monoid $M$ with order-unit $u$ is an
order preserving morphism $s\colon M\to \mathbb{R}$ such that
$s(u)=1$. We denote the (convex) set of states by
$\mathrm{S}(M,u)$.}
\end{definition}

In the case of a unital C$^*$-algebra $A$, we denote
\[
\mathrm{DF}(A)=\mathrm{S}(\W(A),\langle
1\rangle)=\mathrm{S}(\mathrm{K}_0^*(A),[1])\,.
\]
The states on $\W(A)$ are called \emph{dimension functions} on $A$
(see also~\cite{bh},~\cite{Rfunct},~\cite{perijm}).

A dimension function $s$ is \emph{lower semicontinuous} if
$s(\langle a\rangle)\leq \liminf\limits_{n\to\infty} s(\langle
a_n\rangle)$ whenever $a_n\to a$ in norm. Note that any dimension
function $s$ induces a function $d_s\colon M_{\infty}(A)\to
\mathbb{R}$ given by $d_s(a)=s\langle a^*a\rangle$. With this
notation, lower semicontinuity of $A$ as defined above is equivalent
to lower semicontinuity of the function $d_s$, regarding
$M_{\infty}(A)$ as a local C$^*$-algebra.

The set of all lower semicontinous dimension functions on $A$ is
denoted by $\mathrm{LDF}(A)$. It is pertinent at this stage to
mention that Blackadar and Handelman conjectured in \cite{bh} that
$\mathrm{DF}(A)$ is a Choquet simplex and that $\mathrm{LDF}(A)$ is
always dense in $\mathrm{DF}(A)$. These conjectures are known to
have positive answers for those algebras that have good
representations of their Cuntz semigroups (see \cite{bpt}).

\begin{definition}
A quasitrace on a {\rm C}$^*$-algebra $A$ is a function $\tau\colon
A\to\mathbb{C}$ such that
\begin{enumerate}[{\rm (i)}]
\item $\tau(x^*x)=\tau(xx^*)\geq 0$ for any $x\in A$. \item $\tau$
is linear on commutative $^*$-subalgebras of $A$. \item If $x=a+ib$,
where $a$, $b$ are self-adjoint, then $\tau(x)=\tau(a)+i\tau(b)$.
\item $\tau$ extends to a map on $M_n(A)$ with the same properties.
\end{enumerate}
\end{definition}
Note that a trace is just a linear quasitrace. We will say that a
trace or a quasitrace $\tau$ is \emph{normalized} whenever its norm,
$\Vert \tau\Vert=\sup\{\tau(a)\mid 0\leq a, \Vert a\Vert\leq 1\}$,
equals one (it will always be finite, see below). In the case that
$A$ is unital, then this amounts to the requirement that
$\tau(1)=1$.

We shall denote as usual by $\mathrm{T}(A)$ the simplex of
normalised traces defined on a unital C$^*$-algebra $A$, and by
$\mathrm{QT}(A)$ the simplex of quasitraces (it is not obvious that
both sets are simplices, but it is a true theorem, see \cite{bh} and
the references therein). We have $\mathrm{T}(A)\subseteq
\mathrm{QT}(A)$, and equality holds if $A$ is exact and unital by
the main theorem of~\cite{Ha}.

Observe also that, if $A$ is simple, then quasitraces are faithful,
meaning that $\tau(x^*x)=0$ if and only if $x=0$. This follows from
the fact that, for a quasitrace $\tau$, the set $\{x\in
A\mid\tau(x^*x)=0\}$ is a closed, two-sided ideal of $A$.

\begin{theorem}
\label{blackhand}  {\rm (\cite{bh})} There is an affine bijection $\mathrm{QT}(A)\to
\mathrm{LDF}(A)$ whose inverse is continuous.
\end{theorem}
\begin{proof} (Sketch.)
Given a lower semicontinuous dimension function $d$, define $\tau_d$
as follows.

If $B$ is an abelian subalgebra of $A$, then $B\cong C(X)$, for some
space $X$, whence $d$ induces a probability measure $\mu_d$ on the
algebra generated by the $\sigma$-compact open sets
(see~\cite[Proposition II.2.1]{bh}), given by $\mu_d(U)=d(f)$, if
the support of $f$ equals $U$. Then $\tau_d(f)=\int\limits_X
fd\mu_d$ for any $f$ in $B$.

To define $\tau_d$ on a general element $x$, write $x=a+ib$ where
$a$, $b$ are self-adjoint, and then
\[
\tau_d(x):=\tau_d(a)+i\tau_d(b)\,.
\]

The converse map is given by $\tau\mapsto\dt$, where
\[
\dt(a)= \lim_{n \to \infty} \tau(a^{1/n}).
\]
\end{proof}

That the forward map is not continuous in general follows, e.g. from
the following example:

\begin{example}
Let $A=C(\mathbb{N}^*)$, where $\mathbb{N}^*$ is the one point
compactification of the natural numbers. Note that the elements of
$A$ are functions $f$ with $\lim f(n)=f(*)$.

Let $U=\mathbb{N}$, a $\sigma$-compact open set which is dense
Observe that $x_n=n\stackrel{n\to \infty}{\longrightarrow} *$, but
the point mass measure $\delta_{x_n}$ does not converge to
$\delta_x$, because $\delta_{x_n}(U)=1$ for all $n$, while
$\delta_x(U)=0$. On the other hand, we do have that
\[
f(x_n)=\int_X fd(\delta_{x_n})\to \int_X fd(\delta_x)=f(x)\,.
\]
\end{example}

\begin{proposition}
\label{rordam} Let $d\in\mathrm{DF}(A)$. Then
\[
\overline{d}(\langle a\rangle)=\sup\limits_{\epsilon>0} d(\langle
(a-\epsilon)_+\rangle)
\]
defines a lower semicontinuous dimension function, with
$\overline{d}\leq d$ and equality holds if $d\in\mathrm{LDF}(A)$.
\end{proposition}
\begin{proof}
Let us show that $a\precsim b$ implies $\overline{d}(\langle
a\rangle)\leq \overline{d}(\langle b\rangle)$, from which it easily
follows that $\overline{d}\in\mathrm{DF}(A)$. Given $\epsilon>0$,
choose $\delta>0$ with $(a-\epsilon)_+\precsim (b-\delta)_+$, whence
\[
d(\langle(a-\epsilon)_+\rangle)\leq
d(\langle(b-\delta)_+\rangle)\leq\overline{d}(\langle b\rangle)\,.
\]
Since $\epsilon$ is arbitrary, $\overline{d}(\langle a\rangle)\leq
\overline{d}(\langle b\rangle)$.

Next, if $a_n\to a$ in norm and $\epsilon>0$ is given, then for
$n\geq n_0$ we have $\left(a-\frac{\epsilon}{2}\right)_+\precsim
a_n$, and so there is a sequence $\delta_n>0$ such that
$(a-\epsilon)_+\precsim (a_n-\delta_n)_+$. This implies
\[
d(\langle(a-\epsilon)_+\rangle)\leq d(\langle
(a_n-\delta_n)_+\rangle)\leq \overline{d}(\langle a_n\rangle)\,,
\]
for $n\geq n_0$, whence $\overline{d}\in \mathrm{LDF}(A)$.

That $\overline{d}\leq d$ follows because $(a-\epsilon)_+\leq a$ for
all $\epsilon>0$. Finally, if $d\in\mathrm{LDF}(A)$, then since
$(a-\epsilon)_+\to a$ in norm as $\epsilon\to 0$, we have
\[
d(\langle a\rangle)\leq \liminf d(\langle
(a-\epsilon)_+\rangle)=\sup d(\langle
(a-\epsilon)_+\rangle)=\overline{d}(\langle a\rangle)\,.
\]
\end{proof}
\begin{lemma}
\label{cancelstfinite} Let $A$ be a stably finite, unital {\rm
C}$^*$-algebra. Then, $1\oplus a\not\precsim a$.
\end{lemma}
\begin{proof}
For each $n\in\mathbb{N}\cup\{0\}$, and $a\in M_{\infty}(A)_+$
define
\[
s_n(a)=\sup\{s\in\mathbb{N}\cup\{0\}\mid s\langle 1\rangle\leq
n\langle a\rangle\}\,.
\]
Note that if $a\precsim b$, then $s_n(a)\leq s_n(b)$, and also
$s_n(a)+s_n(b)\leq s_n(a\oplus b)$. For the latter assertion, note
that if $s\langle 1\rangle\leq n\langle a\rangle$ and $t\langle
1\rangle\leq n\langle b\rangle$, then
\[
(s+t)\langle 1\rangle\leq n\langle a\rangle+n\langle
b\rangle=n\langle a\oplus b\rangle\,.
\]
Observe that $s_n(a)<\infty$. To see this, suppose that $n\cdot a\in
M_t(A)$ for some $t$. Then $(t+1)\langle 1\rangle\nleq n\langle
a\rangle$. Otherwise, given $0<\epsilon<1$, find $\delta>0$ and
$x\in M_{\infty}(A)$ such that $(t+1)\cdot 1=x(a-\delta)_+x^*$. Then
$p:=(a-\delta)_+^{\frac{1}{2}}x^*x(a-\delta)_+^{\frac{1}{2}}$ is a
projection in $M_t(A)$, equivalent to $(t+1)\cdot 1$. Thus
$(t+1)\cdot 1\precsim t\cdot 1$, which contradicts the stable
finiteness of $A$.

Now, if $1\oplus a\precsim a$, we have
\[
s_n(1)+s_n(a)\leq s_n(1\oplus a)\leq s_n(a)\,,
\]
whence $s_n(1)=0$, which is clearly impossible.
\end{proof}

The following is in \cite{Cu} and \cite{han}.

\begin{theorem}
{\rm (Cuntz, Handelman)} A (unital) stably finite {\rm
C}$^*$-algebra $A$ has a dimension function and, if $A$ is simple,
then the  converse also holds.
\end{theorem}
\begin{proof}
If $A$ is stably finite and unital, in order to check that it has a
dimension function, we only need to see that
$(\mathrm{K}_0^*(A),[1])$ is a non-zero group, since then, being
partially ordered, will have a state (see~\cite[Corollary
4.4]{poag}). Thus it is enough to see that $[1]\neq 0$ in
$\mathrm{K}_0^*(A)$.

If $[1]=0$, then there is $a\in M_{\infty}(A)_+$ such that $1\oplus
a\sim a$, which is impossible by Lemma~\ref{cancelstfinite}.

Now assume that $A$ is simple and has a dimension function. Then, by
Proposition~\ref{rordam}, it has a lower semicontinuous dimension
function, which comes from a quasitrace $\tau$ by
Theorem~\ref{blackhand}. Now, if $x\in M_n(A)$ and $xx^*=1$, then
$\tau(1-x^*x)=n-\tau(x^*x)=0$, so that $x^*x=1$.
\end{proof}




\subsection{Some examples}

The computation of the Cuntz semigroup is in general a very
difficult task. Already in the commutative case it becomes a
monstruous object. In this section we work out examples in which its
structure is determined by the projection semigroup.

\subsubsection{Finite dimensional algebras}

This constitutes the only example where the Cuntz semigroup does not
yield, strictly speaking, any more information than the projection
monoid does. It is based on the following well known result.

\begin{lemma}
Let $A$ be a {\rm C}$^*$-algebra. Then $A$ is infinite dimensional
if and only if there is a purely positive element.
\end{lemma}
\begin{proof}
This follows from the fact that $A$ is infinite dimensional if and
only if there is an element with infinite spectrum.

Now, a purely positive element must have infinite spectrum since
otherwise it would be equivalent to some projection (in matrices
over $A$). Conversely, if $a$ has infinite spectrum, choose an
accumulation point $x \in \sigma(a)$.  Let $f$ be a continuous
function on $\sigma(a)$ such that $f(t)$ is nonzero if and only if
$t \neq x$. Then, $f(a)$ is positive and has zero as an accumulation
point of its spectrum, whence $f(a)$ is thus purely positive.
\end{proof}

\begin{proposition}
$A$ is a finite dimensional algebra if and only if $\W(A)=\V(A)$.
\end{proposition}
\begin{proof}
This follows from the previous lemma and the fact that, in this
case, $\W(A)=\V(A)\sqcup\W(A)_+$, and $\W(A)_+=\emptyset$.
\end{proof}

\subsubsection{Purely infinite simple algebras}

\begin{definition}
A (unital) simple {\rm C}$^*$-algebra $A$ is termed purely infinite
if $A$ is infinite dimensional and for any non-zero element $a$ in
$A$, there are $x$, and $y$ in $A$ with $xay=1$.
\end{definition}

\begin{theorem}
{\rm (Lin, Zhang, \cite{linz})} $A$ is simple and purely infinite if and only if
$a\precsim b$ for any (non-zero) $a$ and $b$.
\end{theorem}

\begin{proposition}
If $A$ is a purely infinite simple algebra, then
$\W(A)=\{0,\infty\}$, with $\infty+\infty=\infty$.
\end{proposition}
\begin{proof}
By the above theorem, any two non-zero elements are Cuntz
equivalent, hence they define a single class $\infty$.
\end{proof}

\subsubsection{Algebras with real rank zero and stable rank one}

We first give the original definition, due to Brown and Pedersen
(\cite{bp}), of what is meant for a C$^*$-algebra to have real rank
zero.

\begin{definition}
A {\rm C}$^*$-algebra $A$ has real rank zero if every self-adjoint
element can be approximated arbitrarily well by self-adjoint,
invertible elements.
\end{definition}

This definition, denoted in symbols as $RR(A)=0$, is the lowest
instance of the so-called real rank and captures the dimension of
the underlying space in the commutative case. In other words,
$RR(C(X))=\dim (X)$. Brown and Pedersen showed that the real rank
zero condition is equivalent to a number of other conditions that
ensure projections on demand.

\begin{theorem}
{\rm (\cite[Theorem 2.6]{bp})} For a {\rm C}$^*$-algebra $A$, the
following are equivalent:
\begin{enumerate}[{\rm (i)}]
\item $RR(A)=0$; \item (FS) Elements with finite spectrum are
dense; \item (HP) Every hereditary subalgebra has an approximate
unit consisting of projections.
\end{enumerate}
\end {theorem}

The class of algebras with real rank zero is huge and has been
studied for a number of years. It includes purely infinite simple
algebras, all AF algebras, irrational rotation algebras and some of
the examples analysed by Goodearl (in \cite{gpubmat}).

The appropriate notion to understand the Cuntz semigroup for such an
algebra is defined below.

\begin{definition}
Let $M$ be an abelian semigroup, ordered algebraically. An interval
in $M$ is a non-empty subset $I$ of $M$ which is order-hereditary
(i.e. $x\leq y$ and $y\in I$ implies that $x\in I$) and upward
directed. We say that an interval $I$ is countably generated if
there is an increasing sequence $(x_n)$ in $I$ such that
\[
I=\{x\in M\mid x\leq x_n\text{ for some }n\}\,.
\]
\end{definition}

Intervals in a semigroup $M$ as above can be added, as follows:
given intervals $I$ and $J$, define
\[
I+J=\{x\in M\mid x\leq y+z\text{ with }y\in I\text{ and }z\in J\}\,,
\]
which then becomes an interval in $M$, and is countably generated if
both $I$ and $J$ are.

The set $\Lambda(M)$ of all intervals in $M$ is a semigroup under
addition, and we denote by $\Lambda_{\sigma}(M)$ the abelian
subsemigroup consisting of those intervals that are countably
generated. Note that this semigroup admits a natural partial
ordering, given by inclusion, that extends the algebraic ordering
but is typically non-algebraic.

For a (separable) C$^*$-algebra $A$, denote by
$D(A)=\{[p]\in\V(A)\mid p\in A\}$, which is an element of
$\Lambda_{\sigma}(\V(A))$. Let us denote by $\Lambda_{\sigma,
D(A)}(\V(A))$ the subsemigroup of $\Lambda_{\sigma}(\V(A))$
consisting of those intervals that are contained in some copies of
$D(A)$. If, furthermore, $A$ has real rank zero, we may define, for
any $a\in M_{\infty}(A)_+$,
\[
I(a)=\{[p]\in \V(A)\mid p\precsim a\}\,.
\]
This is also a countably generated interval in $\V(A)$. In fact, if
$(p_n)$ is an approximate unit for the hereditary algebra generated
by $a$, one can show that
\[
I(a)=\{[p]\in\V(A)\mid p\precsim p_n\text{ for some }n\}\,.
\]
The main result is then (\cite{perijm}):
\begin{theorem}
If $A$ is separable and has real rank zero, the correspondence
\[
\W(A)\to\Lambda_{\sigma, D(A)}(\V(A))\,,
\]
given by $\langle a\rangle\mapsto I(a)$, defines an order-embedding,
which is surjective if furthermore $A$ has stable rank one.
\end{theorem}

It follows from this that
$\W(\mathbb{K})=\mathbb{N}_0\sqcup\{\infty\}$.

Since also $\W(\mathbb{B}/\mathbb{K})=\{0,\infty\}$, we see that
$\W(\mathbb{B})=\mathbb{N}_0\sqcup\{\infty\}\sqcup\{\infty'\}$ with
$\infty+\infty'=\infty'$.

\subsubsection{Irrational rotation algebras}

Let $\theta$ be an irrational number, and let $A_{\theta}^0$ be the
(unital) universal $^*$-algebra generated by two elements $u$ and
$v$ satisfying the relations $u^*u=v^*v=uu^*=vv^*=1$ and $vu=e^{2\pi
i\theta}uv$. By a representation of $A_{\theta}^0$ we mean a Hilbert
space $\mathcal{H}$ and a $^*$-homomorphism
\[
\pi\colon A_{\theta}^0\to \mathbb{B}(\mathcal{H})
\]
such that $\pi (u)$ and $\pi (v)$ are unitary operators. These
representations may be used to define a universal C$^*$-norm on
$A_{\theta}^0$, as follows. Let:
\[
\Vert x\Vert =\sup \{ \Vert \pi (x)\Vert \mid \pi \text{ a
representation of } A_{\theta}^0 \},\quad x\in A_{\theta}^0\,.
\]
Notice that $\Vert x\Vert $ is always finite (because
\[
\Vert \pi \sum {\lambda}_{nm}u^nv^m\Vert \leq \sum
|{\lambda}_{nm}|{\Vert \pi u\Vert }^n{\Vert \pi v\Vert }^m=\sum
|{\lambda}_{nm}|
\]
as $\pi u$ and $\pi v$ are required to be unitaries). Moreover,
since the $\pi $'s are $^*$-morphisms and the operator norms are
C$^*$-norms, $\Vert \text{$\cdot$} \Vert $ is a C$^*$-norm. Define
the (irrational) \emph{rotation algebra} $A_{\theta}$ to be the
completion of $A_{\theta}^0$ with respect to this norm
(see~\cite[Chapter 12]{WO} for a more complete discussion on
$A_{\theta}$).

By~\cite[Theorem 1.5]{BKR}, $RR(A_{\theta})=0$, and
$sr(A_{\theta})=1$. Finally, it is proved in \cite{Ri2} that
$\mathrm{K}_0(A_{\theta})$ is order-isomorphic to $\mathbb{Z}+\theta
\mathbb{Z}$, ordered as a subgroup of $\mathbb{R}$, whence it is
totally ordered and in particular it is lattice-ordered.

It thus follows that
$\V(A_{\theta})=(\mathbb{Z}+{\theta}{\mathbb{Z})}^+$. Assume that
$0<\theta <1$. As before, each bounded interval will be countable,
and loosely speaking it is determined in some sense by its bound.
Geometrically, the monoid $\W(A_{\theta})$ can be studied by
considering the pairs of integers $(x,y)$ satisfying $0\leq y+\theta
x\leq \alpha$ (or $0\leq y+\theta x< \alpha$), for a real number
$\alpha$. It turns out then that each element of the monoid can be
thought of the collection of points of the form $(c,\theta d)$ for
$c,d\in \mathbb{Z}$, lying between the lines $y=-\theta x$ and
$y=-\theta x+\alpha$ in the plane (the points in this last line can
be included or not, depending on whether $\alpha$ is an integer or
an irrational number of the form $a+\theta b$ for $a,b\in
\mathbb{Z}$), being the order of these points determined by the
parallel projection onto the vertical axis along the line $y=-\theta
x$. Thus setting $U=\{ a+b\theta\geq 0\mid a,b\in \mathbb{Z} \}$ we
have in a similar way as before that $\W(A)$ is order-isomorphic to
a disjoint union of $\mathbb{R}^{++}$ and a copy $U'$ of $U$.

\subsubsection{Goodearl algebras}

Let $X$ be a nonempty separable compact Hausdorff space, and choose
elements $x_1$, $x_2,\ldots \in X$ such that $\{ x_n, x_{n+1},\ldots
\}$ is dense in $X$ for each $n$. For all positive integers $n$ and
$k$, let
\[
\delta_n\colon M_k(C(X))\to M_k(\mathbb{C})\subseteq M_k(C(X))
\]
be the C$^*$-homomorphism given by evaluation at $x_n$. Let $\ups
(1), \ups (2),\ldots $ be positive integers such that $\ups (n)\mid
\ups (n+1)$ for all $n$, and set $A_n=M_{\ups (n)}(C(X))$ for each
$n$. We choose maps $\phi_n\colon A_n \to A_{n+1}$ of the following
form:
\[
\phi_n(a)=(a,\ldots ,a,\delta_n(a),\ldots ,\delta_n (a))\,.
\]

Let $\alpha_n$ the number of identity maps involved in the
definition of $\phi_n$, and set
\[
\phi_{s,n}=\phi_{s-1}\phi_{s-2}\ldots \phi_n\colon A_n\to A_s
\]
for $s>n$. Finally, let $A$ be the C$^*$-inductive limit of the
sequence $\{A_n, \phi_n\}$.

The key assumption here is to assume that in each of the maps
$\Phi_n$, at least one identity map and at least one $\delta_n$
occurs, that is $0<\alpha_n< \frac{\ups (n+1)}{\ups (n)}$.

Any C$^*$-algebra constructed as above is a simple unital
C$^*$-algebra with stable rank one \cite[Lemma 1, Theorem
3]{gpubmat}. Defining the {\it weighted identity ratio} for ${\phi
}_{s,n}$ as ${\omega}_{s,n}={\alpha}_n{\alpha}_{n+1}\ldots
{\alpha}_{s-1}\frac{\ups (n)}{\ups(s)}$, it turns out that $A$ has
real rank zero if and only if either $\lim\limits_{t\to \infty}
{\omega}_{t,1}=0$ or $X$ is totally disconnected \cite[Theorem
9]{gpubmat}. The complementary conditions of course ensure that the
real rank is one \cite[Theorem 6]{gpubmat}. Further, all these
algebras have weak unperforation on their Grothendieck groups.

Assume from now on that $X$ is connected. We will compute the monoid
of bounded intervals in $\V(A)$ (in fact, $\V(A)$ satisfies the
Riesz decomposition property if the algebra has either real rank
zero or one, as is observed in \cite[Section 6]{gpubmat}). Let
$U=\cup_{n=1}^{\infty} \frac{1}{\ups (n)}{\mathbb{Z}}\subseteq
\mathbb{Q}$ and $V=\cup_{n=1}^{\infty} \frac{1}{\alpha_1\ldots
\alpha_n}{\mathbb{Z}}\subseteq \mathbb{Q}$. Then it is proved in
\cite[Theorem 13(b)]{gpubmat} that $\V(A)$ is order isomorphic to
\[
W^+=\{(0,0)\} \cup \{ (a,b)\in W\mid a>0\}\,,
\]
where $W=U\oplus (V\otimes {\rm ker} t)$, with $w=(1,0)$ as
order-unit, and $t$ is the unique state on
$(\mathrm{K}_0(C(X)),[1_{C(X)}])$.

Notice that it follows from the key assumption that $\ups (n)$ is a
strictly increasing sequence. For any real number $a$, let us denote
$[0,a)=\{\frac{m}{\ups (n)}\mid m\in {\mathbb{N}},\quad 0\leq
\frac{m}{\ups (n)}< a\}$. Since $A$ is unital, the generating
interval is $D=\{ (a,b)\in W^+\mid 0\leq a< 1\}\cup \{ (1,0)\}$.
Now, every bounded interval (which will be countably generated,
since $U$ is countable) has the form:
\[
\{ (a,b)\in W^+\mid 0\leq a<\alpha \}, \quad \alpha \in
\mathbb{R}^+\,,
\]
or
\[
\{ (a,b')\in W^+\mid (0,0)\leq (a,b')\leq (\beta, b) \}, \quad
\beta \in U^+, b\in V\otimes \ker t\,.
\]
Identifying these intervals with $[0,\alpha), \alpha \in
\mathbb{R}^+$, $[0,(\beta,b)]=[0,\beta )\cup \{(\beta,b)\}, \beta
\in U^+, b\in V\otimes \ker t$, we get an ordered-monoid isomorphism
from $\W(A)$ to the disjoint union $\mathbb{R}^{++}\sqcup {W^+}$,
given by $[0,x)\mapsto x$ for any $x\in \mathbb{R}^{++}$ and
$[0,\beta)\cup \{(\beta, b)\}\mapsto (\beta,b)$, for $\beta \in
U^+,b\in V\otimes kert$. Notice that in this context $\W(A)$ is not
order-cancellative, for $[0,1)+[0,1)=[0,1)+([0,1)\cup \{ (1,0)\})$,
while $[0,1)\not= [0,1)\cup \{(1,0)\} $.

\vskip2cm


\section{Hilbert C*-modules. Kasparov's Theorem}

\vskip1cm

\subsection{Introduction}

This is an overview of some standard facts on Hilbert C*-modules.
Good references for this section are the books by Manuilov and
Troitsky \cite{ManTro} and Lance \cite{LAN}. In particular we have
followed (parts of) Chapters 1 and 2 of the book \cite{ManTro} in
our exposition.


\subsection{Hilbert C$^*$-modules}\label{sect:Hilbert}

Hilbert modules were introduced in the "Appendix" of this book
\cite{ALP}. We recall the basic definition here:

\begin{definition}
\label{def:HilbC*} Let $A$ be a C$^*$-algebra. A (right) Hilbert
$A$-module is a right $A$-module $X$ together with an $A$-valued
inner product $X\times X\rightarrow A$, $(x,y)\mapsto\langle x , y
\rangle$,
enjoying the properties
\smallskip
\begin{itemize}
\item[{\rm (i)}] $\langle x ,\alpha y +\beta z\rangle
                   =\alpha\langle x, y \rangle +\beta\langle x ,z \rangle$,
\item[{\rm (ii)}] $\langle x, ya\rangle=\langle x ,y\rangle a$,
\item[{\rm (iii)}] $\langle y , x \rangle={\langle x, y\rangle}^*$
\item[{\rm (iv)}] $\langle x , x \rangle \ge0\;$ and $\;\langle x ,x \rangle=0$ if and only if
$x=0$.
\end{itemize}
\smallskip\noindent
for all $x,y,z\in X$, $a\in A$, $\alpha,\beta\in\mathbb C$;
moreover, $X$ is complete with respect to the norm given by
$\|x\|^2:=\|\langle x ,x\rangle \|$.
\end{definition}

A basic example of a Hilbert $A$-module is provided by a closed
right ideal $J$ of $A$. In this case, the inner product is given by
$\langle x,y\rangle =x^*y$ for $x,y\in J$. In particular this is the
inner product of the right Hilbert module $A$.

A {\it pre-Hilbert $A$-module} is a right $A$-module $X$ with an
inner product satisfying all the properties of a Hilbert module,
except possibly the completeness condition. The completion
$\overline{X}$ of a pre-Hilbert $A$-module $X$ is a Hilbert
$A$-module in a natural way.

\begin{proposition}
\label{Cauchy-Schw} Let $X$ be a pre-Hilbert $A$-module. Then the
following properties hold:
\begin{itemize}
\item[{\rm (i)}] $\|xa\|\le \|x\| \|a\|$, for all $x\in X$ and $a\in
A$;
\item[{\rm (ii)}] $\langle x, y\rangle\langle y,x\rangle \le\|y\|^2\langle x ,x\rangle $,
for all $x,y\in X$;
\item[{\rm (iii)}] $\| \langle x , y \rangle\|\le \|x\|\,  \| y\|$ for all
$x,y\in X$.
\end{itemize}
\end{proposition}

\begin{proof}
(i) We have $\| xa\|^2=\| \langle xa,xa\rangle \| =\|a^*\langle
x,x\rangle a\|\le \|a\|^2 \|\langle x,x \rangle \| =\| x\|^2 \|a\|
^2$.

\smallskip

(ii) Let $\varphi$ be a positive linear functional on $A$, and let
$x,y$ be elements in $X$. Then $(x,y)\mapsto \varphi (\langle
x,y\rangle )$ is a (degenerated) complex-valued inner product on
$X$, and so the familiar Cauchy-Schwartz inequality gives
\begin{align*}
\varphi (\langle x,y\rangle \langle y,x\rangle ) & =\varphi (\langle
x,y\langle y,x\rangle\rangle)\\ &  \le \varphi (\langle
x,x\rangle)^{1/2}\varphi (\langle y\langle y,x\rangle ,y\langle
y,x\rangle \rangle )^{1/2}\\
& = \varphi (\langle x,x\rangle )^{1/2}\varphi (\langle x,y \rangle
\langle y,y \rangle \langle y,x\rangle )^{1/2}\\
& \le \varphi (\langle x,x\rangle)^{1/2}\| \langle y,y\rangle
\|^{1/2}
\, \varphi (\langle x,y\rangle \langle y,x\rangle )^{1/2}.
\end{align*}
Multiplying both sides on the right by $\varphi (\langle x,y\rangle
\langle y,x\rangle )^{-1/2}$ and then squaring, we get $\varphi
(\langle x,y\rangle \langle y,x\rangle)\le \|y\|^2 \varphi (\langle
x,x\rangle )$. Since this holds for every $\varphi \in A^*_+$, we
get that $\langle x,y\rangle \langle y,x\rangle \le \|y\| ^2 \langle
x,x\rangle$.

\smallskip

(iii) Take norms in (ii).
\end{proof}

\begin{example}
\label{exam:finitesums}(Finite sums) Let $X_1,\dots ,X_n$ be Hilbert
$A$-modules over a C$^*$-algebra $A$. Then the direct sum $X_1\oplus
\cdots \oplus X_n$ is a Hilbert $A$-module, with the inner product:
$$\langle (x_1,\dots ,x_n),(y_1,\dots ,y_n)\rangle =
\sum _{i=1}^n \langle x_i,y_i \rangle.$$
\end{example}

The verification of the properties (i)-(iv) in Definition
\ref{def:HilbC*} is in this case straightforward. Note that the
Cauchy-Schwartz inequality applied to $\oplus _{i=1}^n X_i$ gives
that \begin{equation}\label{eq:CSfinitesums} \|\sum _{i=1}^n \langle
x_i,y_i\rangle\| \le \|\sum _{i=1}^n \langle x_i,x_i\rangle \|^{1/2}
\| \sum _{i=1}^n \langle y_i,y_i\rangle \|^{1/2}  .
\end{equation}

\begin{example}
\label{exam:infinitesums} (Countable sums) Let $\{ X_i : i\in
\mathbb N \}$ be a countable collection of Hilbert $A$-modules. Then
we define the direct sum $\bigoplus X_i$ as
$$\bigoplus X_i =\{ (x_i)\in \prod _{i\in \mathbb N} X_i : \sum _i \langle x_i,x_i\rangle \text{ is convergent
in } A \}.$$ The inner product in $\bigoplus _i X_i$ is defined by
$$\langle (x_i),(y_i)\rangle = \sum _i \langle x_i,y_i\rangle .$$
Then $\bigoplus _i X_i$ is a Hilbert $A$-module.
\end{example}

In this case the verification of the properties of Hilbert module is
more painful. Let us work out the details. First note that, for
$(x_i), (y_i)\in \bigoplus X_i$, given $\epsilon >0$ there exists
$N$ such that for all $n$ we have
$$\| \sum _{i=N}^{N+n} \langle x_i,x_i\rangle \|<\epsilon \, ,\qquad \|\sum _{i=N}^{N+n}
\langle y_i,y_i\rangle \|<\epsilon \, ,$$
so that by the Cauchy-Schwartz inequality (\ref{eq:CSfinitesums}) we get
$$\| \sum _{i=N}^{N+n} \langle x_i,y_i\rangle \|\le
\| \sum _{i=N}^{N+n} \langle x_i,x_i\rangle \|^{1/2}\| \sum
_{i=N}^{N+n} \langle y_i,y_i\rangle \|^{1/2} <\epsilon . $$

This shows that the inner product is well-defined. Now we will check
the completeness of $\bigoplus X_i$. Let $x^{(n)}=(x^{(n)}_i)\in
\bigoplus X_i$ be a Cauchy sequence. Given $\epsilon >0$ there is
$N_0$ such that for $n,m\ge N_0$ we have $\| x^{(n)}-x^{(m)}\|^2
<\epsilon $. Thus for $n,m\ge N_0$ we get
\begin{equation}
\label{eq:first-c} \| \sum _{i=1}^{\infty} \langle
x_i^{(n)}-x_i^{(m)}, x_i^{(n)}-x_i^{(m)} \rangle \| <\epsilon .
\end{equation}
In particular, given $i$, we have $\|\langle
x_i^{(n)}-x_i^{(m)},x_i^{(n)}-x_i^{(m)}\rangle \| <\epsilon$ for
$n,m\ge N_0$, and we see that there exists $x_i=\lim _n x_i^{(n)}\in
X_i$.

Set $x=(x_1,x_2,\dots )$. We want to see that $x\in \bigoplus _i
X_i$ and that $x=\lim _n x^{(n)}$.

We first check that $x\in \bigoplus _i X_i$. Given $\epsilon >0$, we
choose $N_0$ such that (\ref{eq:first-c}) holds for all $m,n\ge
N_0$. Since $x^{(n)}\in \bigoplus _i X_i$, we have $\sum
_{i=1}^{\infty} \langle x_i^{(n)},x_i^{(n)}\rangle \in A$. Take
$n_0>N_0$. There is $N_1\ge N_0$ such that
\begin{equation}
\label{eq:second-c} \|\sum _{N_1}^{N_1+k} \langle
x_i^{(n_0)},x_i^{(n_0)}\rangle \| <\epsilon
\end{equation}
for all $k\ge 0$. For $m>N_0$ and $k\ge 0$ we have
\begin{align*}
\label{eq:third-c} \| \sum _{N_1}^{N_1+k} \langle
x_i^{(m)},x_i^{(m)}\rangle  \| & \le \| \sum _{N_1}^{N_1+k} \langle
x_i^{(m)}-x_i^{(n_0)}, x_i^{(m)}-x_i^{(n_0)}\rangle \| + \| \sum
_{N_1}^{N_1+k} \langle x_i^{(m)}- x_i^{(n_0)}, x_i^{(n_0)}\rangle \|
\\ & + \| \sum _{N_1}^{N_1+k} \langle
x_i^{(n_0)}, x_i^{(m)}-x_i^{(n_0)}\rangle \| +\| \sum _{N_1}^{N_1+k}
\langle x_i^{(n_0)}, x_i^{(n_0)}\rangle \|\\
& < \epsilon
+\epsilon^{1/2}\epsilon^{1/2}+\epsilon^{1/2}\epsilon^{1/2}+\epsilon
=4\epsilon .
\end{align*}
Hence we get $\| \sum _{N_1}^{N_1+k} \langle
x_i^{(m)},x_i^{(m)}\rangle \| <4\epsilon $. Now for $k\ge 0$ we have
\begin{align*}
\|\sum _{N_1}^{N_1+k} \langle x_i,x_i\rangle \| & \le \| \sum
_{N_1}^{N_1+k} \langle x_i^{(m)},x_i^{(m)} \rangle \| + \| \sum
_{N_1}^{N_1+k} \langle x_i^{(m)},x_i- x_i^{(m)} \rangle \|  \\
& + \| \sum _{N_1}^{N_1+k} \langle x_i- x_i^{(m)},x_i^{(m)} \rangle \|
+\| \sum _{N_1}^{N_1+k} \langle x_i-x_i^{(m)},x_i - x_i^{(m)}
\rangle \| <4\epsilon + \delta _m
\end{align*}
for every $m>N_0$. Since $x_i^{(m)}\to x_i$ for $i=N_1,\dots
,N_1+k$, we get $\delta_m \to 0$. We conclude that $\| \sum
_{N_1}^{N_1+k} \langle x_i,x_i\rangle \| \le 4\epsilon $. Thus $x\in
\bigoplus _i X_i$.

Finally we check that $x=\lim x^{(n)}$. Now for $\epsilon$ and $N_0$
as before, and $m,n>N_0$, we get
\begin{align*}
\| \sum _{i=1}^T \langle x_i-x_i^{(n)},x_i-x_i^{(n)} \rangle \| &
\le \| \sum _{i=1}^T \langle
x_i^{(m)}-x_i^{(n)},x_i^{(m)}-x_i^{(n)}\rangle  \| +\| \sum _{i=1}^T
\langle x_i-x_i^{(m)},x_i-x_i^{(m)} \rangle \| \\
& +\| \sum _{i=1}^T \langle x_i-x_i^{(m)},x_i^{(m)}-x_i^{(n)}\rangle
\| +\| \sum _{i=1}^T \langle x_i^{(m)} -x_i^{(n)},x_i-x_i^{(m)}
\rangle \| \\
& <\epsilon +\delta _m .
\end{align*}
Since $\delta _m\to 0$, we get $\| \sum _{i=1}^T \langle
x_i-x_i^{(n)},x_i-x_i^{(n)} \rangle \|\le \epsilon$. Since this
holds for every $T$ we get $\| \sum _{i=1}^{\infty} \langle
x_i-x_i^{(n)},x_i-x_i^{(n)} \rangle \|\le \epsilon$  for $n>N_0$. We
conclude that $\lim x^{(n)} =x$.

\bigskip

We denote by $\widetilde{A}$ the C$^*$-subalgebra of $M(A)$ generated
by $A$ and $1_{M(A)}$. Note that $\widetilde{A}=A+1_{M(A)}\mathbb
C$. Note that every Hilbert $A$-module is also a Hilbert module over
$\widetilde{A}$. Indeed we can identify Hilbert $A$-modules with
Hilbert $\widetilde{A}$-modules $X$ such that $\langle X,X\rangle
\subseteq A$.

\bigskip

We end this section with a couple of useful technical lemmas.

\bigskip

\begin{lemma}
\label{lem:approxunits} If $(e_{\alpha})$ is an approximate unit in
$A$ and $X$ is a Hilbert $A$-module, then we have that $x=\lim
_{\alpha} xe_{\alpha}$ for every $x\in X$.
\end{lemma}

\begin{proof}
For $x\in X$ we have
\begin{align*}
\|x-xe_{\alpha}\|^2 & = \| \langle
x-xe_{\alpha},x-xe_{\alpha}\rangle \| \\
& = \| (1-e_{\alpha})\langle x,x\rangle
(1-e_{\alpha})\|\longrightarrow 0.
\end{align*}
\end{proof}

\begin{lemma}
\label{lem:exact dec} Let $X$ be a Hilbert $A$-module, $0<\alpha
<1/2$, and $x\in X$. Then there is $z\in X$ such that $x=z\langle
x,x\rangle ^{\alpha}$ and $\|z\| \le \|\langle x,x\rangle
\|^{\frac{1}{2}-\alpha}$.
\end{lemma}

\begin{proof}
The proof is the same as the one of Proposition \ref{Pedersen2}.
\end{proof}

\subsection{Kasparov's Theorem}\label{Kasparov}

For a C$^*$-algebra $A$, we will denote by $H_A$ the direct sum of
countably many copies of the Hilbert $A$-module $A$, that is
$H_A=A\oplus A\oplus \cdots $.

\bigskip

\begin{definition}
\label{coungen} A Hilbert $A$-module $X$ is {\it countably
generated} if there is a sequence $(x_i)$ of elements of $X$ such
that $X=\overline{\sum _{i=1}^{\infty} x_i A}$.
\end{definition}

\begin{theorem}
\label{them: Kasparov}{\rm (Kasparov's Theorem)} If $X$ is a
countably generated Hilbert $A$-module then $$X\oplus H_A\cong
H_A.$$ \qed\end{theorem}

For a proof of Kasparov's Theorem, see \cite[Theorem 1.4.2]{ManTro}.

\begin{definition}
\label{projHilbert} Let $X$ be a Hilbert $A$-module over a
C$^*$-algebra $A$. We say that $X$ is a {\it finitely generated
projective Hilbert module} if there exists a natural number $n$ and
a Hilbert $\widetilde{A}$-module $Y$ such that $X\oplus Y\cong
L_n(\widetilde{A})$ as Hilbert $\widetilde{A}$-modules.
\end{definition}

\bigskip

Here $L_n(A)$ denotes the Hilbert C*-module $A^n$. We see $L_n(A)$
as a Hilbert submodule of the standard module $H_A$, more precisely
we identify $L_n(A)$ with the submodule $A\oplus \cdots \oplus
A\oplus 0\oplus \cdots $ ($n$ copies of $A$) of $H_A$.

\bigskip

Note that, if $A$ is not unital, then a Hilbert $A$-module is a
finitely generated projective Hilbert module if and only if it is so
as a Hilbert $\widetilde{A}$-module. We will see in Theorem
\ref{thm:fgimpliesproj} that the finitely generated projective
Hilbert modules are exactly the ones that are algebraically finitely
generated.

\bigskip

Recall that an  $A$-module $M$ over a unital ring $A$ is {\it
finitely generated projective} if there is an isomorphism of
$A$-modules $M\oplus N\cong A^n$ for some $n$ and some $A$-module
$N$. Observe that, in this case, there is an idempotent $e\in
M_n(A)$ such that $e(A^n)\cong M$. We can look this idempotent $e$
as an idempotent in $M_{\infty}(A)$. We have that two finitely
generated projective $A$-modules $M$ and $M'$ are isomorphic if and
only if the corresponding idempotents in $M_{\infty}(A)$ are
equivalent. Here the relation of equivalence of idempotents is
defined by $e\sim f$ if and only if there are $x,y\in M_{\infty}(A)$
such that $e=xy$ and $f=yx$.

The relationship of this theory with the theory of Hilbert modules
as well as with the monoid $V(A)$ introduced in \ref{projs} is the
following:

\begin{proposition}
\label{algebraicVA} Let $A$ be a unital C$^*$-algebra.
\begin{itemize}
\item[{\rm (i)}] If $M$ is a finitely generated projective
$A$-module, then there exists a finitely generated projective
Hilbert $A$-module such that $M$ and $X$ are isomorphic $A$-modules.
\item[{\rm (ii)}] If $X$ and $Y$ are finitely generated projective
Hilbert $A$-modules, then $X$ and $Y$ are isomorphic as Hilbert
modules if and only if they are isomorphic as $A$-modules.
\item[{\rm (iii)}] There are monoid isomorphisms $V(A)\cong V'(A)\cong
V''(A)$, where $V(A)$ is the monoid constructed in \ref{projs},
$V'(A)$ is the monoid of isomorphism classes of finitely generated
projective Hilbert $A$-modules, and $V''(A)$ is the monoid of
isomorphism classes of finitely generated projective $A$-modules.
\end{itemize}
\end{proposition}

\begin{proof}
(i) If $M\oplus N\cong A^n$ then there is an idempotent matrix $e$
in $M_n(A)$ such that $e(A^n)\cong M$. There is a projection $p\in
M_n(A)$ such that $eM_n(A)=pM_n(A)$. Indeed set
$z=(1+(e-e^*)(e^*-e))$. Then $z$ is an invertible positive element
in $M_n(A)$ and $ez=ee^*e=ze$. By using this, it is easy to check
that $p:=ee^*z^{-1}$ is a projection, and $ep=p$ and $e=pe$, showing
the equality $eM_n(A)=pM_n(A)$. So $M$ is isomorphic with the
finitely generated projective Hilbert $A$-module $p(A^n)$.

\bigskip

(ii) Assume that $X=p(A^n)$ and $Y=q(A^n)$ are isomorphic as
$A$-modules, where $p$ and $q$ are projections in $M_n(A)$. It
follows that $p$ and $q$ are equivalent just as idempotents, i.e,
there are $y\in pM_n(A)q$  and $x\in qM_n(A)p$ such that $p=yx$ and
$q=xy$. Then $p=(x^*x)(yy^*)=(yy^*)(x^*x)$, so that $x^*x$ is
invertible in  $pM_n(A)p$, with inverse $yy^*$. Now
$w:=x(x^*x)^{-1/2}$ is a partial isometry with $w^*w=p$ and
$ww^*=q$, so that $p$ and $q$ are equivalent as projections. It
follows that $X=p(A^n)$ is isomorphic with $Y=q(A^n)$ as Hilbert
$A$-modules.

\bigskip

(iii) This follows from (i) and (ii).
\end{proof}

Part (iii) of Proposition \ref{algebraicVA} provides us with three
different ``pictures" of $V(A)$.
This characterization will be partially extended later
to the category of countably generated Hilbert modules (Theorem
\ref{CuandW}). Note that the Grothendieck group of $V(A)$ is the
group $K_0(A)$ (in the unital case), and that, by Proposition
\ref{algebraicVA}(iii), the definition of  $K_0(A)$ given in
\cite[Chapter 1]{Rosenberg} (where we look at $A$ as a plain ring)
agrees with the C*-version given in 1.5.3.

\begin{lemma}
\label{lem:approxgen} Let $N$ be an algebraically finitely generated
Hilbert module over a unital C*-algebra $A$. Let $a_1,\dots ,a_s\in
N$ be a family of generators. It follows that there exists $\epsilon
0$ such that, if $a_1',\dots a_s'\in N$ satisfy
$\|a_i-a_i'\|<\epsilon$ for all $i$, then $a_1',\dots ,a_s'$ are
generators of $N$.
\end{lemma}

\begin{proof}
Let $f\colon L_s(A)\to N$ be the $A$-module map defined by
$f(e_i)=a_i$ for all $i$. Then $f$ is adjointable, with adjoint
$N\to L_s(A)$ given by $ x\mapsto (\langle a_1,x\rangle ,\dots
,\langle a_s,x\rangle )$. In particular, we see that $f$ is a
surjective continuous map.

It follows from the open mapping Theorem that, for Banach spaces $E$
and $G$, the set of surjective continuous linear maps is open in the
Banach space $\text{Hom}(E,G)$ of all the continuous linear maps
from $E$ to $G$. It follows that there is $\delta >0$ such that, if
$\|f-g\|<\delta $, then $g$ is surjective. Let $g\colon L_s(A)\to N$
be the right $A$-module homomorphism defined by $g(e_i)=a_i'$, where
$\|a_i-a_i'\|<\epsilon $  for all $i$. If $\alpha =(\alpha _1,\dots
,\alpha _s)\in L_s(A)$ satisfies $\|\alpha \|\le 1$, we get
$$\|(g-f)\alpha \|=\| \sum _{i=1}^s (a_i-a_i')\alpha _i\| \le s\epsilon .$$
It follows that, if $\epsilon <\delta /s$, then $g$ is surjective.
\end{proof}

\begin{theorem}
\label{thm:fgimpliesproj} Let $N$ be an algebraically finitely
generated Hilbert $A$-module over a C$^*$-algebra $A$. Then $N$ is a
finitely generated projective Hilbert module.
\end{theorem}

\begin{proof}
Passing to $\widetilde{A}$, we can assume that $A$ is unital. Let
$a_1, \dots ,a_s$ be a family of generators of $N$. By Lemma
\ref{lem:approxgen}, there exists $\epsilon >0$ such that any family
$a_1',\dots ,a_s'$ of elements of $N$ such that
$\|a_i-a_i'\|<\epsilon$ for all $i$ is a family of generators of
$N$. On the other hand the map $g\colon L_s(A)\to N$ given by
$g(e_i)=a_i$ is open, so there is $\delta >0$ such that every
element $x\in N$ of norm $\le \delta $ can be written as $\sum
a_k\alpha _k$, with $\|(\alpha_1,\dots ,\alpha _k)\|<1$. Choose the
$\epsilon$ above such that, in addition it satisfies the inequality
$\epsilon<\delta/s$.

By Kasparov's Theorem (Theorem \ref{them: Kasparov}), we can assume
that $N\oplus M=H_A$. Let $P\colon H_A\to N$ be the orthogonal
projection of $H_A$ onto $N$. Clearly we have $\|P\|=1$. Let
$Q\colon H_A\to L_{n_0}(A)$ be an orthogonal projection such that
$\|Qa_k-a_k\|<\epsilon $ for all $k=1,\dots ,s$, and write
$\overline{a}_k=Qa_k$. We have an orthogonal decomposition
$\overline{a}_k=a_k'+a_k''$ with $a_k'\in N$ and $a_k''\in M$.
Observe that
$$a_k-a_k'=P(a_k-\overline{a}_k)$$
and so $\|a_k-a_k'\|\le \|a_k-\overline{a}_k\|<\epsilon $ for all
$k$, which implies that $a_1',\dots ,a_k'$ is a family of generators
of $N$.

Let $\overline{N} =\overline{a} _1A+\cdots +
\overline{a}_sA\subseteq L_{n_0}(A)$. Then we have
$H_A=M+\overline{N}$: given $x\in H_A$, we can write
$$x=x_M+\sum a_k'\alpha _k=(x_M-\sum a_k''\alpha _k) +\sum \overline{a}_k\alpha _k\in M+\overline{N} ,$$
where $x_M=(I-P)x$ is the orthogonal projection of $x$ on $M$.

We have continuous surjective maps $Q_1:=Q_{|N}\colon N\to
\overline{N}$ and $P_1:= P_{|\overline{N}}\colon \overline{N}\to N$
such that $P_1Q_1(a_k)=a_k'$ for all $k$. Let $x\in N$ such that
$\|x\|\le 1$. Then there is $(\alpha _1,\dots ,\alpha _s)\in L_s(A)$
such that $\delta x=\sum _{i=1}^s a_k\alpha _k$ and $\|(\alpha
_1,\dots ,\alpha _s)\|< 1$. Then $\|\alpha_i\|< 1$ for all $i$ and
so
$$\| P_1Q_1(x)-x\|=\frac{1}{\delta} \|\sum _{k=1}^s
(a'_k-a_k)\alpha _k\|\le \frac{s\epsilon}{\delta}<1. $$ Hence
$\|P_1Q_1-I\|<1$ in the Banach algebra $\text{Hom}(N,N)$ and thus
$P_1Q_1$ is an isomorphism. It follows that $Q_1$ is injective and
so it is an isomorphism, and so $P_1$ is an isomorphism as well. It
follows that $\sum a_k'\alpha _k =0\implies \sum
\overline{a}_k\alpha _k=0$, and thus $\overline{N}\cap M=\{0\}$.

Hence, we have obtained that $H_A=\overline{N}\widetilde{\oplus } M$
(algebraic direct sum), and since $\overline{N}\subseteq
L_{n_0}(A)$, it follows from the modular law that
$$L_{n_0}(A)=\overline{N}\widetilde{\oplus}(M\cap L_{n_0}(A)),$$
which implies that $\overline{N}$ is a finitely generated projective
$A$-module. Since $N\cong \overline{N}$ as right $A$-modules, it
follows from Proposition \ref{algebraicVA}  that $N$ is a finitely
generated projective Hilbert $A$-module.
\end{proof}


\subsection{The algebra of compact operators}\label{subsection:compoprs}

Let $A$ be a C$^*$-algebra. For Hilbert $A$-modules $X$ and $Y$, denote
by $\mathcal L_A(X,Y)$ the set of all {\it adjointable operators}
$T\colon X\to Y$, so that there exists a (unique) operator (the {\it
adjoint} of $T$) $T^*\colon Y\to X$ such that $\langle Tx,y\rangle
=\langle x,T^* y\rangle$ for all $x\in X$ and all $y\in Y$.

\begin{proposition}
\label{prop:contiadjoint} If $T\in \mathcal L _A(X,Y)$, then $T$ is
$A$-linear and continuous.
\end{proposition}

\begin{proof}
We left $A$-linearity as an exercise for the reader. To show
continuity, we use the closed graph theorem, so assume that
$(x_{\alpha})$ is a net in $X$ converging to $x\in X$ and such that
$Tx_{\alpha}$ converges to $y\in Y$. We have to show that $y=Tx$.

We have

\begin{align*}
0 & = \langle T^*(y-Tx),x_{\alpha}\rangle - \langle
T^*(y-Tx),x_{\alpha}\rangle = \langle y-Tx,Tx_{\alpha}\rangle
-\langle
T^*(y-Tx),x_{\alpha}\rangle \\
& \longrightarrow \langle y-Tx,y\rangle -\langle T^*(y-Tx),x\rangle
=
\langle y-Tx,y-Tx\rangle
\end{align*}
Hence $\langle y-Tx,y-Tx\rangle =0$ and so $y=Tx$ as desired.
\end{proof}

If $X$ is a Hilbert $A$-module then $\mathcal L _A(X)$ is a
C$^*$-algebra with the involution given by $T\mapsto T^*$ and the usual
operator norm. Indeed we have, for $T\in \mathcal L _A(X)$,
$$\|T^*\| \|T\| \ge \| T^*T\|\ge \text{sup}_{\|x\|\le 1}
\|\langle T^*Tx,x\rangle \|=\text{sup}_{\|x\|\le 1}\|\langle Tx,Tx\rangle \|
=\text{sup}_{\|x\|\le 1} \|Tx\|^2 =\| T\| ^2,$$ from which it
follows easily the C*-identity $\|T^*T\|=\|T\|^2$.

\begin{proposition}
\label{prop:positivity} For $T\in \mathcal L _A(X)$, the following
properties are equivalent:
\begin{itemize}
\item[(i)] $T\ge 0$ in $\mathcal L_A(X)$.
\item[(ii)] $\langle T x,x\rangle \ge 0$ for all $x\in X$.
\end{itemize}
\end{proposition}

\begin{proof}
(i)$\implies $(ii): If $T=S^*S$, then $\langle Tx,x\rangle =\langle
Sx,Sx\rangle \ge 0$ for all $x\in X$.

(ii)$\implies $(i): We have $\langle Tx,x\rangle =\langle
Tx,x\rangle ^* = \langle x,Tx\rangle $, and so $\phi(x,y)= \langle
Tx,y\rangle $ gives a sesquilinear form on $X$ such that $\phi
(x,y)=\phi (y,x)$ for all $x,y\in X$. By polarization we get
$\langle Tx,y \rangle = \langle x,Ty\rangle $ for all $x,y\in X$.
Thus $T=T^*$ and we can write $T=T_+-T_-$, with $T_+,T_-\ge 0$ and
$T_+T_-=0$. By hypothesis we have
$$\langle T_+x,x\rangle \ge \langle T_-x,x\rangle $$
for all $x\in X$, and in particular
$$\langle T_-^3x,x\rangle =\langle T_-^2x,T_-x\rangle \le \langle T_+T_-x,T_-x\rangle =0.$$
We get $T_-^3=0$ and so $T_-=0$. This gives $T=T_+\ge 0$, as
desired.
\end{proof}

We now define a particular type of operators in $\mathcal L
_A(X,Y)$, the so-called {\it compact operators}. For $y\in Y$ and
$x\in X$, define $\theta_{y,x} \colon X\to Y$ by
$$\theta _{y,x}(z)=y\langle x,z\rangle .$$
We have the following properties:
\begin{itemize}
\item[(i)] $\theta_{y,x}^*=\theta _{x,y}$;
\item[(ii)] $\theta _{x,y}\theta_{u,v}=\theta _{x\langle y,u\rangle , v}
=\theta _{x,v\langle u,y\rangle }$;
\item[(iii)] $T\theta _{y,x}=\theta _{Ty,x}$ for $T\in
\text{Hom}_A(Y,Z)$;
\item[(iv)] $\theta _{y,x}S= \theta _{y,S^*x}$ for $S\in \mathcal
L _A(Z,X)$.
\end{itemize}

\medskip

Let $\mathcal K (X,Y)$ be the closed linear span of the set
$\{\theta_{y,x}: x\in X,y\in Y\}$ in $\mathcal L_A(X,Y)$, and set
$\mathcal K (X):=\mathcal K (X,X)$. Observe that $\mathcal K (X)$ is
a closed essential ideal of the C$^*$-algebra $\mathcal L_A(X)$.

\begin{proposition}
\label{compactcomp} Let $A$ be a C$^*$-algebra. Then we have:
\begin{itemize}
\item[\rm (i)] $\mathcal K (A)\cong A$;
\item[\rm (ii)] $\mathcal K (L_n(A))\cong M_n(A)$;
\item[\rm (iii)] $\mathcal K (H_A)\cong A\otimes \mathcal K$, where $\mathcal K$
is the algebra $\mathcal K_{\mathbb C}(H_{\mathbb C})$ of compact
operators on an infinite-dimensional separable Hilbert space.
\end{itemize}
\end{proposition}

\begin{proof}
(i) The map sending $\sum _{i=1}^n \theta_{a_i,b_i}$ to $\sum
_{i=1}^n a_ib_i^*$ provides the desired isomorphisms.

\medskip

(ii) Here the isomorphism is given by the rule
$$\theta _{(a_1,\dots ,a_n),(b_1,\dots ,b_n)}\mapsto \begin{pmatrix}
a_1b_1^* & \cdots & a_1b_n^* \\
    & \cdots   &      \\
a_nb_1^* & \cdots  & a_nb_n^*
\end{pmatrix}.$$

\medskip

(iii) By (ii), there is an isomorphism from $\bigcup
_{i=1}^{\infty}\mathcal K (L_n(A))$ onto $\bigcup _{i=1}^{\infty}
M_n(A)$. This extends to an isomorphism of the completions, which
are $\mathcal K (H_A)$ and $A\otimes \mathcal K$ respectively.
\end{proof}

Note that by the Green-Kasparov's Theorem (Appendix) and part (iii)
of the above proposition we have $$\mathcal L_A(H_A)\cong M(\mathcal
K(H_A))\cong M(A\otimes \mathcal K).$$

If $X$ is a Hilbert $A$-module then it is automatically a Hilbert
$\mathcal K(X)-A$-bimodule. The structure of left Hilbert $\mathcal
K (X)$ is given by $(x,y)=\theta _{x,y}$. In particular we have
$\mathcal K(X)X=X$ by the left analogue of Lemma \ref{lem:exact
dec}.

\begin{definition} Let $A$ be a C$^*$-algebra.
\begin{itemize}
\item[(i)] An {\it strictly positive} element of $A$ is a positive
element $h$ in $A$ such that $\varphi (h)>0$ for every state
$\varphi$ on $A$. Equivalently $A=\overline{hA}$.
\item[(ii)]  $A$ is {\it $\sigma$-unital} in case there is a strictly
positive element $h$ in $A$. Equivalently, there is a countable
approximate unit $(e_n)$ for $A$ (\cite[3.10.5]{bPedersen79}).
\end{itemize}
\end{definition}

\begin{remark}
\label{rem:pApapproxunit} If $A$ is $\sigma$-unital and $p$ is a
projection in $M(A)$, then $pAp$ is also $\sigma$-unital. Indeed, if
$(e_n)$ is a sequential approximate unit for $A$, then $(pe_np)$
is an approximate unit for $pAp$.
\end{remark}

\begin{theorem}
\label{cg=Ksigmaunit} A Hilbert $A$-module $X$ is countably
generated if and only if the C$^*$-algebra $\mathcal K (X)$ is $\sigma
$-unital.
\end{theorem}

\begin{proof}
Assume first that $\mathcal K (X)$ is $\sigma$-unital. Let $(\theta
_n)$ be a sequential approximate unit for $\mathcal K (X)$. For
$x\in X$ we have $x=kz$ for some $k\in \mathcal K (X)$ and some
$z\in X$, and it follows that $\theta _n(x)\longrightarrow  x$. For
each $n$, we can find $x_i^{(n)},y_i^{(n)}\in X$, $1\le i\le s(n)$,
such that
$$\|\theta _n-\sum_{i=1}^{s(n)} \theta _{x_i,y_i}\|<1/n.$$
It follows easily that the set $\{x_i^{(n)}:1\le i\le s(n), n\in
\mathbb N\}$ is a generating set of $X$.

Conversely assume that $X$ is countably generated. Obviously we can
assume that $A$ is unital (passing to $\widetilde{A}$).
By Kasparov's Theorem (Theorem \ref{them: Kasparov}), there is a
projection $p$ in $\mathcal L(H_A)=M(\mathcal K (H_A))$ such that $p(H_A)\cong
X$. Since $\mathcal K (X)=p\mathcal K(H_A)p$, it suffices by
Remark \ref{rem:pApapproxunit} to show that $\mathcal K (H_A)=A\otimes \mathcal K$
is $\sigma$-unital. But this is clear: the canonical projections
$e_n= 1\otimes (\sum _{i=1}^n\theta _{e_i,e_i})$ form an approximate
unit for $A\otimes \mathcal K$.
\end{proof}

\vskip2cm

\section{The Category $\mathsf{Cu}$}

\vskip1cm

\subsection{Introduction}
The order on the Cuntz
semigroup is positive ($y \geq 0$ for every $y$), and respects the
semigroup operation ($x \leq y$ and $w \leq z$ implies $x+w \leq
y+z$).  If one views the Cuntz semigroup as a functor into Abelian
semigroups equipped with such an order, then it has a major
shortcoming:  it is not continuous with respect to inductive limits,
i.e., if $(A_i,\phi_i)$ is an inductive sequence of C$^*$-algebras
then
\[
\lim_{i \to \infty} \left( W(A_i),W(\phi_i) \right) \neq W(A)
\]
in general, where the inductive limit is the algebraic one.

In \cite{cei}, Coward, Elliott, and Ivanescu gave a presentation of
the Cuntz semigroup (indeed of a stable version of it, denoted
$\Cu$) which identified some new properties in its order structure.
These properties allowed them to realise the Cuntz semigroup as a
functor into a new, enriched category.  In this setting, the Cuntz
semigroup {\it is} a continuous functor. This section of the course
notes concerns the enriched category of Coward-Elliott-Ivanescu.

After the abstract definition of the category $\mathsf{Cu}$, we
proceed to attach an object of this category to each C$^*$-algebra
$A$, denoted $\Cu$. This construction is based on the consideration
of a suitable relation on the class of countably generated Hilbert
$A$-modules. We will show that this relation nicely simplifies when
$A$ has stable rank one and that $\Cu \cong \W(A\otimes \mathcal
K)$, so that both semigroups coincide if $A$ is a stable
C$^*$-algebra.

After a close study of direct limits in the category $\mathsf{Cu}$,
we show the continuity of the functor $\mathsf{Cu}$ with respect to
inductive limits. We finish this section with a result of Ciuperca,
Robert and Santiago \cite{crs} concerning exactness of the functor
$\mathsf{Cu}$.

Of course, our basic reference for this section is the paper by
Coward, Elliott and Ivanescu \cite{cei}.

\subsection{Definition of $\mathsf{Cu}$}
\begin{definition}  Define a category $\mathsf{Cu}$ as follows.  An object of $\mathsf{Cu}$ is an ordered Abelian semigroup $S$ having the following
properties.
\begin{enumerate}
\item[{\bf (O1)}] $S$ contains a zero element.
\item[{\bf (O2)}] The order on $S$ is compatible with addition, in the sense that $x_1 + x_2 \leq y_1 + y_2$ whenever $x_i \leq y_i$, $i \in \{1,2\}$.
\item[{\bf (O3)}] Every countable upward directed set in $S$ has a supremum.
\item[{\bf (O4)}] For $x,y \in S$ we write $x \ll y$ if whenever $(y_n)$ is an increasing sequence with $\sup_n y_n \geq y$, then there is some
$n$ such that $x \leq y_n$. We say in this setting that $x$ is
\emph{way below} $y$. The set
\[
x^{\ll} = \{ y \in S \ | \ y \ll x \}
\]
is upward directed with respect to both $\leq$ and $\ll$, and
contains a sequence $(x_n)$ such that $x_n \ll x_{n+1}$ for every $n
\in \mathbb{N}$ and $\sup_n x_n = x$.
\item[{\bf (O5)}] The operation of passing to the supremum of a countable upward directed set and the relation $\ll$ are compatible with addition;
if $S_1$ and $S_2$ are countable upward directed sets in $S$ then
$S_1 + S_2$ is upward directed and $\sup (S_1 + S_2) = \sup S_1 +
\sup S_2$;  if $x_i \ll y_i$ for $i \in \{1,2\}$, then $x_1 + x_2
\ll y_1 + y_2$.
\end{enumerate}

\vspace{2mm} \noindent {\bf Note:}  Properties {\bf (O1)}--{\bf
(O5)} were originally introduced by Coward, Elliott, and Ivanescu.
We will add to them the following property, since it is
automatically for all situations of interest:
\begin{enumerate}
\item[{\bf (O6)}] $x \geq 0$ for every $x \in S$.
\end{enumerate}

\noindent The maps of $\mathsf{Cu}$ are semigroup maps preserving
\begin{enumerate}
\item[{\bf (M1)}] the zero element,
\item[{\bf (M2)}] the order,
\item[{\bf (M3)}] suprema of countable upward directed sets,
\item[{\bf (M4)}] and the relation $\ll$.
\end{enumerate}
\end{definition}

Note that if $x \ll y$ and $y \leq z$, then $x \ll z$.  Similarly,
if $x \leq y$ and $y \ll z$, then $x \ll z$. A sequence $x_1 \ll x_2
\ll \cdots$ is said to be {\bf rapidly increasing}.

\subsection{Examples}

\subsubsection{Lower semicontinuous functions.}  Let $X$ be a compact Hausdorff space, and let $L(X)$
denote the set of lower semicontinuous functions on $X$ taking
values in $\mathbb{R}^+ \cup \{\infty\}$.  One can check that $L(X)$
is an object in $\mathsf{Cu}$ under pointwise addition and with the
pointwise order ($f \leq g \Leftrightarrow f(x) \leq g(x), \ \forall
x \in X$).  This remains true if we restrict the possible values of
our functions to $\mathbb{Z}^+ \cup \{\infty\}$.

\subsubsection{Perforation.} Let $S$ be any proper subsemigroup of $\mathbb{Z}^+ \cup \{\infty\}$ which includes $\infty$.  Such an $S$ belongs to $\mathsf{Cu}$.

\subsubsection{The Cuntz semigroup.}  Of course, the notation $\mathsf{Cu}$
comes from the fact that the version of the Cuntz semigroup of a
C$^*$-algebra to be introduced below is an example of an object in
$\mathsf{Cu}$. The first of the examples above occurs in this
manner, while the second {\it cannot} occur in this manner.

\vskip1cm

Throughout this section we make the blanket assumption that all
C$^*$-algebras are $\sigma$-unital and all Hilbert modules are
countably generated, although some of the results stated are valid
in greater generality.

\subsection{Compact containment and Cuntz comparison of Hilbert modules}

The following definitions are fundamental in the sequel.

\begin{definition}
\label{cc} Let $A$ be a {\rm C}$^*$-algebra, and let $X$, $Y$ be
Hilbert $A$-modules. We say that $X$ is \emph{compactly contained
in} $Y$, in symbols $X\CC Y$, provided that there is a self-adjoint
compact operator $a\in\mK(Y)_{sa}$ such that
$a_{|X}=\mathrm{id}_{|X}$.
\end{definition}

\begin{definition}
\label{cuntzsub} Given Hilbert $A$-modules $X$ and $Y$ over a {\rm
C}$^*$-algebra $A$, we say that $X$ is \emph{Cuntz subequivalent} to
$Y$, in symbols $X\precsim Y$, provided that any $X_0\CC X$ is
isomorphic (isometrically) to some $X_0'\CC Y$. We say that $X$ and
$Y$ are \emph{Cuntz equivalent}, in symbols $X\sim Y$, if both
$X\precsim Y$ and $Y\precsim X$.
\end{definition}

\begin{remark}
Observe that if $X\cong Y$, then $X\sim Y$.
\end{remark}

\begin{definition}
\label{cua} Given a {\rm C}$^*$-algebra $A$, consider the set
$\mathcal{H}(A)$ of isomorphism classes of countably generated
Hilbert (right) $A$-modules, and put
\[
\Cu=\mathcal{H}(A)/\sim\,,
\]
where $\sim$ stands for Cuntz equivalence as defined above. We shall
denote the elements of $\Cu$ by $[X]$, where $X$ is a countably
generated Hilbert module. Note that $\Cu$ becomes a partially
ordered set with ordering given by $[X]\leq [Y]$ if $X\precsim Y$.
\end{definition}

The following two lemmas are part of the standard knowledge in
Hilbert module theory. We include proofs as the results will be used
a number of times.
\begin{lemma}
\label{hacont} Let $A$ be a unital {\rm C}$^*$-algebra. If
$H_A\subseteq Y$, then there is an isometric inclusion
$\mK(H_A)\subseteq\mK(Y)$.
\end{lemma}
\begin{proof}
We have
\[
\mK(H_A)\cong A\otimes\mathbb{K}\cong
\overline{\cup_{n=1}^{\infty}M_n(A)}=\overline{\cup_{n=1}^{\infty}\mK(L_n(A))}\,.
\]
Since $L_n(A)\subseteq Y$ and in fact $L_n(A)\oplus
L_n(A)^{\perp}=Y$, we see that, if $y=y_1+y_2$, with $y_1\in L_n(A)$
and $y_2\in L_n(A)^{\perp}$, and $\theta\in \mK(L_n(A))$, then
$\theta(y)=\theta(y_1)$, whence
\[
\Vert\theta\Vert_{\mK(Y)}=\Vert\theta\Vert_{\mK(L_n(A))}=
\Vert\theta\Vert_{\mK(H_A)}\, .\]
\end{proof}

\begin{lemma}
\label{inclofcompacts} Let $A$ be a {\rm C}$^*$-algebra, and let
$X\subseteq Y$ be Hilbert $A$-modules. Then there is a natural
inclusion $\mK(X)\subseteq \mK(Y)$.
\end{lemma}
\begin{proof}
%
We may assume that $A$ is unital. By Kasparov's theorem (see Theorem
\ref{them: Kasparov}), there exists a countably generated Hilbert
module $X'$ such that $X\oplus X'\cong H_A$.

Since we have $H_A\cong X\oplus X'\subseteq Y\oplus X'$, we may as
well use Lemma~\ref{hacont} to conclude that $\mK(X\oplus
X')\subseteq \mK(Y\oplus X')$. On the other hand, if
$\theta\in\mK(X)$, we have that
\[
\Vert\theta\Vert_{\mK(X)}=\Vert\theta\Vert_{\mK(X\oplus X')}\text{
and }\Vert\theta\Vert_{\mK(Y\oplus X')}=\Vert\theta\Vert_{\mK(Y)}\,,
\]
whence the result follows.
\end{proof}

The notion of equivalence just introduced might appear slightly
unnatural at first sight. The lemma below shows that containment of
Hilbert modules is an instance of Cuntz subequivalence.

\begin{lemma}
\label{cuntzgencont} Let $A$ be a {\rm C}$^*$-algebra and let $X$,
$Y$ be Hilbert $A$-modules. If $X\subseteq Y$, then $X\precsim Y$.
\end{lemma}
\begin{proof}
Let $X_0\CC X$. Then, there is $a\in\mK(X)_+$ such that
$a_{|X_0}=\mathrm{id}_{|X_0}$. Since, by Lemma~\ref{inclofcompacts},
we have an isometric inclusion $\mK(X)\subseteq\mK(Y)$, we may take
$X_0'=X_0\CC Y$.
\end{proof}

\begin{lemma}
\label{useful} If $X$ is a Hilbert $A$-module and if
$a\in\mK(X)_{sa}$ satisfies $a_{|Y}=\mathrm{id}_{|Y}$ (in other
words, $Y\CC X$), then $Y\CC\overline{a(X)}$.
\end{lemma}
\begin{proof}
Write down $a=\lim\sum\theta_{x_i,y_i}$, where $x_i$ and $y_i\in X$
(the notation here is admittedly loose). Then
\[
a^3=aaa^*=a(\lim\sum\theta_{x_i,y_i})a^*=\lim\sum\theta_{a(x_i),a(y_i)}\in\mK(\overline{a(X)})\,,
\]
and clearly $a^3_{|Y}=\mathrm{id}_{|Y}$.
\end{proof}

\begin{lemma}
\label{equivalence} Let $X$ be a Hilbert $A$-module, and let
$h\in\mK(X)$. Then, $\overline{hX}\cong\overline{h^*X}$
isometrically, and in particular
$\overline{hh^*X}\cong\overline{h^*hX}$.
\end{lemma}
\begin{proof}
Observe first that if $X$ is a Hilbert (right) $A$-module, then it
becomes a Hilbert (left) $\mK(X)$-module, with structure given by
$[x,y]=\theta_{x,y}$.

Now we have that, from standard facts in C$^*$-algebra theory
\[
\overline{hX}=\overline{h\mK(X)X}=\overline{(hh^*)^{1/2}\mK(X)X}=\overline{hh^*\mK(X)X}=\overline{hh^*X}\,.
\]
Next, let $h=w|h|$ be the polar decomposition of $h$, with $w$ a
partial isometry that sits in $\mK(X)^{**}$. Left multiplication by
$w$ induces then an $A$-module map $\overline{h^*X}\to
\overline{hX}$, with inverse provided by left multiplication by
$w^*$.
\end{proof}

\begin{proposition}
\label{indlimits} Let $X=\lim\limits_{\longrightarrow} X_n$ be an
inductive limit of Hilbert $A$-modules. If $Y\CC X$, then there
exists $n$ and $Y'\CC X_n$ such that $Y\cong Y'$.
\end{proposition}
\begin{proof}
If $Y\CC X$, then there is $b\in\mK(X)_{sa}$ such that
$b_{|Y}=\mathrm{id}_{|Y}$. By taking $b^2$ instead of $b$, we may
assume without loss of generality that in fact $b\geq 0$.

Next, if $h(b)$ is any polynomial in $b$ with zero constant term, we
see that $h(b)(y)=h(1)(y)$ for any $y\in Y$ (because $b$ acts as the
identity on $Y$). Thus, if $f(b)\in C^*(b)\cong C_0(\sigma(b))$,
then by putting $f$ as a limit of polynomial functions as above we
obtain
\[
f(b)(y)=f(1)(y)\text{ for any }y\in Y\,,
\]
from which we conclude that in fact $f(b)_{|Y}=\mathrm{id}_{|Y}$ if
and only if $f(1)=1$.

Let $0<\epsilon<1$, and choose $b'\in C^*(b)$ such that
$\left((b'-\epsilon)_+\right)_{|Y}=\mathrm{id}_{|Y}$. To do so, use
the functional calculus together with the previous observation.
Furthermore, choose $c\in  C^*(b)$ which is a unit for $b'$, i.e.
$cb'=b'c=b'$.

Since $\mK(X)=\lim\limits_{\longrightarrow} \mK(X_n)$, we can find
elements $c_n$ in $\mK(X_n)_+$ such that (their images through the
natural maps) converge in norm to $c$, whence
\[
\lim\limits_{n\to\infty} c_nb'c_n=cb'c=b'\,.
\]
Now, there is $n$ such that
\[
\Vert c_nb'c_n-b'\Vert<\epsilon\,,
\]
hence we may apply Theorem \ref{Kirchberg-Rordam} to obtain $d_n$
with $\Vert d_n\Vert\leq 1$ and
\[
d_nc_nb'c_nd_n^*=(b'-\epsilon)_+\,.
\]
We now observe that (the image of) $X_n$ sits inside $X$, from which
it follows that (the image of) $\mK(X_n)$ sits, isometrically,
inside $\mK(X)$ (invoke Lemma~\ref{inclofcompacts}). Thus, if we
have $x$, $y$ in $X$, $c_n\theta_{x,y}c_n=\theta_{c_n(x),c_n(y)}$,
and since $c_n\in\mK(X_n)$, it follows that $c_n(x)$, $c_n(y)\in
X_n$. This implies, altogether, that $c_nb'c_n\in\mK(X_n)_+$.

Put $g_n=(c_nb'c_n)^{\frac{1}{2}}\in\mK(X_n)_+$, so that we now have
the equality
\[
(d_ng_n)(d_ng_n)^*=(b'-\epsilon)_+\,.
\]
By Lemma~\ref{equivalence}, there is an isometry between
$\overline{(b'-\epsilon)_+X}$ and $\overline{(d_ng_n)^*X}$. Since
$Y\subseteq \overline{(b'-\epsilon)_+X}$, this isometry carries $Y$
onto $Y'\subseteq
\overline{(d_ng_n)^*X}=\overline{g_nd_n^*X}\subseteq X_n$.

Since $Y\CC X$, we may apply Lemma~\ref{useful} to conclude that
$Y\CC \overline{(b'-\epsilon)_+X}$. Therefore
$Y'\CC\overline{(d_ng_n)^*X}\subseteq X_n$, whence $Y'\CC X_n$.
\end{proof}

\subsection{Order structure and existence of suprema}

Recall that if $(M,\leq)$ is an ordered set (in this setting, it
will be a partially ordered abelian semigroup), a supremum of a
countable subset is understood to be the least upper bound of such
set. In other words, if $S\subseteq M$, then $x=\sup S$ if $s\leq x$
for every $s\in S$ and whenever $ y\in M$ is such that $s\leq y$ for
any $s\in S$, then $x\leq y$.

\begin{proposition}
\label{suprema1} Let $X=\lim\limits_{\longrightarrow} X_n$ be an
inductive limit of Hilbert $A$-modules. Then
$[X]=\sup\limits_{n}[X_n]$ in $\Cu$.
\end{proposition}
\begin{proof}
Since $X_n\subseteq X$ for all $n$, we have $[X_n]\leq [X]$. Now let
$[Y]\in\Cu$ and suppose that $[X_n]\leq [Y]$ for all $n$. We are to
show that $[X]\leq [Y]$.

So, let $Z\CC X$. By Proposition~\ref{indlimits}, there is $n$ and
$Z'\CC X_n$ such that $Z'\cong Z$. As, in turn, $[X_n]\leq [Y]$,
there is $Z''\CC Y$ with $Z''\cong Z'\,(\cong Z)$. Altogether this
implies that $[X]\leq [Y]$.
\end{proof}

We now introduce two seemingly different order relations in $\Cu$,
closely related to compact containment, which later on will turn out
to be equivalent.

\begin{definition}
\label{neworder} If $X$ and $Y$ are Hilbert modules, write: $[X]\CC
[Y]$ if there is $X'\CC Y$ with $[X]\leq [X']$.
\end{definition}


\begin{proposition}
\label{suprema2} Given a chain $[X_1]\CC [X_2]\CC [X_3]\CC\cdots$ in
$\Cu$, there is a chain of elements $[X_i']$ with
\begin{enumerate}[{\rm (i)}]
\item $X_1'\CC X_2'\CC\cdots$
\item $[X_i]\leq [X_i']\leq [X_{i+1}]$ for all $i$.
\item $\sup\limits_{n} [X_n']=\sup\limits_{n} [X_n]=[\lim\limits_{\longrightarrow} X_n']$.
\end{enumerate}
\end{proposition}
\begin{proof}
Since $[X_i]\CC [X_{i+1}]$, we can find (for each $i$) a Hilbert
module $X_i''$ with $[X_i]\leq [X_i'']$ and $X_i''\CC X_{i+1}$.
Next, since $[X_{i+1}]\leq [X_{i+1}'']$ and $X_i''\CC X_{i+1}$, we
have an isometry $f_i\colon X_i''\to X_{i+1}''$ (onto its image). We
get in this way an inductive system
\[
X_1''\stackrel{f_1}{\to}X_2''\stackrel{f_2}{\to}X_3''\stackrel{f_3}{\to}\cdots\,.
\]
Put $X=\lim\limits_{\longrightarrow} X_n''$ and denote
$\varphi_n\colon X_n''\to X$ the natural maps. Set
$X_n'=\varphi_n(X_n'')$. Then $[X_n']=[X_n'']$ and by
Proposition~\ref{suprema1}, the increasing sequence $[X_i']$ has a
supremum in $\Cu$, which agrees with $[\lim\limits_{\longrightarrow}
X_n']$. The conclusions in items (i)-(iii) follow easily.
\end{proof}

Before proceeding we record a general observation that will come
useful in a number of instances below.

\begin{remark}
\label{rem} If $X$ is a countably generated Hilbert $A$-module, then
we know by Theorem \ref{cg=Ksigmaunit} that $\mK(X)$ is a
$\sigma$-unital {\em C}$^*$-algebra. Let $(u_n)$ be an approximate
unit for $\mK(X)$ such that $u_{n+1}u_n=u_n$, and set
$X_n=\overline{u_nX}$. Then, since $u_{n+1}\in\mK(X)$ and
${u_{n+1}}_{|X_n}=\mathrm{id}_{|X_n}$, it follows that $X_n\CC X$.
Thus, using Lemma~\ref{useful} we see that
\[
X_n\CC\overline{u_{n+1}X}=X_{n+1}\,.
\]
Therefore we have obtained an increasing chain $X_1\CC X_2\CC
X_3\CC\cdots\CC X$, and since $X=\overline{\mK(X)X}$, we have in
fact that $X=\overline{\cup_{n=1}^{\infty}X_n}=
\lim\limits_{\longrightarrow} X_n$. Therefore,
Proposition~\ref{suprema1} tells us that $[X]=\sup\limits_n[X_n]$ in
$\Cu$.
\end{remark}

\begin{proposition}
\label{supremacc} Given an element $x$ in $\Cu$, the set
\[
x^{\CC}:=\{y\in \Cu\mid y\CC x\}
\]
satisfies the following:
\begin{enumerate}[{\rm (i)}]
\item $x^{\CC}$ is upwards directed.
\item There is a sequence $(x_n)$ in $x^{\CC}$ such that $x_1\CC
x_2\CC x_3\CC\cdots$ and such that $\sup\limits_{n}(x_n)=x=\sup
x^{\CC}$. \item For any $y$ in $x^{\CC}$, there is $n$ such that
$y\CC x_n$.
\end{enumerate}
\end{proposition}
\begin{proof}
Write $x=[X]$. By Remark \ref{rem}, we may also write
$X=\overline{\cup_{n=1}^{\infty}X_n}$, where $X_i\CC X_{i+1}\CC X$.
It follows from this that $[X]=\sup\limits_n [X_n]$.

If we now put $x_n=[X_n]$, we see that $x_n\CC x_{n+1}$ (just by
definition). Therefore condition (ii) follows.

If now $y_1$, $y_2\in x^{\CC}$, we have $y_i\leq [Y_i']$ with
$Y_i'\CC X$ for each $i=1,2$. Use Proposition \ref{indlimits} to
find $n$, $m$ such that $Y_1'\cong Y_1''\CC X_n$ and $Y_2'\cong
Y_2''\CC X_m$ (for some Hilbert modules $Y_i''$, $i=1$, $2$). Then
$y_1$, $y_2\CC x_{n+m}$, thus verifying conditions (i) and (iii).
\end{proof}

\begin{proposition}
\label{suprema} Let $x_1\leq x_2\leq x_3\cdots$ be an increasing
sequence in $\Cu$. Then, there exists a sequence $(x_n')$ in $\Cu$
with $x_n'\leq x_n$, $x_n'\CC x_{n+1}'$ and also
$\sup\limits_{n}x_n'=\sup\limits_n x_n$. In particular, any
increasing sequence in $\Cu$ has a supremum.
\end{proposition}
\begin{proof}
Write $x_n=[X_n]$. Now, for each $n$ choose, as in Remark \ref{rem},
modules $X_{n,m}$ such that
\[
X_{n,1}\CC X_{n,2}\CC\cdots\CC X_n\,,
\]
and $X_n=\overline{\cup_m X_{n,m}}$.

Since we have $[X_{n,m}]\CC [X_n]\leq [X_{n+1}]$, we obtain that
$[X_{n,m}]\CC [X_{n+k}]$ for all $k\geq 0$.

Thus $[X_{n,m}]\in x_{n+k}^{\CC}$.

Given $n$, $m$, apply Proposition~\ref{supremacc} to find $l=l(n,m)$
such that $[X_{n,m}]\CC [X_{n+1, l}]$. By using this observation, we
will construct the sequence $(x_n')$.

Find $l_1$ such that $[X_{1,1}]\CC [X_{2,l_1}]$. Reindexing the
sequence $[X_{2,n}]$ (and throwing out a finite number of the
$X_{2,*}$'s), we may assume that $l_1=2$.
Next, there is
$l_2$ such that
\[
[X_{2,2}]\,,\,[X_{1,2}]\CC [X_{3,l_2}]\,.
\]
Again reindexing the sequence $[X_{3,n}]$ we may assume that
$l_2=3$. In general and after reindexing sequences we may assume
that $[X_{n,m}]$ satisfy $[X_{n,m}]\CC [X_{k+1,k+1}]$, where
$k=\text{max}\{n,m\}$. Put $x_n':=[X_{n,n}]\leq [X_n]=x_n$.

By construction $x_n'\CC x_{n+1}'$, whence the supremum of $(x_n')$
exists (by Proposition~\ref{suprema2}). But now observe that since
$[X_{n,m}]\leq [X_{k+1,k+1}]=x_{k+1}'$,  for all $n$ and $m$, where
$k=\text{max}\{n,m\}$, we have that $\sup\limits_k
x_k'\geq\sup\limits_i x_{n,i}$ for each $n$. This implies that
$x_n\leq \sup x_k'$, from which we conclude that $(x_n)$ also has a
supremum, which agrees with $\sup_k x_k'$, as was to be shown.
\end{proof}

\begin{lemma}
\label{llcc} The relations introduced above, $\ll$ and $\CC$, are
equal.
\end{lemma}
\begin{proof}
Suppose $x\CC y$ in $\Cu$, and suppose there is an increasing
sequence $(y_n)$ with $y\leq y'=\sup_n y_n$. We may choose by
Proposition~\ref{suprema} a sequence $(y_n')$ such that $y_n'\leq
y_n$, $y_n'\CC y_{n+1}'$ and $\sup_n y_n'=\sup_n y_n=y'$.

Use now Proposition~\ref{suprema2} to find another increasing
sequence $(z_n)$ such that $y_n'\leq z_n\leq y_{n+1}'$ with \[
z_n=[Z_n]\,, \,\, Z_n\CC Z_{n+1}\,,\,\,\sup\limits_n
z_n=\sup\limits_n y_n'=y'=[\lim\limits_{\longrightarrow} Z_n]\,.
\]
Now, $x\CC y\leq y'$ implies that $x\CC y'$, whence (by definition
of $\CC$), there is $X'$ such that $x\leq [X']$ and
$X'\CC\lim\limits_{\longrightarrow} Z_n$. Using
Proposition~\ref{indlimits}, there is $n$ and a Hilbert module $Z$
such that $X'\cong Z\CC Z_n$. Therefore,
\[
x\leq [X']\leq [Z_n]\leq y_{n+1}'\leq y_{n+1}\,,
\]
as wanted.

Conversely, if $x\ll y$, write $y=[Y]=\sup y_n$, with $y_n=[Y_n]$
and $Y_1\CC Y_2\CC Y_3\CC\cdots\CC Y$. Then $x\leq [Y_n]$ for some
$n$ and $Y_n\CC Y$, whencefore $x\CC y$.
\end{proof}

We are now ready for the harvest.

\begin{theorem}
\label{cuntzstructure1} Let $A$ be a {\rm C}$^*$-algebra. Then
\begin{enumerate}[{\rm (i)}]
\item Every countable, upwards directed, subset of $\Cu$ has a supremum in $\Cu$.
\item Given $x\in \Cu$, the set
\[
x^{\ll}:=\{y\in \Cu\mid y\ll x\}
\]
is upwards directed (with respect to $\ll$) and contains an
increasing sequence $x_1\ll x_2\ll x_3\ll\cdots$ with $\sup\limits_n
x_n=x$.
\end{enumerate}
\end{theorem}


\subsection{Semigroup structure in \boldmath{$\Cu$}}

The set $\Cu$ may be equipped with a natural addition operation,
under which it becomes an Abelian semigroup with neutral element
$[0]$. This operation is defined, as one would expect, in terms of
direct sums of countably generated Hilbert modules.

The main result that needs to be proved is the following:

\begin{theorem}
For a {\rm C}$^*$-algebra $A$, the object $\Cu$ belongs to the
category $\mathrm{Cu}$.
\end{theorem}

This follows from the results proved previously (see Theorem
\ref{cuntzstructure1}) and the following:

\begin{proposition}
\label{cuntzstructure2} Let $A$ be a {\rm C}$^*$-algebra. Then $\Cu$
is an abelian semigroup with zero, and the operation is compatible
with the natural order $\leq$, with $\ll$ and with taking suprema of
countable, upward directed sets.
\end{proposition}

To prove Proposition \ref{cuntzstructure2}, we assemble a series of
lemmas below and also use Theorem \ref{cuntzstructure1}.

\begin{lemma}
\label{compwithorder} Let $X_i$, $Y_i$ ($i=1$, $2$) be Hilbert
modules and suppose that $X_i\precsim Y_i$. Then
\[
X_1\oplus X_2\precsim Y_1\oplus Y_2\,.
\]
\end{lemma}
\begin{proof}
Write $X_1=\overline{\cup_n X_{1,n}}$ and $X_2=\overline{\cup_n
X_{2,n}}$ for Hilbert modules such that $X_{1,n}\CC X_{1,n+1}$ and
$X_{2,n}\CC X_{2,n+1}$ for all $n$ as in Remark \ref{rem}. Then one
checks that
\[
X_1\oplus X_2=\overline{\cup_n (X_{1,n}\oplus X_{2,n})}
\]
with $X_{1,n}\oplus X_{2,n}\CC X_{1,n+1}\oplus X_{2,n+1}$.

Let $Z\CC X_1\oplus X_2$. We then know (by
Proposition~\ref{indlimits}) that $Z\cong Z'$ for some Hilbert
module $Z'$ such that $Z'\CC X_{1,n}\oplus X_{2,n}$ (for some $n$).
Using now that also $X_i\precsim Y_i$, we can find modules
$(X_{i,n})'\CC Y_i$ (for $i=1,2$) such that $X_{i,n}\cong
(X_{i,n})'$.

Altogether this implies that
\[
Z\cong Z'\CC X_{1,n}\oplus X_{2,n}\cong (X_{1,n})'\oplus
(X_{2,n})'\CC Y_1\oplus Y_2\,,
\]
as was to be shown.
\end{proof}

It follows from Lemma~\ref{compwithorder} that the following is a
good definition:
\begin{definition}
Given elements $[X]$ and $[Y]$ in $\Cu$, define
\[
[X]+[Y]=[X\oplus Y]\,,
\]
and
\[
[X]\leq [Y] \text{ if and only if } X\precsim Y\,.
\]
\end{definition}
In this fashion, $\Cu$ becomes an abelian, partially ordered
semigroup, whose neutral element is the class of the zero module.

\begin{lemma}
\label{compwithsuprema} Let $S_1$ and $S_2$ be two countable,
upwards directed subsets of $\Cu$. Then
\[
\sup (S_1+ S_2)=\sup S_1+\sup S_2\,.
\]
\end{lemma}
\begin{proof} It is obvious that $\sup (S_1+ S_2)\le \sup S_1+\sup
S_2\,.$

Write $S_i=\{s_i^n\}$, $s_i=\sup S_i$ for $i=1,2$, and write
$s_1=[X]$, $s_2=[Y]$ with
\[
X=\overline{\cup_n X_n}\,,\,\, Y=\overline{\cup_n Y_n}\,,
\]
and $X_n\CC X_{n+1}$, $Y_n\CC Y_{n+1}$ for each $n$. Then $X\oplus
Y=\overline{\cup_n (X_n\oplus Y_n)}$.

Now
\[
[X\oplus Y]=[X]+[Y]=s_1+s_2=\sup\{[X_n\oplus
Y_n]\}=\sup_n\{[X_n]+[Y_n]\}\leq \sup (S_1+S_2)\,,
\]
as $[X_n]\leq s_1^{i_n}$ and $[Y_n]\leq s_2^{j_n}$ for some $n$
(because $[X_n]\ll s_1$ and $[Y_n]\ll s_2$).
\end{proof}

\begin{lemma}
\label{compwithll} Suppose that $x_i\ll y_i$ in $\Cu$ (for $i=1,2$).
Then $x_1+ x_2\ll y_1+y_2$.
\end{lemma}
\begin{proof}
We know from Lemma~\ref{llcc} that $x_i\ll y_i$ implies $x_i\CC
y_i$. Thus there are modules $X_i'$ and $Y_i$ with $x_i\leq [X_i']$,
$y_i=[Y_i]$ and $X_i'\CC Y_i$.

Therefore $x_1+x_2\leq [X_1'\oplus X_2']$, $X_1'\oplus X_2'\CC
Y_1\oplus Y_2$ and $y_1+y_2=[Y_1\oplus Y_2]$. This shows that
$x_1+x_2\CC y_1+y_2$ and a second usage of Lemma~\ref{llcc} implies
that $x_1+x_2\ll y_1+y_2$.
\end{proof}

\subsection{Representation of $\Cu$ in the stable rank one case}

The aim of this section is to prove the remarkable result that Cuntz
equivalence of Hilbert modules, as defined previously, amounts to
isomorphism whenever the algebra has stable rank one.

We begin recording the following (well-known):
\begin{lemma}
\label{lance} Let $X$, $Y$ be Hilbert modules. Then the unit ball of
$\mK(X,Y)$ is strictly dense in the unit ball of $\mL(X,Y)$, that
is, given an operator $T\in \mL(X,Y)$ with $\Vert T\Vert\leq 1$,
$\epsilon>0$ and finite sets $F\subseteq X$ and $G\subseteq Y$,
there is $S\in\mK(X,Y)$ with $\Vert S\Vert\leq 1$ and
\[
\Vert S(x)-T(x)\Vert<\epsilon\,,\,\,\Vert
S^*(y)-T^*(y)\Vert<\epsilon\,,
\]
for all $x\in F$ and all $y\in G$.
\end{lemma}

\begin{lemma}
\label{sr1} Let $A$ be a {\rm C}$^*$-algebra with stable rank one.
Then $\mK(X)$ also has stable rank one for any countably generated
Hilbert $A$-module $X$.
\end{lemma}
\begin{proof}
By Kasparov's Theorem (see Theorem \ref{them: Kasparov}), there is a
module $Y'$ such that $Y\oplus Y'\cong H_A$ (the standard Hilbert
module). Identify $\mL(H_A)$ with $\M(A\otimes \mathbb{K})$, which
restricts to an isomorphism $\mK(H_A)\cong A\otimes\mathbb{K}$.

There is a projection $p$ in $\mL(H_A)$ such that $p(H_A)\cong Y$.
It then follows that $\mK(Y)\cong p\mK(H_A)p\cong
p(A\otimes\mathbb{K})p$, and the latter has stable rank one as it is
a hereditary subalgebra of $A\otimes\mathbb{K}$, which also has
stable rank one.
\end{proof}

The following lemma is the key to the proof of the main result in
this section.

\begin{lemma}
\label{lem:key} Let $A$ be a {\rm C}$^*$-algebra with stable rank
one, and let $X$, $Y$, $Z$ be countably generated Hilbert modules
with $X$, $Y\subseteq Z$, and such that there is an isometric
isomomorphism $\varphi\colon X\to Y$. Then, given $1>\epsilon>0$ and
a finite set $F\subseteq (X)_1$ -- the unit ball of $X$ --, there is
a unitary $u$ in $\mK(Z)^{\widetilde{}}$ such that
\[
\Vert u\varphi(x)-x\Vert<\epsilon\,,
\]
for all $x$ in $F$.
\end{lemma}
\begin{proof}
Since $\varphi$ is an isometry, we have that
$\Vert\varphi(x)\Vert=\Vert x\Vert\leq 1$ for every element $x\in
F$. By Lemma~\ref{lance}, there is $\theta\in\mathcal{K}_A(X,Y)$
with $\Vert\theta\Vert\leq 1$ such that
\[
\Vert\varphi(x)-\theta(x)\Vert<\frac{\epsilon^2}{9}
\]
for all $x\in F$.

Consider now the isometric inclusion (see
Lemma~\ref{inclofcompacts}) $\mK(X,Y)\to \mK(Z)$. Since $\mK(Z)$ has
stable rank one (by Lemma~\ref{sr1}), we may find an invertible
element $\gamma$ of $\mK(Z)^{\widetilde{}}$ such that $\Vert
\gamma\Vert\leq 1$ approximating $\theta$ as close as we wish. In
fact, we may choose the approximant $\gamma$ so that
\[
\Vert \varphi(x)-\gamma (x)\Vert<\frac{\epsilon^2}{9}
\]
for all $x\in F$.

Observe now that, for any $x\in F$,

\begin{eqnarray*}
\Vert \langle \gamma(x),\gamma (x)\rangle-\langle x,x\rangle\Vert
&\leq &\Vert \langle
\gamma(x)-\varphi(x),\gamma(x)\rangle\Vert+\Vert \langle \varphi(x)
,\gamma (x)-\varphi(x)\rangle\Vert\\ & + &\Vert \langle \varphi(x)
,\varphi(x)\rangle-\langle x, x\rangle\Vert\leq
\frac{\epsilon^2}{9}+\frac{\epsilon^2}{9}+0=2\frac{\epsilon^2}{9}\,.
\end{eqnarray*}

Thus, taking into account that $(1-|\gamma|)^2\leq 1-|\gamma|^2$, we
see that, for $x\in F$,
\begin{eqnarray*}
\langle (1-|\gamma|)(x), (1-|\gamma|)(x) \rangle & =&\langle
x,(1-|\gamma|)^2(x) \rangle\leq\langle x, (1-|\gamma|^2)(x)\rangle\\
& = &\langle x, (1-\gamma^*\gamma)(x)\rangle\\ & = &\langle x,x
\rangle-\langle \gamma(x),\gamma(x) \rangle\,.
\end{eqnarray*}
Therefore
\[
\Vert
(1-|\gamma|)(x)\Vert\leq\sqrt{2\frac{\epsilon^2}{9}}=\sqrt{2}\frac{\epsilon}{3}\,.
\]
Now let $u^*=\gamma |\gamma|^{-1}$. It is an easy exercise to check
that $u^*$ is unitary (in $\mK(Z)^{\widetilde{}}\,$) and
$\gamma=u^*|\gamma|$. Then, for $x\in F$,
\[
\Vert\gamma(x)-u^*(x) \Vert=\Vert u^*|\gamma|(x)-u^*(x) \Vert=\Vert
|\gamma|(x)-x \Vert\leq\sqrt{2}\frac{\epsilon}{3}\,.
\]
Therefore
\[
\Vert u^*(x)-\varphi(x) \Vert\leq \Vert u^*(x)-\gamma(x)
\Vert+\Vert\gamma(x)-\varphi(x)
\Vert<2\frac{\epsilon}{3}+\frac{\epsilon}{3}=\epsilon\,,
\]
for any $x\in F$, whence $\Vert u\varphi(x)-x\Vert<\epsilon$ for any
$x\in F$, as was to be shown.
\end{proof}

\begin{theorem}
Let $A$ be a {\rm C}$^*$-algebra of stable rank one, and let $X$,
$Y$ be countably generated Hilbert $A$-modules. Then
\begin{enumerate}[{\rm (i)}]
\item $[X]\leq [Y]$ in $\Cu$ if and only if $X\cong X'\subseteq
Y$. \item $[X]=[Y]$ if and only if $X\cong Y$.
\end{enumerate}
\end{theorem}
\begin{proof}
Write $X=\overline{\cup_{i=1}^{\infty} X_i}$, where the $X_i$'s are
countably generated,
\[
X_1\CC X_2\CC\cdots\CC X\,,
\]
and likewise $Y=\overline{\cup_{i=1}^{\infty} Y_i}$ for countably
generated modules $Y_i$ with
\[
Y_1\CC Y_2\CC\cdots\CC Y\,.
\]
It follows from Proposition~\ref{suprema1} that $[X]=\sup\limits_i
[X_i]$ and $[Y]=\sup\limits_i [Y_i]$.

(i). We only need to show that, if $[X]\leq [Y]$, then $X\cong
X'\subseteq Y$.

Since $X_2\CC X$, we have $[X_2]\ll [X]\leq [Y]=\sup\limits_i
[Y_i]$. Therefore, there is an index $i$ such that $[X_2]\leq
[Y_i]$, and since $X_1\CC X_2$, we have by the order-relation in
$\Cu$ that there is an (isometric) isomorphism of $X_1$ onto a
compactly contained submodule of $Y_i$. Ignoring finitely many terms
from the sequence $Y_i$, we may as well assume that $i=1$, so there
is $\widetilde{\varphi}_1\colon X_1\to Y_1$ with
$X_1\cong\widetilde{\varphi}_1(X_1)\CC Y_1$.

Proceed similarly with $X_2\CC X_3$, so after re-indexing the
sequence $Y_i$ (and ignoring again finitely many terms), there is
$\widetilde{\varphi}_2\colon X_2\to Y_2$ with
$X_2\cong\widetilde{\varphi}_2(X_2)\CC Y_2$. Continue in this way,
and construct $\widetilde{\varphi}_i\colon X_i\to Y_i$, with each
$\widetilde{\varphi}_i$ isometry onto its image.

We thus get a diagramme:
\[\begin{CD} X_1 @>{\CC}>> X_2 @>{\CC}>> X_3 @>{\CC}>> \cdots X\\
@V{\widetilde{\varphi}_1}VV  @VV{\widetilde{\varphi}_2}V @VV{\widetilde{\varphi}_3}V @.\\
Y_1 @>{\CC}>> Y_2 @>{\CC}>> Y_3 @>{\CC}>> \cdots Y
\end{CD}\]

Next, label $\{x_i^{(n)}\mid  i=1,2,\ldots\}$ a set of generators of
$X_n$ (with $\Vert x_i^{(n)}\Vert\leq 1$), and put
\[
F_n=\{x_i^{(j)}\mid 1,\leq i,j,\leq n\}\subseteq X_n\,.
\]
Note that each $F_n$ is a finite set, and the union $\cup_n F_n$ is
a countable set that generates $X$. We are going to modify the maps
$\widetilde{\varphi}_i$ via unitaries constructed from
$\mK(Y_i)^{\widetilde{}}$. Set $\varphi_1=\widetilde{\varphi}_1$.
Suppose that $\varphi_1,\ldots,\varphi_n$ have been constructed.

Apply Lemma~\ref{lem:key} to $\varphi_n(X_n)$,
$\widetilde{\varphi}_{n+1}(X_n)\subseteq Y_{n+1}$, the isometric
isomorphism
\[
\widetilde{\varphi}_{n+1}\circ\varphi_n^{-1}\colon \varphi_n(X_n)\to
\widetilde{\varphi}_{n+1}(X_{n+1})\,,
\]
and the finite set $\varphi_n(F_n)$. There is then a unitary
$u_{n+1}\in\mK(Y_{n+1})^{\widetilde{}}$ such that
\[
\Vert u_{n+1}\widetilde{\varphi}_{n+1}\varphi_n^{-1}(y)-y
\Vert<2^{-n}
\]
for all $y\in\varphi_n(F_n)$, that is,
\[
\Vert u_{n+1}\widetilde{\varphi}_{n+1}(x)-\varphi_n(x)
\Vert<2^{-n}\,,
\]
for all $x\in F_n$. Put
$\varphi_{n+1}=u_{n+1}\widetilde{\varphi}_{n+1}$, so that we have
\[
\Vert \varphi_{n+1}(x)-\varphi_n(x) \Vert<2^{-n}\,,
\]
for all $x\in F_n$.

Let us now define $\varphi\colon X\to Y$. Consider the (dense)
subset of $X$:
\[
S=\{x_1^{(n)}a_1+\cdots+x_m^{(n)}a_m\mid a_i\in A,
n,m\in\mathbb{N}\}
\]
If $x=\sum_i x_i^{(n)}a_i$, define
\[
\varphi(x)=\lim\limits_{k\geq
n,\,k\to\infty}\varphi_k(x)=\lim\limits_{k\to\infty}\left(\varphi_k(x_1^{(n)})a_1+\cdots+\varphi_k(x_m^{(n)})a_m
\right)\,.
\]
Let $K=\sup\limits_{1\leq i\leq m}\Vert a_i\Vert$, and let
$k\geq\max\{n,m\}$. Then:
\[
\Vert\varphi_{k+1}(x)-\varphi_k(x)
\Vert\leq\sum\limits_{i=1}^m\Vert\varphi_{k+1}(x_i^{(n)})-\varphi_k(x_i^{(n)})
\Vert K\leq Km2^{-k}\,,
\]
which implies that the sequence $(\varphi_k(x))$ is Cauchy. Thus
$\varphi$ defines a map on $S$ which is isometric (we just checked
it is well defined and is isometric as each $\varphi_k$ is), and so
it extends to a Hilbert module map $\varphi\colon X\to Y$, isometric
onto its image.

(ii). Assume now that $[X]=[Y]$. Working as in (i), now with the two
inequalities $[X]\leq [Y]$ and $[Y]\leq [X]$, and except maybe
passing to subsequences of the $X_i$'s and the $Y_j$'s, we obtain
maps $\widetilde{\varphi}_n\colon X_n\to Y_n$ and
$\widetilde{\psi}_n\colon Y_n\to X_{n+1}$ (isometries onto their
images)

We shall modify as in (i) the given maps. Put
$\varphi_1=\widetilde{\varphi}_1$. Suppose that $\varphi_1$,
$\psi_1$, $\varphi_2,\ldots, \psi_{n-1}$ and $\varphi_n$ have been
constructed, and let us construct $\psi_n$ and $\varphi_{n+1}$. Let
\[
F_n^X=\{x_j^{(i)}\mid 1\leq i,j\leq
n\}\cup\{\psi_{n-1}(y_j^{(i)})\mid 1\leq i,j\leq n-1\}\subseteq
X_n\,.
\]
Using Lemma~\ref{lem:key} again, we can find a unitary $u_{n+1}$ in
$\mK(X_{n+1})^{\widetilde{}}$ such that
\[
\Vert u_{n+1}\widetilde{\psi}_n\varphi_n(x)-x\Vert<2^{-n}\,,
\]
for all $x\in F_n^X$. Define $\psi_n=u_{n+1}\widetilde{\psi}_n$, so
that
\[
\Vert \psi_n\varphi_n(x)-x\Vert<2^{-n}\,,
\]
for $x\in F_n^X$. In a similar way, put
\[
F_n^Y=\{y_j^{(i)}\mid 1\leq i,j\leq
n\}\cup\{\varphi_n(x_j^{(i)})\mid 1\leq i,j\leq n\}\,,
\]
and we can find a unitary $v_{n+1}$ in $\mK(Y_{n+1})^{\widetilde{}}$
such that
\[
\Vert v_{n+1}\widetilde{\varphi}_{n+1}\psi_n(y)-y\Vert<2^{-n}
\]
for any $y\in F_n^Y$. Define now
$\varphi_{n+1}=v_{n+1}\widetilde{\varphi}_{n+1}$, so
\[
\Vert \varphi_{n+1}\psi_n(y)-y\Vert<2^{-n}
\]
for all $y\in F_n^Y$.

In particular, for $y_{j}^{(i)}$ and $1\leq i,j,\leq n-1$, we have
\[
\Vert\psi_n\varphi_n\psi_{n-1}(y_{j}^{(i)})-\psi_{n-1}(y_j^{(i)})
\Vert<2^{-n}\,,
\]
and
\[
\Vert \varphi_n\psi_{n-1}(y_{j}^{(i)})-y_j^{(i)} \Vert<2^{-n+1}\,.
\]
Therefore
\[
\Vert \psi_n(y_{j}^{(i)})-\psi_{n-1}(y_{j}^{(i)})\Vert\leq \Vert\psi_n(y_{j}^{(i)})-\psi_n\varphi_n\psi_{n-1}(y_{j}^{(i)}) \Vert+\Vert\psi_n\varphi_n\psi_{n-1}(y_{j}^{(i)})-\psi_{n-1}(y_{j}^{(i)}) \Vert<3\cdot 2^{-n}\,,
\]
whence we can define, much in the same way as in (i), a Hilbert
module map $\psi\colon Y\to X$.

Similarly, for any $1\leq i,j\leq n$, we have
\[
\Vert \varphi_n(x_{j}^{(i)})-\varphi_{n+1}(x_{j}^{(i)})\Vert <3\cdot
2^{-n}\,,
\]
which allows us to define a Hilbert module map $\varphi\colon X\to
Y$. Our construction ensures that $\psi$ and $\varphi$ will be
inverses for one another, and so $X\cong Y$ as desired.
\end{proof}

\subsection{The relationship between $\boldmath{\Cu}$ and
$\boldmath{\W(A)}$}

The purpose of this section is to show that the two semigroups we
have introduced, namely $\Cu$ and $\W(A)$, are closely tied up.

It can be shown that the assingment $A\mapsto \Cu$ is functorial.
Let us sketch roughly how this works.

If $X$ is a Hilbert $A$-module and $\varphi\colon A\to B$ is a
$^*$-homomorphism, then consider the algebraic tensor product
$X\widehat{\otimes}_A B$ (with $B$ viewed as an $A$-module via
$\varphi$). Define the following (possibly degenerate) inner product
on $X\widehat{\otimes}_A B$
\[
\langle\sum\xi_i\widehat{\otimes} b_i,\sum\xi_j'\widehat{\otimes}
b_j'\rangle=\sum_{i,j}b_i^*\langle\xi_i,\xi_j'\rangle b_j
\]
and put
\[
X\otimes_A B=\overline{\left(X\widehat{\otimes}_A B/\{z\mid \langle
z,z\rangle =0\}\right)}\,.
\]
One then has a map $\mathrm{Cu}(\varphi)\colon \Cu\to
\mathrm{Cu}(B)$, given by $\mathrm{Cu}(\varphi)([X])=[X\otimes_A
B]$, which actually belongs to the category $\mathrm{Cu}$.

\begin{proposition}
\label{Custable} Let $A$ be a {\rm C}$^*$-algebra and let $p$ be a
projection in $\mathcal{M}(A)$ such that $\overline{ApA}=A$. Then
\[
\mathrm{Cu}(pAp)\cong\Cu\,,
\]
via  $[X]\mapsto [X\otimes_{pAp} A]$.
\end{proposition}
\begin{proof}
The map just given is a map in the category $\mathrm{Cu}$ by
functoriality. The converse is given by $[X]\mapsto [Xp]$.

Note indeed that
\[
(X\otimes_{pAp} A)p=\overline{XpAp}\otimes Ap=X\otimes pAp=X\,,
\]
and also that
\[
Xp\otimes_{pAp}A\cong
\overline{XpA}=\overline{X\overline{ApA}}=\overline{XA}=X\,.
\]

Let us check that the converse map preserves compact containment. If
$X\CC Y$, then there is $a\in\mK(Y)$ with $a_{|X}=\mathrm{id}_{|X}$.
Since $\overline{ApA}=A$, we may write $a$ as a limit of sums of
elements of the form $\theta_{xpa,ypb}$ for $x$, $y\in Y$ and $a$,
$b\in A$. But now observe that:
\[
\theta_{xpa,ypb}=\theta_{xp, y(pba^*p)}\in\mathcal{K}_{pAp}(Yp)\,,
\]
whence it follows that $Xp\CC Yp$.
\end{proof}

\begin{corollary}
\label{CuStable} For a {\em C}$^*$-algebra $A$, we have
$\Cu\cong\mathrm{Cu}(A\otimes\mathbb{K})$.
\end{corollary}
\begin{proof}
This follows from Proposition \ref{Custable} by taking $e=1\otimes
e_{11}\in
\mathcal{M}(A)\otimes\mathbb{K}\subset\mathcal{M}(A\otimes\mathbb{K})$,
where $e_1$ is a rank one projection.
\end{proof}

\begin{lemma}
\label{compcont} If $a\in\mK(X)$ and $0<\epsilon<\epsilon'$, then
\[
\overline{(a-\epsilon')_+X}\CC\overline{(a-\epsilon)_+X}\,.
\]
\end{lemma}
\begin{proof}
Let $f(t)$ be a function which is $1$ if $t\geq\epsilon'$, is $0$ if
$t<\frac{\epsilon+\epsilon'}{2}$, and is linear otherwise. Then
$f(a)=(a-\epsilon)_+g(a)$, where $g$ is a continuous function, and
also $f(a)(a-\epsilon')_+=(a-\epsilon')_+$ by the functional
calculus. Thus $f(a)\in\mK(\overline{(a-\epsilon)_+X})$ and
$f(a)_{\vert\overline{(a-\epsilon')_+X}}=\mathrm{id}_{\vert\overline{(a-\epsilon')_+X}}$.
\end{proof}

\begin{theorem}
\label{CuandW} Let $A$ be a stable {\rm C}$^*$-algebra. Then
\[
\W(A)\cong \Cu\,.
\]
\end{theorem}
\begin{proof}
We shall prove that the assignment $\langle a\rangle\mapsto
[\overline{aA}]$ is an ordered semigroup isomorphism.

We first show that, if $a\precsim b$, then $[\overline{aA}]\leq
[\overline{bA}]$.

Indeed, there is a sequence $(c_n)$ such that $c_nbc_n^*\to a$. By
passing to a subsequence if necessary, we may choose a decreasing
sequence $\epsilon_n$ of positive numbers with limit zero such that,
for each $n$,
\[
\Vert a-c_nbc_n^*\Vert<\epsilon_n\,,
\]
so we have contractions $d_n$ with
$d_nc_nbc_n^*d_n^*=(a-\epsilon_n)_+$.

Now, by Lemma~\ref{equivalence} we have
\[
\overline{(a-\epsilon_n)_+A}=\overline{(d_nc_nb^{1/2})(d_nc_nb^{1/2})^*A}\cong
\overline{b^{1/2}c_n^*d_n^*d_nc_nb^{1/2}A}\subseteq\overline{bA}\,.
\]
Therefore, $[\overline{(a-\epsilon_n)_+A}]\leq [\overline{bA}]$ for
each $n$. On the other hand, it is clear that
$\overline{aA}=\overline{\cup_n\overline{(a-\epsilon)_+A}}$, so
Proposition~\ref{suprema1} yields
\[
[\overline{aA}]=\sup\limits_n [\overline{(a-\epsilon_n)_+A}]\,,
\]
and it follows from this that $[\overline{aA}]\leq[\overline{bA}]$.

Conversely, if $[\overline{aA}]\leq [\overline{bA}]$, it then
follows that $a\precsim b$. To see this, let $\epsilon>0$ and write
$f(a)=(a-\epsilon)_+$. By Lemma~\ref{compcont}, we have
$\overline{f(a)A}\CC \overline{aA}$.

By definition of $\leq$ in $\Cu$, we have an isomorphism $\psi\colon
\overline{f(a)A}\stackrel{\cong}{\to} X$, where $X$ is compactly
contained in $\overline{bA}$.

Let $x=\psi(f(a)^{\frac{1}{2}})$. We then have that (being $\psi$ an
isometry):
\[
f(a)=\langle f(a)^{\frac{1}{2}},f(a)^{\frac{1}{2}}\rangle=\langle
x,x\rangle=x^*x\,,
\]
and $\overline{xA}=\overline{xx^*A}=X$ (because $f(a)^{\frac{1}{2}}$
is a generator of $\overline{f(a)A}$).

Now, as $\overline{bA}=\overline{b^{1/2}A}$ and $x\in\overline{bA}$,
there is a sequence $(c_n)$ with $x=\lim\limits_n
b^{\frac{1}{2}}c_n$. Therefore, $f(a)=x^*x=\lim\limits_n c_n^*bc_n$.

This implies that $(a-\epsilon)_+\precsim b$ and as $\epsilon$ is
arbitrary, it follows that $a\precsim b$, as desired.

Next, let us show that an element $a\in (A\otimes M_n)_+$ is (Cuntz)
equivalent to an element of $A$.

Put, by stability, $A=B\otimes\mathbb{K}$. There is an isometry
$v\in\mathcal{M}(\mathbb{K}\otimes
M_n)=\mathbb{B}(\mathcal{H})\otimes M_n$ such that
$v(\mathbb{K}\otimes M_n)v^*=\mathbb{K}\otimes e$ (with $e$ a rank
one projection). (For example, one can construct this by splitting
$\mathcal{H}=\mathcal{H}_1\oplus\cdots\mathcal{H}_n$, pairwise
isomorphic -- and all isomorphic to $\mathcal{H}$, whence
$\mathbb{K}(\mathcal{H})\cong\mathbb{K}(\mathcal{H}_1^n)\cong
M_n(\mathbb{K}(\mathcal{H}))$).

Now $A\otimes M_n=B\otimes(\mathbb{K}\otimes M_n)$ so $1\otimes
v\in\mathcal{M}(B)\otimes\mathcal{M}(\mathbb{K}\otimes M_n)\subseteq
\mathcal{M}(B\otimes(\mathbb{K}\otimes M_n))$.

Then $a\sim (1\otimes v)a(1\otimes v)^*\in
(B\otimes\mathbb{K}\otimes e)_+\cong A_+$.

Finally, if $X$ is a countably generated module, then there is by
Kasparov's Theorem a module $Y$ such that $X\oplus Y\cong H_A$. And,
by stability, $H_A\cong A$. We get that every countably generated
Hilbert $A$-module is isomorphic to $\overline{aA}$ for some $a\in
A_+$.
\end{proof}

\subsection{Continuity with respect to inductive limits}

Before one can even ask whether a functor is continuous with respect
to inductive limits, one must verify that such limits exist in both
the domain and co-domain categories.

\begin{theorem}\label{limexist}
Inductive limits always exist in $\mathsf{Cu}$.
\end{theorem}

\begin{proof}
\noindent {\bf 1. Definition of the limit object.}  Let
\[
S_1 \stackrel{\gamma_1}{\longrightarrow} S_2
\stackrel{\gamma_2}{\longrightarrow} S_3
\stackrel{\gamma_3}{\longrightarrow} \cdots
\]
be an inductive sequence in the category $\mathsf{Cu}$.  As in the
setting of C$^*$-algebra inductive sequences, we define
\[
\gamma_{i,j} := \gamma_{j-1} \circ \gamma_{j-2} \circ \cdots \circ
\gamma_i.
\]
Set
\[
S^\circ = \{ (s_i)_{i \in \mathbb{N}} \ | \ s_i \in S_i, \ s_i \leq
s_{i+1} \},
\]
where we understand that the comparison of $s_i$ and $s_{i+1}$ takes
place in $S_{i+1}$, and is in fact comparison between $s_{i+1}$ and
the {\it image} of $s_i$ under $\gamma_i$.  We will frequently
suppress $\gamma_i$ (or $\gamma_{i,j}$) in this manner, to avoid
cumbersome notation. Define an addition operation $(t_i) + (s_i) =
(t_i + s_i)$ on $S^\circ$.  This makes sense because addition in
each $S_i$ is compatible with the order in $S_i$---see property {\bf
(O2)} above---so that $(s_i + t_i)$ is an increasing sequence.  Also
define the following relation:  $(s_i) \leq (t_i)$
if for any $i$ and any $s \in S_i$ such that $s \ll s_i$, we have
$\gamma_{i,j}(s) \ll t_j$ for all sufficiently large $j$.  We will
prove that $\leq$ is a pre-order compatible with addition, and
obtain an equivalence relation $\sim$ on $S^\circ$ by declaring that
$(s_i) \sim (t_i)$ if $(s_i) \leq (t_i)$ and $(t_i) \leq (s_i)$.
Finally, we will prove that the quotient $S := S^\circ / \sim$ with
its inherited order and addition operation is an object in
$\mathsf{Cu}$, and is the inductive limit of the system $(S_i,
\gamma_i)$.

\vspace{2mm} Let us first see why $\leq$ is a pre-order on
$S^\circ$.  If $s \ll s_i$, then $s \ll s_j$ for every $j \geq i$ by
{\bf (M4)}.  This shows that $\leq$ is reflexive.  The transitivity
of $\leq$ follows directly from its definition.

To check the compatibility of $\leq$ with addition, let $(s_i),
(\bar{s}_i), (t_i), (\bar{t}_i) \in S^\circ$ satisfy
\[
(s_i) \leq (\bar{s}_i) \ \ \mathrm{and}  \ \ (t_i) \leq (\bar{t}_i).
\]
If $s \ll s_i + t_i$ for some $i$, then we must show that  $s \ll
\bar{s}_j + \bar{t}_j$ for all $j$ sufficiently large.  By {\bf
(O4)} there are increasing sequences
\[
s_i^1 \leq s_i^2 \leq \cdots \ll s_i \ \ \mathrm{and} \ \ t_i^1 \leq
t_i^2 \leq \cdots \ll t_i
\]
in $S_i$ such that
\[
\sup_j s_i^j = s_i \ \ \mathrm{and} \ \ \sup_j t_i^j = t_i.
\]
By {\bf (O2)} we have
\[
s_i^1 + t_i^1 \leq s_i^2 + t_i^2 \leq \cdots \leq s_i + t_i,
\]
and by {\bf (O5)} we have $\sup_j s_i^j + t_i^j = s_i + t_i$.  We've
assumed that $s \ll s_i + t_i$, so by the definition of the relation
$\ll$ we have that $s \leq s_i^n + t_i^n$ for all $n$ sufficiently
large.  By our assumption that $(s_i) \leq (\bar{s}_i)$ and $(t_i)
\leq (\bar{t}_i)$, we have $s_i^n \ll \bar{s}_j$ and $t_i \ll
\bar{t}_j$ for all $j$ sufficiently large.  Finally, using {\bf
(O5)},
\[
s \leq s_i^n + t_i^n  \ll \bar{s}_j + \bar{t}_j, \ \forall j \
\mathrm{sufficiently \ large}
\]
as desired. It follows that $S$ is an Abelian semigroup satisfying
{\bf (O2)}, that is, compatibility of the order with addition.

\vspace{2mm} \noindent {\bf 2. $S$ is in $\mathsf{Cu}$.}  It is easy
to see that the equivalence class of the sequence $(0,0,0,\ldots)
\in S^{\circ}$ is a zero element for $S = S^\circ / \sim$. This
establishes {\bf (O1)}.  It is also straightforward to see that $0
\ll x$ for each $x \in S_i$, $i \in \mathbb{N}$, whence
$(0,0,0,\ldots) \leq y$ for every $y \in S$.  This establishes {\bf
(O6)}.

\vspace{2mm} To establish {\bf (O2)}--{\bf(O5)} for $S$, we need the
following intermediate result.

\vspace{2mm} {\bf Claim:} Each $(s_i) \in S^\circ$ is equivalent to
$(\bar{s}_i)$ where $\bar{s}_i \ll \bar{s}_{i+1}$ for each $i \in
\mathbb{N}$.

\vspace{1mm} {\it Proof of claim.}  Let $(s_i)$ be given, and find,
using {\bf (O4)}, sequences $(s^j_i)$ for each $i$ such that $s^j_i
\ll s^{j+1}_i$ and $\sup_j s^j_i = s_i$.  Using the same argument as
in Proposition \ref{suprema}, we see that, after suitably modifying
the sequences $(s^j_i)$, we may assume that $s^j_i\ll
s^{k+1}_{k+1}$, where $k=\text{max}\{i,j\}$. We strongly encourage
you to write this down.

Write $\bar{s}_i=s^i_i$ for all $i$. Clearly, we have that
$\bar{s}_i\ll \bar{s}_{i+1}$ and that $(\bar{s}_i)\le (s_i)$. To
show that $(s_i)\le (\bar{s}_i)$, take $s\in S_i$ such that $s\ll
s_i$. Then $s\ll s^j_i$ for some $j$, and since $s^j_i\ll
s^{k+1}_{k+1}$, where $k=\text{max}\{i,j\}$, we get that $s\ll
\bar{s}_n$ for every $n\ge k+1$, showing that $(s_i)\le
(\bar{s}_i)$.

\vspace{2mm} Let us now verify {\bf (O3)}---the property that every
increasing sequence in $S$ has a supremum.  Let $s^1 \leq s^2 \leq
s^3 \leq \cdots$ be such a sequence, where $s_j$ is represented by a
sequence $(s^j_i)_{i=1}^\infty$ with $s^j_i \in S_i$.  Using the
claim above, we may assume that $s_i^j \ll s_{i+1}^j$ for each $i$
and $j$.  Now using the fact that our sequence is increasing, we may
find a sequence $(n_j)_{j=1}^\infty$ of natural numbers, with
$n_1=1$, such that $n_j \geq j$ and
\[
s^j_{n_j} \gg s^k_l, \ \ 1 \leq k \leq j-1, \ \ 1 \leq l \leq
n_{j-1}.
\]
Define a sequence $d:=(d_i)$ as follows:  $d_i = s^1_1$ for each $i
< n_1$, and $d_i = s^j_{n_j}$ for each $ n_j \leq i < n_{j+1}$.  We
claim that $\sup s^j = d$.  Let us first see why $d$ is an upper
bound for our sequence. Fix $j$ and $r \ll s^j_i$.  Find $l \in
\mathbb{N}$ such that $l > j$ and $n_l > i$. It follows that $r \ll
s^l_{n_l}$ by construction, and the latter element is in fact an
element of the increasing sequence $(d_i)$.  This yields $s^j \leq
d$. Now suppose that $s^j \leq t := (t_i)$ for each $j \in
\mathbb{N}$.  Let $r \ll d_i$ for some $i$, so that $r \ll
s^k_{n_k}$ for some $k$.  Since $s^k \leq t$ we have $r \ll t_l$ for
all $l$ sufficiently large.  This gives $d \leq t$, so that $d$ is
the least upper bound of the increasing sequence $s^1 \leq s^2 \leq
s^3 \leq \cdots$.

\vspace{2mm} We now establish {\bf (O4)} for $S$--- for each $s \in
S$ the set
\[
s^{\ll} = \{ y \in S \ | \ y \ll s \}
\]
is upward directed with respect to both $\leq$ and $\ll$, and
contains a sequence $(r_n)$ such that $r_n \ll r_{n+1}$ for every $n
\in \mathbb{N}$ and $\sup_n r_n = s$. We know that an element $s \in
S$ can be represented by a rapidly increasing sequence $(s_i)$ with
$s_i \in S_i$.  It is straightforward to check that the sequence
\[
r_1:=(s_1,s_1,\ldots), \ r_2:=(s_1,s_2,s_2,\ldots), \
r_3:=(s_1,s_2,s_3,s_3,\ldots),\ldots
\]
is rapidly increasing in $S$, and that $s$ dominates each element of
the sequence.  To show that $s$ is in fact the supremum of this
sequence, let $t \in S$ have the property that
\[
t \geq (s_1,s_1,\ldots), \ (s_1,s_2,s_2,\ldots), \
(s_1,s_2,s_3,s_3,\ldots),\ldots;
\]
we must show that $s \leq t$. Let $(t_i)$ be an increasing sequence
representing $t$ (with $t_i \in S_i$).   For each $i$ we have $s_i
\ll s_{i+1}$ in $S_{i+1}$ and
$(s_1,\ldots,s_i,s_{i+1},s_{i+1},\ldots) \leq t$.  It follows from
the definition of $\leq$ that $s_i \ll t_j$ for all $j$ sufficiently
large, whence
\[
(s_1,s_2,s_3,\ldots) = s \leq (t_1,t_2,t_3,\ldots) = t.
\]
Now let $x,y \in s^{\ll}$, and let $r_n$ be as above.  It follows
that $r_n \geq x,y$ for all sufficiently large $n$, whence $r_{n+1}
\gg x,y$ for these same $n$.  This gives the required upward
directedness for $s^{\ll}$.

\vspace{2mm} Next let us verify {\bf (O5)} for $S$---the relation
$\ll$ and the operation of passing to the supremum of an increasing
sequence are all compatible with addition in $S$.  To this end, let
us collect some facts from our work above.
\begin{enumerate}
\item[(i)] Any element of $S$ can be represented by a rapidly increasing sequence $(s_i)$ with $s_i \in S_i$.
\item[(ii)] For any increasing sequence $(s^i)$ in $S$ with supremum $s$ there is a rapidly increasing sequence $(s_i)$
with $s_i \in S_i$, representing $s$, such that
\[
\bar{s}_i := (s_1,\ldots,s_{i-1},s_i,s_i,\ldots) \leq s^i \ \
\mathrm{and} \ \ \sup_i \bar{s}_i \leq \sup_i s^i;
\]
the sequence $\bar{s}_i$ is moreover rapidly increasing in $S$. In
fact, we have $\sup_i \bar{s}_i = s$ for any rapidly increasing
sequence $(s_i)$ representing $s$, where $s_i \in S_i$.
\end{enumerate}
{\bf Warning:}  From here on we use the ``bar'' notation {\it only}
in the sense of (ii) above.

\vspace{2mm} We are now ready to prove the compatibility of suprema
with addition.  Let $(s^i)$ and $(t^i)$ be increasing sequences in
$S$ with suprema $s$ and $t$, respectively.  Choose representing
sequences $(s_i)$ and $(t_i)$ for $s$ and $t$, respectively, as in
(i) and (ii) above.  It follows that $(s_i+t_i)$ is a representing
sequence for $s+t$ which is rapidly increasing and therefore
satisfies $\sup_i \overline{s_i+t_i} =s+t$.  It is also true that
$\overline{s_i + t_i} \leq s^i + t^i$.  To see this, let $r \ll
(\overline{s_i+t_i})_j$ for some $j$.  Using $\overline{s_i+t_i}=
(s_1+t_1,\ldots,s_i+t_i,s_i+t_i,\ldots)$ we see that $r \ll
s_i+t_i$.  Let $(s_j^i)$ and $(t_j^i)$ be the representing sequences
for $s^i$ and $t^i$ which were used to produce $(s_i)$ and $(t_i)$.
(In other words, $s_i = s_i^i$ and $t_i=t^i_i$.) Now $r \ll s_i+t_i
= s_i^i + t_i^i$, so $r \ll s_j^i + t_j^i$ for all sufficiently
large $j$ by the fact that $(s^i_j)$ and $(t^i_j)$ are rapidly
increasing in $j$.  This shows that $\overline{s_i + t_i} \leq s^i +
t^i$.  Now we calculate:
\[
s+t = \sup \overline{s_i+t_i} \leq \sup (s^i + t^i) \leq \sup s^i +
\sup t^i = s+t.
\]
This proves that
\[
\sup(s^i+t^i) = \sup s^i + \sup t^i,
\]
as desired.

\vspace{2mm} Let us now prove that the relation $\leq$ is compatible
with addition. Let $s^1 \leq t^1$ and $s^2 \leq t^2$ in $S$, and
choose rapidly increasing representative sequences $(s^1_i)$,
$(s^2_i)$, $(t^1_i)$, and $(t^2_i)$ for $s^1, s^2, t^1,$ and $t^2$,
respectively, where the $i^{\mathrm{th}}$ element of each sequence
belongs to $S_i$.  It is an easy exercise to show that one may
modify these sequences to arrange that $s^1_i \leq t^1_i$ and $s^2_i
\leq t^2_i$, and the possible expense of their being rapidly
increasing.  Now
\begin{eqnarray*}
s^1 + s^2 & = & \sup \bar{s}^1_i + \sup \bar{s}^2_i \\
& = & \sup \overline{(s^1_i + s^2_i)} \\
& \leq & \sup \overline{(t^1_i + t^2_i)} \\
& = & \sup \bar{t}^1_i + \sup \bar{t}^2_i \\
& = & t^1 + t^2,
\end{eqnarray*}
i.e., $s^1+s^2 \leq t^1 + t^2$, as desired.

\vspace{2mm} To complete the proof that {\bf (O5)} holds in $S$, we
must prove the compatibility of the relation $\ll$ with addition.

\vspace{2mm} Let $s^1 \ll t^1$ and $s^2 \ll t^2$ in $S$ be given,
and choose rapidly increasing representing sequences $(t^1_i)$ and
$(t^2_i)$ for $t^1$ and $t^2$, respectively. It follows that $s^1
\leq \bar{t}_i^1$ and $s^2 \leq \bar{t}_i^2$ for all $i$
sufficiently large (this uses (ii) above). Also, by our claim above,
$\bar{t}^1_i + \bar{t}^2_i \ll \bar{t}^1_{i+1} + \bar{t}^2_{i+1}$.
Thus, for these $i$,
\[
s^1 + s^2 \leq \bar{t}^1_i + \bar{t}^2_i  \ll \bar{t}^1_{i+1} +
\bar{t}^2_{i+1} \leq t^1 + t^2,
\]
i.e., $s^1+s^2 \ll t^1+t^2$, as desired.

\vspace{2mm} \noindent {\bf 3. $S$ is the limit of
$(S_i,\gamma_i)$.} To complete the proof of the theorem, we must
show that $S$ is the inductive limit of the sequence
$(S_i,\gamma_i)$ in the category $\mathsf{Cu}$.  We must show that
for every object $T$ in $\mathsf{Cu}$ and every sequence of maps
$\phi_i:S_i \to T$ satisfying $\phi_{i+1} \circ \gamma_i = \phi_i$,
there exists a unique map $\phi:S \to T$ such that $\phi_i = \phi
\circ \eta_i$, where $\eta_i:S_i \to S$ is given by
\[
s \mapsto (\underbrace{0,\ldots,0}_{i-1 \
\mathrm{times}},s,s,s,\ldots).
\]
(That $\eta_i$ preserves the zero element, the relation $\leq$, and
addition is plain;  that it preserves the relation $\ll$ follows
easily from the way the supremum of an increasing sequence in $S$ is
built. That the $\eta_i$ are compatible with the $\gamma_i$ follows
from the fact that for $j \geq i$,
\[
(\underbrace{0,\ldots,0}_{i-1 \ \mathrm{times}},s,s,s,\ldots) \sim
(\underbrace{0,\ldots,0}_{j-1 \ \mathrm{times}},s,s,s,\ldots)
\]
in $S$.)

\vspace{2mm} To define $\phi$, we view each $s \in S$ as being
represented by a rapidly increasing sequence $(s_i)$ with $i \in
\mathbb{N}$, and set $\phi(s) = \sup \phi_i(s_i)$.  If $(s_i^{'})$
is another rapidly increasing sequence representing $s$, then we may
find strictly increasing sequences $(n_i)$ and $(m_i)$ such that
\[
s_{n_i} \leq s_{m_{i}}^{'} \leq s_{n_{i+1}}, \ \forall i \in
\mathbb{N}.
\]
Since the maps $\phi_i$ preserve the order relation we conclude that
\[
\sup \phi_i(s_i) = \sup \phi_{n_i}(s_{n_i}) = \sup
\phi_{m_i}(s_{m_i}^{'}) = \sup \phi_i(s_i^{'}),
\]
whence $\phi$ does not depend on the choice of rapidly increasing
representing sequence.

\vspace{2mm} Let us check that $\phi \circ \eta_i = \phi_i$, that
$\phi$ is a morphism in the category $\mathsf{Cu}$, and that it is
unique with these properties.

\vspace{2mm} Given $s \in S_i$, we must show that $\phi(\eta_i(s)) =
\phi_i(s)$. To make this calculation we must first represent
\[
\eta_i(s) = (\underbrace{0,\ldots,0}_{i-1 \
\mathrm{times}},s,s,s,\ldots)
\]
as a rapidly increasing sequence $(r_i)$ with $r_i \in S_i$.  Choose
a rapidly increasing sequence $(t_j)$ in $S_i$ with supremum $s$,
and set
\begin{equation}\label{phii}
(r_1,r_2,r_3,\ldots) = (\underbrace{0,\ldots,0}_{i-1 \
\mathrm{times}},t_i,t_{i+1},t_{i+2},\ldots).
\end{equation}
Now
\begin{eqnarray*}
\phi(\eta_i(s)) & = & \sup_j \phi_j(r_j) \\
& = & \sup_{j \geq i} \phi_j \circ \gamma_{i,j}(t_j) \\
& = & \sup_{j \geq i} \phi_i(t_j) \\
& = & \phi_i \left( \sup_{j \geq i} t_j \right) \\
& = & \phi_i(s),
\end{eqnarray*}
as desired.

\vspace{2mm} We now verify properties {\bf (M1)}-{\bf (M4)} for
$\phi$, and show that $\phi$ belongs to $\mathsf{Cu}$ as a map.

\vspace{2mm} First, we show that $\phi$ preserves addition.  Let
$(r_i)$ and $(s_i)$ be two rapidly increasing sequences with
$r_i,s_i \in S_i$. Then,
\begin{eqnarray*}
\phi \left( (r_i) + (s_i) \right) & = & \phi \left( (r_i+s_i) \right) \\
& = & \sup \phi_i(r_i+s_i) \\
& = & \sup \left( \phi_i(r_i) + \phi_i(s_i) \right) \\
& = & \sup \phi_i(r_i) + \sup \phi_i(s_i) \\
& = & \phi \left((r_i) \right) + \phi \left((s_i)\right),
\end{eqnarray*}
as required.

\vspace{2mm} Property {\bf (M1)}---preservation of the zero
element---is trivial, just observe that $(0,0,0,\dots )$ is rapidly
increasing.

\vspace{2mm} For property {\bf (M2)}---preservation of $\leq$---let
$(r_i) \leq (s_i)$ be rapidly increasing sequences with $r_i,s_i \in
S_i$. Thus, for any $i$, we have $r_i \ll s_j$ for all $j$
sufficiently large, and so $\phi_j(\gamma_{i,j}(r_i)) \ll
\phi_j(s_j)$.  It follows that
\[
\phi\left( (r_i) \right) = \sup \phi_i(r_i) \leq \sup \phi_j(s_j) =
\phi \left( (s_j) \right),
\]
as required.

\vspace{2mm} For property {\bf (M3)}---preservation of suprema---let
$r^k$ be an increasing sequence in $S$ such that the representing
sequence $(r^k_i)$ of each $r^k$ is rapidly increasing, the sequence
$(r^k_i)_{i \in \mathbb{N}}$ is rapidly increasing, and for some
sequence $(n_i)$ of natural numbers the corresponding sequence
$(r^i_{n_i})$ is rapidly increasing and represents the supremum of
$r^1 \leq r^2 \leq \cdots$.  (We saw that this could be done
earlier, when proving the existence of suprema in $S$.) With this
preparation $\phi ( \sup r^k ) = \sup \phi_{n_i}(r^i_{n_i})$.  On
the other hand, $\sup \phi(r^k) = \sup_{k,i} \phi_i(r^k_i)$.  For
every $j$ such that $j>k$ and $n_{j} > i$ we have $\phi_i(r^k_i)
\leq \phi_{n_j}(r^j_{n_j})$, and for every $i \geq n_k$ we have
$\phi_{n_k}(r^k_{n_k}) \leq \phi_i(r^k_i)$. It follows that the two
suprema we are considering in $T$ are in fact the same, proving {\bf
(M3)}.

\vspace{2mm} For property {\bf (M4)}---preservation of $\ll$---let
$r= (r_i)$ and $s=(s_i)$ be rapidly increasing sequences with
$r_i,s_i \in S_i$ satisfying $(r_i) \ll (s_i)$.  We must show that
$\phi \left( (r_i) \right)\ll \phi\left( (s_i) \right)$ in $T$.  We
have $s = \sup \bar{s}_i$ and $r \ll s$, whence $r \leq \bar{s}_i$
for some $i$. It follows that $\phi(r) \leq \phi(\bar{s}_i)$.  Since
$s_i \ll s_{i+1}$ in $S_{i+1}$, their images under $\phi_{i+1}$ also
bear the same relationship.  These images are equal to
$\phi(\bar{s}_i)$ and $\phi(\bar{s}_{i+1})$, respectively, so we
have
\[
\phi(r) \leq \phi(\bar{s}_i) \ll \phi(\bar{s}_{i+1}) \leq \phi(s),
\]
as required.
\end{proof}

\subsection{Continuity of the functor $\mathsf{Cu}$}

\begin{theorem}\label{continuity}
The map which assigns to a C$^*$-algebra its Cuntz semigroup is a
functor from the category of C$^*$-algebras into $\mathsf{Cu}$,
which is moreover continuous with respect to inductive limits.
\end{theorem}

\begin{lemma}\label{epsiloncompact}
Let $a$ be a positive element in a C$^*$-algebra $A$, and let
$\epsilon>0$ be given.  It follows that $\langle (a-\epsilon)_+
\rangle \ll \langle a \rangle$.  We also have $\langle a \rangle =
\sup \langle (a-\epsilon_n)_+ \rangle$ for any decreasing sequence
$(\epsilon_n)$ converging to zero.
\end{lemma}

We leave the details of the proof of Lemma \ref{epsiloncompact} as
an exercise, but note that the first assertion is most easily proved
with the Hilbert module picture of the relation $\ll$.  The second
assertion follows from the fact that the supremum in question
dominates $\langle (a-\epsilon)_+ \rangle$ for any $\epsilon>0$, and
therefore dominates $\langle a \rangle$.

\begin{proof} (Theorem \ref{continuity}.)
\noindent {\bf 1. Induced morphisms.}  To discuss continuity of the
functor in question---and even to prove that it is a functor in the
first place---we must explain how a homomorphism between
C$^*$-algebras $\phi:A \to B$ induces a morphism
$\mathsf{Cu}(\phi):\mathsf{Cu}(A) \to \mathsf{Cu}(B)$ in the
category $\mathsf{Cu}$. To this end we will revert to our original
view of the Cuntz semigroup:  that it is a semigroup consisting of
equivalence classes of positive elements from $A \otimes
\mathcal{K}$, i.e. $\mathsf{Cu}(A)\cong \W(A\otimes \mathcal K)$,
see Corollary \ref{CuStable} and Theorem \ref{CuandW}. As a set map,
we define $\mathsf{Cu}(\phi)(\langle a \rangle) = \langle (\phi
\otimes \mathrm{id})(a) \rangle$ for each positive $a \in A \otimes
\mathcal{K}$.  We will assume from here on that $A$ is stable, and
replace $\phi \otimes \mathrm{id}$ with $\phi$. It is
straightforward to verify that $\mathsf{Cu}(\phi)$ is a semigroup
homomorphism (check this!). If $a \precsim b$, where $a,b \in A
\otimes \mathcal{K}$ are positive, then $\phi(a) \precsim \phi(b)$,
for if $(v_n)$ is a sequence drawn from $A \otimes \mathcal{K}$ with
the property that $v_n b v_n^* \to a$, then $\phi(v_n) \phi(b)
\phi(v_n)^* \to \phi(a)$.  It follows that $\mathsf{Cu}(\phi)$
preserves order, and so enjoys properties {\bf (M1)} and {\bf (M2)}
above.

\vspace{2mm} Let us show that $\mathsf{Cu}(\phi)$ preserves suprema
of increasing sequences, and so enjoys property {\bf (M3)}.  Let
$\langle a \rangle = \sup \langle a_n \rangle$, where $\langle a_n
\rangle$ is increasing.  Since $\mathsf{Cu}(\phi)$ preserves order,
we have that $\mathsf{Cu}(\phi)( \langle a \rangle) \geq \sup
\mathsf{Cu}(\phi)(\langle a_n \rangle)$.  For every $\epsilon>0$
there is a natural number $n(\epsilon)$ with the property that
$\langle a_n \rangle \geq \langle (a-\epsilon)_+ \rangle$ for every
$n \geq n(\epsilon)$ (see Lemma \ref{epsiloncompact} above). For
these same $n$ we have
\[
\mathsf{Cu}(\phi)(\langle a_n \rangle) \geq
\mathsf{Cu}(\phi)(\langle (a-\epsilon)_+ \rangle )= \langle (\phi(a)
- \epsilon)_+ \rangle.
\]
It follows that $\sup \mathsf{Cu}(\phi)(\langle a_n \rangle) \geq
\langle (\phi(a)-\epsilon)_+ \rangle$ for every $\epsilon > 0$,
whence
\[
\mathsf{Cu}(\phi)(\langle a \rangle) \geq \sup
\mathsf{Cu}(\phi)(\langle a_n \rangle) \geq
\mathsf{Cu}(\phi)(\langle a \rangle),
\]
as required.

\vspace{2mm} Now we show that $\mathsf{Cu}(\phi)$ preserves the
relation $\ll$. Suppose that $\langle a \rangle \ll \langle b
\rangle$ in $\mathsf{Cu}(A)$. It follows that $\langle a \rangle
\leq \langle (b-\epsilon)_+ \rangle$ for some $\epsilon > 0$ (this
uses the second assertion of Lemma \ref{epsiloncompact}).  Applying
$\mathsf{Cu}(\phi)$ yields
\[
\mathsf{Cu}(\phi)(\langle a \rangle) \leq \mathsf{Cu}(\phi)(\langle
(b-\epsilon)_+ \rangle) = \langle (\phi(b) - \epsilon)_+ \rangle \ll
\langle \phi(b) \rangle = \mathsf{Cu}(\phi)(\langle b \rangle),
\]
establishing {\bf (M4)}.

\vspace{2mm} Thus $\mathsf{Cu}(\phi)$ belongs, as a map, to the
category $\mathsf{Cu}$. It is also clear that the association $\phi
\mapsto \mathsf{Cu}(\phi)$ respects composition of maps, whence the
assignation described in the statement of the theorem is indeed a
functor.

\vspace{2mm} \noindent {\bf 2. Preservation of inductive limits.}
Next, we will prove that the functor $\mathsf{Cu}$ respects
sequential inductive limits.  Let $A_1
\stackrel{\phi_1}{\longrightarrow} A_2
\stackrel{\phi_2}{\longrightarrow} \cdots$ be an inductive sequence
of C$^*$-algebras with limit $A$, and let $\phi_{i\infty}:A_i \to A$
denote the canonical homomorphism.

\vspace{2mm} {\bf Warning:}  We will prove the theorem under the
additional assumption that the $\phi_i$ are all injective;  the
proof of the general case is similar, but introduces some
technicalities which obscure the general idea of the proof.

\vspace{2mm} We will first show that for any $\langle a \rangle \in
\mathsf{Cu}(A)$ there is a sequence $(a_i)_{i \in \mathbb{N}}$ with
the following properties:
\begin{enumerate}
\item[(i)] $a_i$ is a positive element of $A_i \otimes \mathcal{K}$;
\item[(ii)] $(\phi_i \otimes \mathrm{id}_{\mathcal{K}})(a_i) \precsim a_{i+1}$ in $\mathsf{Cu}(A_{i+1})$;
\item[(iii)] $\langle a \rangle = \sup \langle (\phi_{i\infty} \otimes \mathrm{id}_{\mathcal{K}})(a_i) \rangle$.
\end{enumerate}
Let $a \in A \otimes \mathcal{K}$ be positive.  Since $A \otimes
\mathcal{K}$ is the limit of the inductive sequence $(A_i \otimes
\mathcal{K}, \phi_i \otimes \mathrm{id}_{\mathcal{K}})$, we assume
henceforth that $A$ and all of the $A_i$ are stable, and use
$\phi_i$ and $\phi_{i\infty}$ in place of $\phi_i \otimes
\mathrm{id}_{\mathcal{K}}$ and $\phi_{i\infty} \otimes
\mathrm{id}_{\mathcal{K}}$, respectively. We can find $a_i \in A_i$
such that $\phi_{i\infty}(a_i) \to a$.  By passing to the positive
part of $(a_i+a_i^*)/2$, we may assume that each $a_i$ is positive.
It remains to arrange for (ii) and (iii) to hold.

\vspace{2mm} Set $\epsilon_i = \Vert \phi_{i\infty}(a_i) - a \Vert$.
Find an increasing sequence of natural numbers
$(i_k)_{k=1}^{\infty}$ with the property that $3 \epsilon_{i_{k+1}}
< \epsilon_{i_k}$ and $i_1=1$.  Now, by the injectivity of the
$\phi_i$,
\begin{eqnarray*}
\Vert \phi_{i_{k}i_{k+1}}(a_{i_k}) -
(a_{i_{k+1}}-2\epsilon_{i_{k+1}})_+ \Vert & = & \Vert \phi_{i_{k}
\infty}(a_{i_k}) -
\phi_{i_{k+1}\infty}((a_{i_{k+1}}-2\epsilon_{i_{k+1}})_+) \Vert \\
& \le & \Vert \phi_{i_{k} \infty}(a_{i_k})-a \Vert + \Vert a-\phi_{i_{k+1}\infty}(a_{i_{k+1}}) \Vert + 2\epsilon_{i_{k+1}}\\
& = & \epsilon_{i_k} + 3 \epsilon_{i_{k+1}} < 2 \epsilon_{i_k}.
\end{eqnarray*}
It follows from Theorem \ref{Kirchberg-Rordam} that
$(\phi_{i_{k}i_{k+1}}(a_{i_k})-2\epsilon_{i_k})_+ \precsim
(a_{i_{k+1}}-2\epsilon_{i_{k+1}})_+$. In fact note that since the
inequality above is strict we even have $\langle
(\phi_{i_{k}i_{k+1}}(a_{i_k})-2\epsilon_{i_k})_+ \rangle \ll \langle
(a_{i_{k+1}}-2\epsilon_{i_{k+1}})_+\rangle $.  For $j =
\{i_k,\ldots,i_{k+1}-1\}$, replace $a_j$ with $\phi_{i_k j}(
(a_{i_k}-2\epsilon_{i_k})_+)$. With this modification we have $a_j
\precsim a_{j+1}$ (and even $\langle \phi_j(a_j)\rangle \ll \langle
a_{j+1}\rangle $ if we slightly perturb the steps where equality
holds). We also have $\phi_{j \infty}(a_j) \precsim a$ and
$\phi_{j\infty}(a_j) \to a$ as $j \to \infty$.  Let $b \in A$ be a
positive element such that $\langle b \rangle = \sup \langle \phi_{j
\infty}(a_j) \rangle$. Given $\epsilon > 0$, there is some $j_0$
such that for all $j \geq j_0$, $\Vert a - \phi_{j \infty}(a_j)
\Vert < \epsilon$. In particular, using again Theorem
\ref{Kirchberg-Rordam}, we get $(a-\epsilon)_+ \precsim \phi_{j
\infty}(a_j) \precsim b$. Since $\epsilon$ was arbitrary, we
conclude that $a \precsim b$.  On the other hand, $\phi_{j
\infty}(a_j) \precsim a$ for each $j$ by construction, so $a = \sup
\langle \phi_{j \infty}(a_j) \rangle$, as desired.

\vspace{2mm} By functoriality we have a sequence
\[
\mathsf{Cu}(A_1) \stackrel{\mathsf{Cu}(\phi_1)}{\longrightarrow}
\mathsf{Cu}(A_2) \stackrel{\mathsf{Cu}(\phi_2)}{\longrightarrow}
\cdots \mathsf{Cu}(A).
\]
Let us show that $\mathsf{Cu}(A)$ is indeed the inductive limit of
the sequence $(\mathsf{Cu}(A_i),\mathsf{Cu}(\phi_i))$ in the
category $\mathsf{Cu}$.  By the construction of $\lim_{i \to \infty}
(\mathsf{Cu}(A_i),\mathsf{Cu}(\phi_i))$ in Theorem \ref{limexist}
and what we have proved above, it will suffice to show that if $x_1
\ll x_2 \ll \cdots$ and $y_1 \ll y_2 \ll \cdots$ with $x_i, y_i \in
\mathsf{Cu}(A_i)$, then
\[
\sup \mathsf{Cu}(\phi_{i \infty})(x_i) \leq \sup \mathsf{Cu}(\phi_{i
\infty})(y_i)
\]
if and only if whenever $z \ll x_i$ in $\mathsf{Cu}(A_i)$ for some
$i$ then $z \ll y_j$ in $\mathsf{Cu}(A_j)$ for some $j \geq i$. This
will establish an isomorphism from $\lim_{i \to \infty}
(\mathsf{Cu}(A_i),\mathsf{Cu}(\phi_i))$ onto a sub ordered semigroup
of $\mathsf{Cu}(A)$, and this subsemigroup was shown above to be all
of $\mathsf{Cu}(A)$.

\vspace{2mm} Suppose that $\sup x_i \leq \sup y_j$. We may find a
sequence of positive elements $(a_i)$ satisfying $a_i \in A_i$ and
$\langle a_i \rangle = x_i$, and a similar sequence $(b_j)$
corresponding to $y_j$. Suppose that $\sup x_i \leq \sup y_i$, and
that $\langle d \rangle =z \ll x_i = \langle a_i \rangle$ in
$\mathsf{Cu}(A_i)$ for some $i$. Since $\mathsf{Cu}(\phi_{i
\infty})$ preserves the relations $\leq$ and $\ll$, we have
\[
x_i \ll x_{i+1} \leq \sup x_i \leq \sup y_j
\]
in $\mathsf{Cu}(A)$.  It follows that $x_i \leq y_j$ for some $j
\geq i$ (equivalently, $\langle \phi_{i \infty}(a_i) \rangle \leq
\langle \phi_{j \infty}(b_j) \rangle$).  Since $z = \langle d
\rangle \ll \langle a_i \rangle$, there is a $\gamma>0$ such that $d
\precsim (a_i-2\gamma)_+$.  Also, for any $k \geq i$, we have
$\phi_{ik}(d) \precsim \phi_{ik}((a_i-2\gamma)_+) = (\phi_{ik}(a_i)
- 2\gamma)_+$.  To prove that $\langle d \rangle =z \ll y_j$ in
$\mathsf{Cu}(A_j)$ for all $j$ sufficiently large, it will suffice
to prove that for some $k \geq i$, $(\phi_{ik}(a_i) - \gamma)_+
\precsim y_k = \langle b_k \rangle$, for then
\[
z \leq \langle (\phi_{ik}(a_i) - 2\gamma)_+ \rangle \ll \langle
(\phi_{ik}(a_i)-\gamma)_+ \rangle \leq y_k
\]
in $\mathsf{Cu}(A_k)$. Since $\langle \phi_{i \infty}(a_i) \rangle
\leq \langle \phi_{j \infty}(b_j) \rangle$, and since
$\cup_{j=1}^\infty \phi_{j \infty}(A_j)$ is dense in $A$, we may
find a sequence $(v_k)$ such that $v_k \in A_k$ and
\[
\Vert \phi_{i \infty}(a_i) - \phi_{k \infty}(v_k) \phi_{j
\infty}(b_j) \phi_{k \infty}(v_k)^* \Vert \stackrel{k \to
\infty}{\longrightarrow} 0.
\]
By the injectivity of the $\phi_i$, $\phi_{k \infty}$ is an isometry
for every $k$, so for $k$ sufficiently large we have
\[
\Vert \phi_{ik}(a_i) - v_k \phi_{jk}(b_j) v_k^* \Vert < \gamma.
\]
It follows that $(\phi_{ik}(a_i)-\gamma)_+ \precsim v_k
\phi_{jk}(b_j) v_k^* \precsim \phi_{jk}(b_j) \precsim b_k$, as
required.

\vspace{2mm} Now suppose, conversely, that whenever $z \ll x_i$ in
$\mathsf{Cu}(A_i)$, then $z \ll y_j$ in $\mathsf{Cu}(A_j)$ for all
$j$ sufficiently large.  Let us show that $\sup x_j \leq \sup y_j$.
We must show that $x_i \leq \sup y_j$ in $\mathsf{Cu}(A)$ for every
$i$.  We know that $x_i$ is the supremum of a rapidly increasing
sequence $z_n$ in $\mathsf{Cu}(A_i)$.  It follows from our
hypothesis that for every $n$ there is a $j \geq i$ such that $z_n
\ll y_j$ in $\mathsf{Cu}(A_j)$.  By functoriality we also have $z_n
\ll y_j$ in $\mathsf{Cu}(A)$.  It follows that $\sup y_j$ dominates
each $z_n$ in $\mathsf{Cu}(A)$, and so dominates their supremum,
namely, $x_i$.

\end{proof}

\subsection{Exactness of $\mathsf{Cu}$}

In this section we will examine the relationship between the functor
$\mathsf{Cu}$ and short exact sequences
\[
0 \longrightarrow I \stackrel{\iota}{\longrightarrow} A
\stackrel{\pi}{\longrightarrow} B \longrightarrow 0.
\]
Specifically, we will see (but not prove in full detail) that
$\mathsf{Cu}$ is exact, i.e., that the sequence
\[
0 \longrightarrow \mathsf{Cu}(I)
\stackrel{\mathsf{Cu}(\iota)}{\longrightarrow} \mathsf{Cu}(A)
\stackrel{\mathsf{Cu}(\pi)}{\longrightarrow} \mathsf{Cu}(B)
\longrightarrow 0.
\]
is exact.  The reader is referred to \cite{crs} for full proofs.

\vspace{2mm} Let $A$ be a C$^*$-algebra and $I \subseteq A$ a
$\sigma$ unital ideal.  If $M$ is a countably generated right
Hilbert module over $A$, then $MI$ is a countably generated right
Hilbert module over $I$.  Suppose that $N$ is another countably
generated Hilbert module over $A$, and that $[M] \leq [N]$ in
$\mathsf{Cu}(A)$.  It follows that $[MI] \leq [NI]$.  Indeed,
suppose that $F$ is a compactly contained submodule of $MI$.  Then
$F$ is isomorphic to a compactly contained submodule $F^{'}$ of $N$.
Since $F = FI$, we must have $F^{'}I = F^{'}$, so that $F^{'}
\subseteq NI$.  Thus, $[F]=[F^{'}]\leq[NI]$.  Taking the supremum
over all such $F$, we get $[MI] \leq [NI]$.  In particular, if $M$
and $N$ are Cuntz equivalent, then so are $MI$ and $NI$.  Thus
justifies the notation $[MI] := [M]I$.  The map $[M] \mapsto [M]I$
is order preserving, and, since $(M \oplus N)I = MI \oplus NI$, it
is also additive. Notice that $M$ is a Hilbert $I$-module if and
only if $[M]I=[M]$.

\vspace{2mm} In the sequence of Cuntz semigroups above we have
\[
\mathsf{Cu}(\iota)([M]) = [M] \ \ \mathrm{and} \ \
\mathsf{Cu}(\pi)([M]) = [M/MI].
\]
Our exactness results will follow from the next theorem.

\begin{theorem}[Ciuperca-Robert-Santiago \cite{crs}]\label{crs1}
Let $I$ be a $\sigma$ unital, closed, two-sided ideal of the
C$^*$-algebra $A$ and let $\pi:A \to A/I$ denote the quotient
homomorphism.  Let $M$ and $N$ be countably generated right Hilbert
C$^*$-modules over $A$.  It follows that $\mathsf{Cu}(\pi)([M]) \leq
\mathsf{Cu}(\pi)([N])$ if and only if
\[
[M] + [N]I \leq [N] + [M]I.
\]
\end{theorem}

From the theorem we see that $\mathsf{Cu}(\pi)([M]) =
\mathsf{Cu}(\pi)([N])$ if and only if $[M]+[N]I = [N]+[M]I$.  Adding
$[H_I]$ to both sides and using Kasparov's Stabilization Theorem we
get
\[
\mathsf{Cu}(\pi)([M]) = \mathsf{Cu}(\pi)([N]) \Leftrightarrow [M]
+[H_I] = [N]+[H_I].
\]
Alternatively, one can say that $M/MI$ and $N/NI$ are Cuntz
equivalent as $A/I$ modules if and only if $M \oplus H_I$ and $N
\oplus H_I$ are Cuntz equivalent as $A$ modules.

\begin{corollary}
The map $\mathsf{Cu}(\pi)$, restricted to $\mathsf{Cu}(A) + [H_I]$,
is an isomorphism onto $\mathsf{Cu}(A/I)$.
\end{corollary}

\begin{proof}
Injectivity follows from the discussion above (it is a consequence
of Theorem \ref{crs1}).  $\mathsf{Cu}(\pi)$ is surjective since
every $A/I$ module can be embedded in $H_{A/I}$, and then have its
pre-image taken by the quotient map $H_A \to H_{A/I}$.
$\mathsf{Cu}(\pi)$ is also surjective when restricted to
$\mathsf{Cu}(A) + [H_I]$, since adding $H_I$ does not change the
image in $\mathsf{Cu}(A/I)$.
\end{proof}

Collecting the results above, it is straightforward to check that
\[
0 \longrightarrow \mathsf{Cu}(I)
\stackrel{\mathsf{Cu}(\iota)}{\longrightarrow} \mathsf{Cu}(A)
\stackrel{\mathsf{Cu}(\pi)}{\longrightarrow} \mathsf{Cu}(B)
\longrightarrow 0
\]
is exact.

\vskip2cm


\section{Classification of C$^*$-algebras}

\vskip1cm
\subsection{Introduction} The purpose of the current section is to provide an explicit link between the Elliott invariant (roughly consisting of the K-groups and traces) and the Cuntz semigroup. This is done under abstract hypotheses that are satisfied by a wide class of algebras. This includes the (agreeably) largest class for which the Elliott conjecture in its classical form can be expected to hold: those C$^*$-algebras that absorb the Jiang-Su algebra tensorially. The connection is given by proving representation theorems for the Cuntz semigroup, both for unital and stable algebras (discovered in the papers \cite{bpt}, \cite{bt}). Some of the proofs we recover have different angles than in their original format. The rest of the section, largely expository, focuses on three regularity properties that a C$^*$-algebra may enjoy that have become intimately related to the classification programme. The connections among them are mentioned and sketches of proofs in some instances are given. 
We close by relating the Cuntz semigroup to the classification results and gathering some evidence in favour of its potential use in the future.

\subsection{The Elliott Conjecture}

The Elliott conjecture for C$^*$-algebras can be thought of, at the heuristic level, the assertion that separable and nuclear C$^*$-algebras can be classified (up to $^*$-isomorphism) by means of $\mathrm{K}$-theoretic invariants. At the more down-to-earth level, it consists of a collection of concrete conjectures, where the invariant used depends on the class of algebras in question.

For example, for stable Kirchberg algebras (simple, nuclear, purely infinite algebras that satisfy the Universal Coefficients Theorem), the correct invariant is the graded Abelian
group $\mathrm{K}_0 \oplus \mathrm{K}_1$ (\cite{K}, \cite{P}). In the unital, stably finite, separable, and
nuclear case, consider the invariant
\[
\mathrm{Ell}(A) := \left(
(\mathrm{K}_0(A),\mathrm{K}_0(A)^+,[1_A]),\mathrm{K}_1(A),\mathrm{T}(A),
r_A \right)\,,
\]
where we consider topological $\mathrm{K}$-theory, $\mathrm{T}(A)$ is the Choquet simplex of tracial states,
and $r_A: \mathrm{T}(A) \times \mathrm{K}_0(A) \to
\mathbb{R}$ is the pairing between $\mathrm{K}_0$ and traces given by evaluating a trace at a $\mathrm{K}_0$-class. This is known as the Elliott invariant, and has been very
successful in confirming Elliott's conjecture for simple algebras.

In its most general form, the Elliott conjecture may be stated as follows:
\begin{blank}
[Elliott, c. 1989]\label{ec}
There is a $\mathrm{K}$-theoretic functor $F$ from the category of separable and nuclear
C$^*$-algebras such that if $A$ and $B$ are separable and nuclear, and there is an
isomorphism
\[
\phi\colon F(A) \to F(B),
\]
then there is a $^*$-isomorphism
\[
\Phi\colon A \to B
\]
such that $F(\Phi) = \phi$.
\end{blank}
We will let (EC) denote the conjecture above with the Elliott invariant $\mathrm{Ell}(\bullet)$
substituted for $F(\bullet)$, and with the class of algebras under consideration
restricted to those which are simple, unital, and stably finite.  (EC) has been shown
to hold in many situations, but counterexamples have also been found recently. The literature originated by the classification programme would make any attempt to collect the results an almost impossible task, although we will review some of them later on, in Section 6. As an authoritative source for its treatment of the subject, we still keep R\o rdam's monograph (\cite{R3}) on our bookshelf.



We shall also consider the following:

\begin{blank}[WEC]
Let $A$ and $B$ be simple, separable, unital, nuclear, and stably finite
C$^*$-algebras.  If there is an isomorphism
\[
\phi\colon \left(\W(A), \langle 1_A \rangle, \mathrm{Ell}(A)\right) \to
\left(\W(B), \langle 1_B \rangle, \mathrm{Ell}(B)\right),
\]
then there is a $^*$-isomorphism $\Phi\colon A \to B$ which induces $\phi$.
\end{blank}
\noindent
In contrast to the Elliott Conjecture mentioned above, there appears that no known counterexamples exist to (WEC).  But asking for the Cuntz semigroup as part of the invariant seems
strong indeed, given its sensitivity and the fact that (EC) alone is so often true. We will in the sequel see that (WEC) and (EC) are reconciled upon restriction to the largest class of C$^*$-algebras for which
(EC) may be expected to hold.  In this light, (WEC) appears as an appropriate
specification of the Elliott conjecture for simple, separable, unital, nuclear,
and stably finite C$^*$-algebras.



\subsection{Representation Theorems for the Cuntz semigroup: the
unital case}

We continue to assume here that all
C$^*$-algebras are separable, and in this section they will moreover be unital and exact (the latter assumption is quite standard and allows us to consider traces rather than quasi-traces, see \cite{Ha}). The results of this section are mainly taken from \cite{bpt}.

\subsubsection{An order-embedding}

We begin with some notation.  For a compact convex set $K$, denote
by $\aff(K)^+$ the semigroup of all positive, affine, continuous,
and real-valued functions on $K$; $\laff (K)^+ \supseteq
\aff(K)^+$ is the subsemigroup of lower semicontinuous functions,
and $\laff_b(K)^+ \subseteq \laff (K)^+$ is the subsemigroup
consisting of those functions which are bounded above.  The use of
an additional ``+'' superscript (e.g., $\aff(K)^{++}$) indicates
that we are considering only strictly positive functions. Unless otherwise noted, the order on
these semigroups will be pointwise. Thus $\aff (K)^+$ is algebraically
ordered with this ordering, but $\laff (K)^+$, in general, is not
(unless $K$ is, for example, finite dimensional).

\begin{definition}
\label{orderembedding}
We say that a homomorphism $\varphi\colon M\to N$ between two partially ordered semigroups $M$ and $N$ is an \emph{order-embedding} provided that $\varphi(x)\leq \varphi(y)$
if and only if $x\leq y$. A surjective order-embedding will be
called an \emph{order-isomorphism}.
\end{definition}

The object of study in this subsection and also the next one is the following semigroup:

\begin{definition}\label{wtilde}
Let $A$ be a unital {\rm C}$^*$-algebra. Define a semigroup structure on the set
\[
\widetilde{\W}(A):= \V(A) \sqcup \laff_b(\mathrm{T}(A))^{++}
\]
by extending the natural semigroup operations in both $\V(A)$ and $\laff(\mathrm{T}(A))^{++}$, and setting $[p]+f=\widehat{p}+f$, where $\widehat{p}(\tau)=\tau(p)$.
Define an order $\leq$ on $\widetilde{\W}(A)$ such that:
\begin{enumerate}[{\rm (i)}]
\item $\leq$ agrees with the usual order on $\V(A)\,;$
\item $f \leq g$ for $f$, $g$ in $\laff(\mathrm{T}(A))^{++}$ if and only if
\[
f(\tau) \leq g(\tau) \ \text{for all } \tau \in \mathrm{T}(A)\,;
\]
\item $f \leq [p]$ for
$[p] \in \V(A)$ and $f$ in $\laff (\mathrm{T}(A))^{++}$ if and only if
\[
f(\tau) \leq \tau(p) \ \text{for all } \tau \in \mathrm{T}(A)\,;
\]
\item $[p] \leq f$ for $f$, $[p]$ as in (iii) whenever
\[
\tau(p) < f(\tau) \ \text{for all } \tau \in \mathrm{T}(A)\,.
\]
\end{enumerate}
\end{definition}

In the results below, and also in the next section, we shall use repeatedly the fact that $A$ has strict comparison of positive elements. Taking into account that we are after a representation of the Cuntz semigroup that has the flavour of representing on the state space, this ought to be no surprise -- after all, this is what strict comparison tells us: we can recover order from a special type of states.

This condition is equivalent to a condition on the semigroup $\W(A)$, known as \emph{almost unperforation}.
\begin{definition}
A positively, partially ordered semigroup $M$ is almost unperforated if $(n+1)x\leq ny$ implies $x\leq y$.
\end{definition}

This has been proved by M. R\o rdam (in \cite{Rfunct}):

\begin{lemma}
Let $M$ be a partially ordered abelian semigroup which is almost unperforated. If $t$ and $t'\in M$ and $t'$ is an order-unit, and if $d(t')<d(t)$ for all states on $M$, then $t'\leq t$.
\end{lemma}

Assuming the lemma, the result below follows:

\begin{lemma}
If $A$ is unital and simple, then $\W(A)$ is almost unperforated if and only if $A$ has strict comparison of positive elements.
\end{lemma}
\begin{proof}
It is clear that strict comparison implies almost unperforation. For the converse, suppose that $d(a)<d(b)$ for all $d\in\mathrm{LDF}(A)$ (for $b\neq 0$). Let $d$ be any dimension function (i.e. a state on $\W(A)$), and construct a lower semicontinuous dimension function $\overline{d}$ as in Proposition \ref{rordam} (i.e. $\overline{d}(\langle a\rangle)=\sup_{\epsilon>0}d(\langle (a-\epsilon)_+\rangle)$. Then, for $\epsilon>0$,
\[
d(\langle (a-\epsilon)_+\rangle)\leq\overline{d}(\langle a\rangle)<\overline{d}(\langle b\rangle)\leq d(\langle b\rangle)\,.
\]
Now use that $A$ is simple and unital to conclude that $\langle b\rangle$ is an order unit for $\W(A)$, whence the assumption and the previous lemma yield $(a-\epsilon)_+\precsim b$. Since $\epsilon$ is arbitrary the result follows.
\end{proof}

Recall that any trace $\tau$ defines a lower semicontinuous function by $\dt(a)=\lim\limits_{n\to\infty} \tau(a^{1/n})$.

\begin{lemma}
\label{comparisonlemma} Let $A$ be a {\rm C}$^*$-algebra, and let
$a\in A_{+}$. For any faithful trace $\tau$, and
$\epsilon<\delta$ where $\epsilon, \delta\in\sigma(a)$, we have that $\dt ((a-\delta)_+)<\dt
((a-\epsilon)_+)$.
\end{lemma}

\begin{proof}
Since $(a-\epsilon)_+$ and $(a-\delta)_+$ belong to the
C$^*$-algebra $C^*(a)$ generated by $a$, we may assume that
$A=C^*(a)$ (which is commutative). Then $\tau$, being
a positive functional, corresponds to a probability measure
$\mu_\tau$ defined on $\sigma (a)$ and by~\cite[Proposition
I.2.1]{bh} we have $\dt (b)=\mu_\tau(\mathrm{Coz}(b))$, where
$\mathrm{Coz}(b)$ is the cozero set of a function $b$ in $C^*(a)$
(using the functional calculus).

Now write
$\mathrm{Coz}((a-\delta)_+)=U_{\delta,\epsilon}\sqcup\mathrm{Coz}((a-\epsilon)_+)$,
where $\sqcup$ stands for disjoint set union and
$U_{\delta,\epsilon}=\{t\in (\delta,\epsilon]\mid
(a-\delta)_+(t)>0\}$. Since $\epsilon\in U_{\delta,\epsilon}$ and
there is a non-zero $b$ in $C^*(a)$ such that
$\mathrm{Coz}(b)\subseteq U_{\delta,\epsilon}$, we have (using the
faithfulness of $\tau$) that $\mu_\tau(U_{\delta,\epsilon})\geq
\mu_\tau(\mathrm{Coz}(b))=\dt(b)>0$.

Finally,
\[
\dt((a-\delta)_+)-\dt((a-\epsilon)_+)=\mu_\tau(\mathrm{Coz}((a-\delta)_+))-\mu_\tau(\mathrm{Coz}((a-\epsilon)_+))=\mu_\tau(U_{\delta,\epsilon})>0\,.
\]
\end{proof}

Recall that, if $A$ is a C$^*$-algebra, we denote by $A_{++}$ the set of purely positive elements, that is, those positive elements that are not (Cuntz) equivalent to a projection. We also denote by $\W(A)_+$ the subset of $\W(A)$ consisting of those classes that are not classes of projections. We know that if $A$ is simple and stably finite or of stable rank one, then $\W(A)_+$ is a subsemigroup, and $\W(A)=\V(A)\sqcup\W(A)_+$.

\begin{proposition}\label{pureposcomparison}
Let $A$ be a simple $C^*$-algebra with strict comparison of
positive elements. Let $a \in A_{++}$ and $b \in A_+$ satisfy
$\dt(a) \leq \dt(b)$ for every $\tau \in \mathrm{QT}(A)$. Then, $a
\precsim b$.
\end{proposition}

\begin{proof}
If $A$ has no trace, then it is purely infinite and the
conclusion follows from~\cite[Proposition 5.4]{KR}.

Suppose that $\mathrm{T}(A)$ is nonempty.  Since $A$ is simple, each trace is
faithful. Since $a \in A_{++}$, we have that
$a\neq 0$ and we know there is a sequence sequence
$\epsilon_n$ of positive reals in $\sigma(a)$ strictly decreasing to zero.
We also know by~\cite[Section 6]{blsur} (see
also~\cite[Proposition 2.6]{KR}) that the set $\{x\in A_+\mid
x\precsim b\}$ is closed, and since $(a-\epsilon_n)_+\to a$ in
norm it suffices to prove that $(a-\epsilon_n)_+ \precsim b$ for
every $n \in \mathbb{N}$.

Let $\tau\in \mathrm{T}(A)$ be given, and apply Lemma
\ref{comparisonlemma} with $\epsilon = 0$ and $\delta =
\epsilon_n$ to see that
\[
\dt((a-\epsilon_n)_+)< \dt(a) \leq \dt(b).
\]
Using strict comparison on $A$ we conclude that
$(a-\epsilon_n)_+\precsim b$ for all $n$, as desired.
\end{proof}

\begin{proposition}\label{projpureposcomp}
Let $A$ be a simple {\rm C}$^*$-algebra with strict comparison of positive elements. Let $p$ be a projection in $A$, and
let $a \in A_{++}$.  Then, $\langle p \rangle \leq \langle a \rangle$ if and only if
$\dt(p) < \dt(a)$ for each $\tau \in \mathrm{QT}(A)$.
\end{proposition}

\begin{proof}
In light of Proposition~\ref{pureposcomparison} it will suffice to prove that if
$\dt(a)\leq \dt(p)$ for some $\tau \in \mathrm{QT}(A)$, then $p$ cannot be subequivalent
to $a$.  Suppose such a $\tau$ exists.  Let $\epsilon>0$ be given. By~\cite[Proposition 2.4]{Rfunct} there exists a
$\delta > 0$ such that
\[
(p-\epsilon)_+ \precsim (a-\delta)_+.
\]
This implies that
\[
\dt((p-\epsilon)_+) \leq \dt((a-\delta)_+).
\]
But $p$ is a projection, so for $\epsilon < 1$ we have $(p-\epsilon)_+\sim p$, so
\[
\dt((p-\epsilon)_+) = \dt(p).
\]
On the other hand,
\[
\dt((a-\delta)_+) < \dt(a) \leq \dt(p) = \dt((p-\epsilon)_+).
\]
This contradiction proves the proposition.
\end{proof}

\begin{proposition}\label{iotadef}
Let $A$ be a simple {\rm C}$^*$-algebra of stable rank one. Then, the map
\[
\iota\colon \W(A)_+\to \laff_b(\mathrm{T}(A))^{++}
\]
given by $\iota(\langle a\rangle)(\tau)=\dt(a)$ is a homomorphism.  If
$A$ has strict comparison of positive elements, then $\iota$ is an
order embedding.
\end{proposition}

\begin{proof}
Since $A$ is simple, every trace on $A$ is faithful and hence
$\iota(\langle a \rangle)$ is strictly positive.
We also know that the set $W(A)_+$ of purely positive elements is a semigroup so it is easily checked that $\iota$ is a homomorphism.

If $A$ has strict comparison of positive elements, then $\iota$ is an
order embedding by Proposition \ref{pureposcomparison}.
\end{proof}

\begin{theorem}\label{embedding}
Let $A$ be a simple, exact, unital {\rm C}$^*$-algebra with stable rank one. If $A$ has strict comparison of positive elements, then there is an order embedding
\[
\phi\colon \W(A) \to \widetilde{\W}(A)
\]
such that $\phi|_{\V(A)} = \mathrm{id}_{\V(A)}$ and $\phi|_{\W(A)_+} = \iota$.
\end{theorem}

\begin{proof}
As we have observed already, the Cuntz semigroup of a
C$^*$-algebra $A$ of stable rank one is always the disjoint union
of the monoid $\V(A)$ and $\W(A)_+$.

The map $\phi$ is well-defined, so it will suffice to prove that
it is an order embedding.  We verify conditions (i)-(iv) from
Definition \ref{wtilde}:
the image of $\phi|_{V(A)}$ is $V(A)$, with the same order, so (i)
is satisfied;  (ii) and (iii) follow from Proposition~\ref{pureposcomparison};
(iv) is Proposition~\ref{projpureposcomp}.
\end{proof}


\subsubsection{Surjectivity of the representation}

\subsubsection{A particular description of suprema}

The purpose of this section is to show that, under the assumption of stable rank one, suprema in the Cuntz semigroup $\W(A)$ have a particular description. This will lead to the fact that, if $a\precsim 1$ with $a\in M_{\infty}(A)_+$, then $a\sim a'$ for $a'\in A$.

Recall that, for $a\in A$ we use $A_a$ to denote the hereditary C$^*$-subalgebra generated by $a$, which, in the case $a$ is positive, equals the norm closure of $aAa$.

\begin{lemma}
\label{lem:incrhereditary} Let $A$ be a unital and separable
C$^*$-algebra, and let $a_n$ be a sequence of positive elements in
$A$ such that $A_{a_1}\subseteq A_{a_2}\subseteq \cdots$. Let
$A_{\infty}=\overline{\cup_{n=1}^{\infty}A_{a_n}}$, and let
$a_{\infty}$ be a strictly positive element of $A_{\infty}$. Then
\[
\langle a_{\infty}\rangle =\sup_n \langle a_n\rangle\,.
\]
Moreover, for any trace $\tau$ in $\T(A)$, we have
$\dt(a_{\infty})=\sup_n\dt(a_n)$.
\end{lemma}
\begin{proof}
To prove that $\langle a_{\infty}\rangle\geq \langle a_n\rangle$,
it suffices to prove that $A_{\infty}=A_{a_{\infty}}$. For this, it is enough to show that
$A_{\infty}$ is hereditary. Indeed, if $a\in A$ and $c_1$, $c_2\in
A_{\infty}$, then choose sequences $x_n$ and $y_n$ in $A_{a_n}$
such that
\[
\Vert x_n-c_1\Vert\rightarrow 0\text{ and }\Vert
y_n-c_2\Vert\rightarrow 0\,.
\]
Then $x_nay_n\in A_n$, and since $c_1ac_2=\lim\limits_n x_nay_n$,
we see that $c_1ac_2$. (Recall from, e.g.~\cite[Theorem 3.2.2]{bMurphy90}, that a
C$^*$-subalgebra $C$ of $A$ is hereditary if and only if
$c_1ac_2\in C$ whenever $a\in A$ and $c_1$, $c_2\in C$.)

Now assume that $\langle a_n\rangle\leq \langle b\rangle$ for all
$n$ in $\mathbb{N}$. Choose positive elements $x_n$ in $A_{a_n}$
such that $\Vert x_n-a_{\infty}\Vert<\delta_n $, where
$\delta_n\rightarrow 0$. It then follows by~\cite[Lemma 2.5
(ii)]{KR} that $\langle (a_{\infty}-\delta_n)_+ \rangle\leq
\langle x_n \rangle\leq \langle a_n\rangle\leq\langle b\rangle$.
Thus~\cite[Proposition 2.6]{KR} (or~\cite[Proposition
2.4]{Rfunct}) entails $\langle a_{\infty}\rangle\leq \langle
b\rangle$, as desired.

Also, since $\langle x_n\rangle\leq\langle a_n\rangle\leq\langle
a_{n+1}\rangle\leq\langle a_{\infty}\rangle$ for all $n$ and
$\lim_n x_n=a_{\infty}$, we have that, if $\tau\in\T(A)$,
\[
\sup_{n\to\infty}\dt(a_n)\leq\dt(a_{\infty})\leq\liminf_{n\to\infty}\dt(x_n)\leq\liminf_{n\to\infty}\dt(a_n)=\sup_{n\to\infty}\dt(a_n)
\]
\end{proof}

We shall assume in the results below that $\mathrm{sr}(A)=1$. Recall that, under this assumption, Cuntz
subequivalence is implemented by unitaries as we saw previously (in Proposition~\ref{basics}).
Note that, in this case, $a\precsim b$ implies that for each
$\epsilon>0$, there is $u$ in $U(A)$ such that
$A_{(a-\epsilon)_+}\subseteq uA_bu^*$. Indeed, if $a\in
A_{(a-\epsilon)_+}$, then find a sequence $(z_n)$ in $A$ such that
$a=\lim_n (a-\epsilon)_+z_n(a-\epsilon)_+$. Writing
$(a-\epsilon)_+=uc_{\epsilon}u^*$, with $c_{\epsilon}$ in $A_b$,
we see that $a=u(\lim_n c_{\epsilon}u^*z_n uc_{\epsilon})u^*\in
uA_b u^*$.
\begin{lemma}
\label{lem:supremainA} Let $A$ be a unital and separable
C$^*$-algebra with $\mathrm{sr}(A)=1$. Let $(a_n)$ be a sequence
of elements in $A$ such that $\langle a_1\rangle\leq\langle
a_2\rangle\leq\cdots$. Then $\sup_n\langle a_n\rangle$ exists in
$\W(A)$. If $\langle a_{\infty}\rangle=\sup_n\langle a_n\rangle$ then, for any $\tau$ in $\T(A)$, we have
$\dt(\sup_n\langle a_n\rangle)=\sup_n\dt(a_n)$.
\end{lemma}
\begin{proof}
Define numbers $\epsilon_n>0$ recursively. Let $\epsilon_1=1/2$,
and choose $\epsilon_n<1/n$ such that
\[
(a_j-\epsilon_j/k)_+\precsim (a_n-\epsilon_n)_+
\]
for all $1\leq j<n$ and $1\leq k\leq n$. (This is possible
using~\cite[Proposition 2.6]{KR} and because $a_j\precsim a_n$ for
$1\leq j< n$. Notice also that $(a_n-\epsilon)_+\leq
(a_n-\delta)_+$ whenever $\delta\leq\epsilon$.)

Since $(a_1-\epsilon_1/2)_+\precsim (a_2-\epsilon_2)_+$ and
$\mathrm{sr}(A)=1$, there is a unitary $u_1$ such that
\[
A_{((a-\epsilon_1/2)_+-\epsilon_1/2)_+}\subseteq
u_1A_{(a_2-\epsilon_2)_+}u_1^*\,.
\]
But $((a-\epsilon_1/2)_+-\epsilon_1/2)_+=(a_1-\epsilon_1)_+$
(see~\cite[Lemma 2.5]{KR}), so
\[
A_{(a-\epsilon_1)_+}\subseteq u_1A_{(a_2-\epsilon_2)_+}u_1^*\,.
\]
Continue in this way, and find unitaries $u_n$ in $A$ such that
\[
A_{(a-\epsilon_1)_+}\subseteq
u_1A_{(a_2-\epsilon_2)_+}u_1^*\subseteq
\]
\[
\subseteq u_1u_2A_{(a_3-\epsilon_3)_+}u_2^*u_1^*\subseteq \cdots
\subseteq
(\prod\limits_{i=1}^{n-1}
u_i)A_{(a_n-\epsilon_n)_+}(\prod\limits_{i=1}^{n-1}u_i)^*\subseteq\cdots
\]
Use Lemma~\ref{lem:incrhereditary} to find a positive element
$a_{\infty}$ in $A$ such that
\[
\langle a_{\infty}\rangle=\sup_n \langle
(a-\epsilon_n)_+\rangle\,,
\]
and also $\dt(a_{\infty})=\sup_n\dt((a-\epsilon_n)_+)\leq\sup_n\dt
(a_n)$ for any $\tau$ in $\T(A)$.

We claim that $\langle a_{\infty}\rangle=\sup_n\langle a_n\rangle$
as well. From this it will readily follow that
$\dt(a_{\infty})=\sup_n\dt(a_n)$.

To see that $\langle a_n\rangle\leq \langle a_{\infty}\rangle$ for
all $n$ in $\mathbb{N}$, fix $n<m$ and recall that, by
construction,
\[
\langle (a_n-\epsilon_n/(m-1))_+\rangle\leq \langle
(a_m-\epsilon_m)_+\rangle\leq \langle a_{\infty}\rangle\,.
\]
Hence, letting $m\rightarrow \infty$, we see that $\langle
(a_n-\epsilon)_+\rangle\leq \langle a_{\infty}\rangle$ for any
$\epsilon>0$, and so $\langle a_n\rangle\leq \langle
a_{\infty}\rangle$ for all $n$. Conversely, if $\langle
a_n\rangle\leq \langle b\rangle$ for all $n$ in $\mathbb{N}$, then
also $\langle (a_n-\epsilon_n)_+\rangle\leq \langle b\rangle$ for
all natural numbers $n$, and hence $\langle a_{\infty}\rangle\leq
\langle b\rangle$.
\end{proof}

\begin{theorem}
\label{thm:cuntzsuprema} Let $A$ be a unital and separable
{\rm C}$^*$-algebra with stable rank one. Then the supremum $\langle
a_{\infty}\rangle$ of every bounded sequence
$\{\langle a_n\rangle\}$ in $\W(A)$ stays in $\W(A)$. Moreover, $\dt(a_{\infty})=\sup_n\dt(a_n)$ for any
$\tau$ in $\T(A)$.
\end{theorem}
\begin{proof}
Let $\langle x_1\rangle\leq\langle x_2\rangle\leq\cdots$ be given,
and assume that $\langle x_n\rangle\leq k\langle 1_A\rangle$ for
all $n$.

The proof of Lemma~\ref{lem:supremainA} shows us that
we may choose a sequence $\epsilon_n>0$ strictly decreasing to zero with the following
properties:
\begin{enumerate}
\item[(i)] $\langle (x_n-\epsilon_n)_+ \rangle \leq \langle
(x_{n+1}-\epsilon_{n+1})_+\rangle$.

\item[(ii)] If $\langle (x_n-\epsilon_n)_+\rangle\leq \langle b\rangle$
for all $n$, then $\langle x_n\rangle\leq\langle b\rangle$ for all
$n$.
\end{enumerate}
Since $\langle x_n\rangle\leq k\langle 1_A\rangle$, find $y_n$ in
$M_{\infty}(A)_+$ such that
\[
(x_n-\epsilon_n)_+=y_n(1_A\otimes 1_{M_k})y_n^*\,.
\]
Define $a_n=(1_A\otimes 1_{M_k})y_n^*y_n(1_A\otimes 1_{M_k})$,
which is an element of $M_k(A)$. Then $\langle a_n\rangle=\langle
(x_n-\epsilon_n)_+\rangle\leq \langle a_{n+1}\rangle$ for all $n$.
Since $M_k(A)$ also has stable rank one, we may use
Lemma~\ref{lem:supremainA} to conclude that $\{\langle
a_n\rangle\}$ has a supremum $\langle a_{\infty}\rangle$ with
$a_{\infty}$ in $M_k(A)$. It follows that then $\langle
a_{\infty}\rangle$ is the supremum of $\{\langle
(x_n-\epsilon_n)_+\rangle\}$ in $\W(A)$. Evidently, our selection of
the sequence $\epsilon_n>0$ yields that $\langle
a_{\infty}\rangle=\sup_n\langle x_n\rangle$.

The proof that $\dt(a_{\infty})=\sup_n\dt(a_n)$ is identical to
the one in Lemma~\ref{lem:supremainA}.
\end{proof}
\begin{corollary}
Let $A$ be a unital and separable C$^*$-algebra with stable rank
one. If $x\in \W(A)$ is such that $x\leq \langle 1_A\rangle$, then
there is $a$ in $A$ such that $x=\langle a\rangle$.
\end{corollary}
\begin{proof}
There are a natural number $n$ and an element $b$ in $M_n(A)_+$
such that $x=\langle b\rangle$. For any $m$ in $\mathbb{N}$, find
elements $x_m$ such that
\[
(b-1/m)_+=x_m1_Ax_m^*\,,
\]
so the element $a_m=1_Ax_m^*x_m1_A\in A$ and clearly $a_m\sim (b-1/m)_+$. Moreover,
the sequence $\langle a_m\rangle$ is increasing, and another use of the proof of
Lemma~\ref{lem:supremainA} ensures that it has a supremum $a$ in
$A$. Clearly,
\[
\langle a\rangle=\sup_m\langle a_m\rangle=\sup_m\langle
(b-1/m)_+\rangle=\langle b\rangle\,.
\]
\end{proof}

\begin{corollary}\label{cor:whensupproj}
Let $A$ be a unital and separable {\rm C}$^*$-algebra with stable rank
one. If $\langle a_n\rangle$ is a bounded and increasing sequence
of elements in $W(A)$ with supremum $\langle a\rangle$. Then
$\langle a\rangle=\langle p\rangle$ for a projection $p$, if and
only if, there exists $n_0$ such that $\langle a_n\rangle=\langle
p\rangle$ whenever $n\geq n_0$.
\end{corollary}
\begin{proof}
Suppose that $\langle a\rangle=\sup_n\langle a_n\rangle=\langle
p\rangle$ for a projection $p$. We may assume that all the
elements $a$, $a_n$ and $p$ belong to $A$. For any $n$, we have
that $a_n\precsim p$. On the other hand, the proof of
Lemma~\ref{lem:supremainA} shows that $p=\lim_n b_n$, for some
elements $b_n\precsim (a_n-\epsilon_n)_+$ (where $\epsilon_n>0$ is
a sequence converging to zero). From this it follows that for
sufficiently large $n$, $p\precsim b_n\precsim
(a_n-\epsilon_n)_+\precsim a_n$. Thus $p\sim a_n$ if $n$ is large
enough, as desired.
\end{proof}

\subsubsection{Weak divisibility and surjectivity}

We first require a lemma, whose proof will be omitted.

\begin{lemma}
\label{lms:lift}
Let $A$ be a (unital) {\rm C}$^*$-algebra with non-empty tracial simplex $\T(A)$. For every $f\in\aff(\T(A))$, there is $a\in A_{sa}$ such that $f(\tau)=\tau(a)$ for every $\tau\in\T(A)$. If moreover $A$ is simple and $f\gg 0$, then we may actually choose $a\in A_+$.
\end{lemma}

\begin{definition}
\label{dfs:approxdivis}  We say that $\W(A)$ is \emph{weakly divisible} if for each $x \in \W(A)_+$ and $n \in \N$, there is $y \in \W(A)_+$ such that $n y \leq x \leq (n + 1) y$.
\end{definition}

\begin{lemma}
\label{one}
Let $A$ be a simple {\rm C}$^*$-algebram and assume that $\W(A)$ is weakly divisible. If $f \in \aff(\T(A))$ and $\epsilon > 0$ is given, then there exists $x \in \W(A)_+$ such that $|f(\tau) - \iota x (\tau)| < \epsilon$ for all $\tau \in\T(A)$.
\end{lemma}

\begin{proof}  By Lemma \ref{lms:lift}, we can find a positive element $b \in A$ such that $f(\tau) = \tau(b)$ for all $\tau \in\T(A)$. Identifying $C^*(b)$ with $C_0(\sigma(b))$, we may choose rational numbers $r_i = \frac{p_i}{q_i}$ ($p_i, q_i \in \N$) and open sets $\mathcal{O}_i \subset \sigma(b)$ such that
\[
\|b - \sum_{i=1}^n r_i \chi_{\mathcal{O}_i}\| < \epsilon/2\,,
\]
where $\chi_\mathcal{O}$ is the characteristic function of an open set (and $\| \cdot \|$ denotes the point-wise supremum norm). Now, let $b_i \in C^*(b)$ be positive functions such that $b_i (s) > 0$ if and only if $s \in  \mathcal{O}_i$. It follows that for every (tracial) state $\gamma$ on $C^*(b)$, we have $d_{\gamma} \langle b_i \rangle = \gamma(\chi_{\mathcal{O}_i})$ (where here we have extended $\gamma$ to the enveloping von Neumann algebra) and hence
\[
|\sum_{i=1}^n r_i d_{\gamma} (\langle b_i \rangle) - \gamma(b)| < \epsilon/2\,.
\]
By weak divisibility, we can find $y_i \in \W(A)_+$ such that $q_i y_i \leq \langle b_i \rangle \leq (q_i + 1) y_i$.  Finally, we define
\[
x = \sum_{i=1}^n p_i y_i\,,
\]
and claim that this is the desired approximation.  Indeed, for every $\tau \in \T(A)$ we have
\[
\frac{1}{q_i + 1} \dt(\langle b_i \rangle ) \leq \dt(y_i) \leq \frac{1}{q_i} \dt(\langle b_i \rangle)\,,
\]
and hence
\[
\sum_{i=1}^n \frac{p_i}{q_i + 1} \dt(\langle b_i \rangle ) \leq \dt(x) \leq \sum_{i=1}^n r_i \dt(\langle b_i \rangle)\,.
\]
If necessary, replacing $p_i$ and $q_i$ with $kp_i$ and $kq_i$, respectively, for large $k$, we may assume that
\[
| \dt(x) - \sum_{i=1}^n r_i \dt(\langle b_i \rangle) | < \epsilon/2
\]
for all $\tau \in \T(A)$. This completes the proof.
\end{proof}

\begin{lemma}
\label{lms:increasingsup} Let $f \in\laff_b(\T(A))^{++}$, and choose $\delta > 0$ such that $f(\tau) \geq \delta$ for every
$\tau \in \T(A)$.  There is then a sequence $f_n$ of elements in $\aff(\T(A))$ with the following properties:
\begin{enumerate}[{\rm (i)}]
\item $\sup_n f_n(\tau) = f(\tau)$ for every $\tau \in \T(A)$;
\item
\[
f_{n+1}(\tau) - f_n( \tau) \geq \frac{\delta}{2} \left( \frac{1}{n} - \frac{1}{n+1} \right), \ \text{for all } \tau \in \T(A)\,.
\]
\end{enumerate}
\end{lemma}

\begin{proof}
Since lower semicontinuous affine functions are suprema of \emph{strictly} increasing sequences of continuous affine functions (cf.\ \cite[Proposition 11.8]{poag}), there is a strictly increasing sequence $h_n$ in $\aff(\T(A))^{++}$ such that
\[
\sup_n h_n(\tau) = f( \tau)-\frac{\delta}{2}, \ \text{ for all } \tau \in\T(A).
\]
Now set
\[
f_n(\tau) = h_n(\tau) + \frac{\delta}{2} - \frac{\delta}{2n}.
\]
Straightforward calculation shows that $(f_n)$ has properties (i), and (ii).
\end{proof}

\begin{theorem}
\label{props:surjective}
Assume $A$ is separable, simple, exact, with stable rank one and strict comparison.  If $\W(A)$ is weakly divisible, then $\iota \colon \W(A)_+ \to \laff_b(\T(A))^{++}$ is surjective.
\end{theorem}

\begin{proof}
Fix $f \in \laff_b(\T(A))$ and choose $f_n \in\aff(\T(A))^{++}$ satisfying the conclusion of Lemma \ref{lms:increasingsup}. Let $\epsilon_n > 0$ be small enough that if $| g(\tau) - f_n(\tau) | < \epsilon_n$ and $|h(\tau) - f_{n+1}(\tau) | < \epsilon_{n+1}$ for all $\tau \in \T(A)$, then $g(\tau) < h(\tau)$ for all $\tau \in \T(A)$.  By Lemma \ref{one}, choose $x_n \in \W(A)_+$ such that $|\iota x_n(\tau) - f_n(\tau)| < \epsilon_n$ for all $\tau\in \T(A)$.  Strict comparison implies $x_n \leq x_{n+1}$ for all $n$, and hence
\[
x := \sup x_n
\]
exists in $\W(A)$ (by Theorem \ref{thm:cuntzsuprema}). Moreover, the same Theorem \ref{thm:cuntzsuprema} yields $\iota x (\tau) = f(\tau)$ for all $\tau \in \T(A)$.  Since $x_n$ is a strictly increasing sequence, $x \neq x_n$ for all $n\in \N$ and hence Corollary \ref{cor:whensupproj} ensures that $x \in \W(A)_+$, completing the proof.
\end{proof}

\subsection{Representation Theorems for the Cuntz semigroup: the
stable case}

The representation results for the stable case (established in \cite{bt}), where we assume stabilisation of a unital algebra, follow from the unital case but extra care is needed here. We shall denote by $\laff(\T(A))^{++}$ the semigroup of those strictly positive, lower semicontinuous affine functions defined on the trace simplex of $A$, which are not necessarily bounded, hence may take infinite values.

\begin{lemma}
\label{supfromA}
Let $A$ be a {\rm C}$^*$-algebra. Then every element from $\W(A\otimes\mathbb{K})$ is a supremum of a rapidly increasing sequence from $\W(A)$.
\end{lemma}
\begin{proof}
Let $x\in \W(A\otimes\mathbb{K})$ and assume $a\in A\otimes\mathbb{K}_+$ is such that $x=\langle a\rangle$ Then, for each $n$, we can find $a_n\in M_{\infty}(A)_+$ such that
\[
\Vert a-a_n\Vert<\frac{1}{n}\,,
\]
whence $(a-\frac{1}{n})_+=d_na_nd_n^*$ for some contraction $d_n$. Then $(a- \frac{1}{n})_+\sim a_n^{1/2}d_n^*d_na_n^{1/2}\in M_{\infty}(A)_+$, which gives a rapidly increasing sequence with supremum $a$.
\end{proof}

Next, for  unital C$^*$-algebra $A$ of stable rank one (so that we can split the Cuntz semigroup into the projection part and the purely positive part), define
\[
\iota\colon\W(A\otimes\mathbb{K})_+\to\laff(\T(A))^{++}\,,
\]
by $\iota(x)=\sup\limits_n\iota(x_n)$, where $(x_n)$ is any rapidly increasing sequence from $\W(A)$ with supremum $x$ (shown above to exist), and abusing the language to mean that $\iota(x_n)$ is defined as in the previous section.

\begin{lemma}
$\iota$ is well defined.
\end{lemma}
\begin{proof}
We need to show this is not dependant on the rapidly increasing sequence we choose. Indeed, if $(x_n)$ and $(y_n)$ are rapidly increasing sequences in $\W(A)$ with the same supremum, then by the very definition we have that for each $n$ there is $m$ such that $x_n\leq y_m$ and we can also find $k$ such that $y_m\leq x_k$. Altogether this implies that $\iota(x_n)\leq\iota(y_m)\leq\iota(x_k)$. so that
\[
\sup_n\iota(x_n)=\sup_n\iota(y_n),,
\]
as desired.
\end{proof}
\begin{lemma}
$\iota$ preserves order and suprema.
\end{lemma}
\begin{proof}
That $\iota$ preserves order follows from its definition and the results we have established previously.

To see it preserves suprema, let $(x_n)$ be an increasing sequence in $\W(A\otimes\mathbb{K})$ and let $x=\sup x_n$. Write $x$ as a supremum of a rapidly increasing sequence $(y_n)$ coming from $\W(A)$. Then, by definition, $\iota(x)=\sup_n\iota(y_n) $. On the other hand, since $(y_n)$ is rapidly increasing and yields the same supremum as $(x_n)$, we have
that for each $n$, there is $m$ with $y_n\leq x_m$, whence $\iota (y_n)\leq\iota(x_m)\leq\sup\iota(x_i)$. Therefore
\[
\iota(x)=\sup\iota(y_n)\leq\sup_n\iota(x_n)\,,
\]
and the result follows.
\end{proof}

\begin{proposition}
\label{prop:themaphi}
Let $A$ be a simple, unital, exact {\rm C}$^*$-algebra of stable rank one. Then, the map
\[
\phi\colon\W(A\otimes\mathbb{K})\to \V(A)\sqcup\laff(\T(A))^{++}\,,
\]
defined as $\phi_{|\V(A)}=\mathrm{id}_{|\V(A)}$ and $\phi_{|\W(A\otimes\mathbb{K})_+}=\iota_{\W(A\otimes\mathbb{K})_+}$ is an order-embedding when $A$ has strict comparison.
\end{proposition}
\begin{proof}
The proof is left as an exercise.
\end{proof}

Using our previous observations and pretty much in the same way as in the unital case, surjectivity is established in the presence of weak divisibility.

\begin{theorem}
Let $A$ be a simple, unital, exact {\rm C}$^*$-algebra of stable rank one. If $A$ has strict comparison and is $\W(A)$ is weakly divisible, the map $\phi$ in Proposition \ref{prop:themaphi} is an order-isomorphism.
\end{theorem}

\subsection{Regularity properties}

\subsubsection{$\Z$-stability}

The Jiang-Su algebra $\Z$ is one of the most prominent examples of simple, separable, amenable and infinite dimensional C$^*$-algebras. It was discovered in \cite{jiangsu}, and it has the same Elliott invariant as the complex numbers.

We briefly describe the Jiang-Su algebra below.

Given natural numbers $p$ and $q$, we define the \emph{dimension drop} algebra as
\[
Z_{p,q}=\{f\in C([0,1],M_p\otimes M_q)\mid f(0)\in M_p\otimes\mathbb{C},\,\, f(1)\in \mathbb{C}\otimes M_q\}\,,
\]
which is called a \emph{prime} dimension drop algebra if $p$ and $q$ are relatively prime. In this case, it is known that $Z_{p,q}$ has no non-trivial projections and $\mathrm{K}_0(Z_{p,q})\cong\mathbb{Z}$, $\mathrm{K}_1(Z_{p,q})=0$.

\begin{theorem} {\rm (Jiang-Su, \cite{jiangsu})}
Any inductive limit of prime dimension drop {\rm C}$^*$-algebras with unital maps which is simple and has a unique tracial state is isomorphic to the Jiang-Su algebra $\Z$.
\end{theorem}

Jiang and Su also proved in their seminal paper that $\Z\otimes \Z\cong \Z$ and that this also holds for infinitely many copies of $\Z$.

The following definition is very important:
\begin{definition}
A {\rm C}$^*$-algebra $A$ is $\Z$-stable (or absorbs $\Z$ tensorially) provided that $A\otimes\Z\cong A$.
\end{definition}

Notice that, since $\Z$ itself is $\Z$-stable, given any C$^*$-algebra $A$, it follows that $A\otimes\Z$ is $\Z$-stable. Jiang and Su proved that separable, simple, unital AF algebras are $\Z$-stable (so long they are infinite dimensional). This is also the case for purely infinite simple, amenable algebras.

The property of $\Z$-stability is highly relevant for the classification programme. Indeed, we have the following (that can be found in \cite{GJSu}):

\begin{theorem}
{\rm (Gong, Jiang, Su)} If $A$ is simple, unital with weakly unperforated $\mathrm{K}_0$-group, then the Elliott invariant of $A$ and that of $A\otimes\Z$ are isomorphic.
\end{theorem}

This result would make it agreeable that the largest restricted class of algebras for which (EC) can hold consists of those algebras that are $\Z$-stable. This is what the conjecture predicts for the ones that have weakly unperforated K-Theory.

\subsubsection{Uniqueness of $\Z$}

The ad hoc description of $\mathcal{Z}$ asks for some profound reason that singles out this algebra among all unital, separable,
nuclear, simple C$^*$-algebras without finite-dimensional representations. This can take the form of a universal property
that hopefully identifies $\mathcal{Z}$. We briefly examine two such properties, the first of which was proposed by M.
R\o rdam, and the second by A. S. Toms. For a class $\mathcal{C}$ of separable, unital and nuclear C$^*$-algebras, consider the following two pairs of
conditions that might be satisfied by $A\in\mathcal{C}$:

\begin{enumerate}[{\rm (i)}]
\item  every unital endomorphism of $A$ is approximately inner, and
\item $A$ embeds unitally in every $B\in\mathcal{C}$.
\end{enumerate}

and

\begin{enumerate}[{\rm (i)'}]
\item $A^{\otimes\infty}\cong A$, and
\item $B^{\otimes\infty}\otimes A\cong B^{\otimes\infty}$ for every $B\in\mathcal{C}$
\end{enumerate}

Any C$^*$-algebra satisfying either pair of conditions is unique
up to isomorphism. Concentrating on the second pair, we check that if $A_1$ and $A_2\in\mathcal{C}$, then
\[
A_1\stackrel{(i)'}{\cong} A_1^{\otimes\infty}\stackrel{(ii)'}{\cong}A_1^{\otimes\infty}\otimes A_2\stackrel{(i)'}{\cong}A_1\otimes A_2^{\otimes\infty}\stackrel{(ii)'}{\cong}A_2^{\otimes\infty}\stackrel{(i)'}{\cong} A_2\,.
\]
It has been shown recently, by Dadarlat and Toms (\cite{datoms}), that  $\Z$ satisfies the second of the abovementioned universal properties with $\mathcal{C}$ the class of unital C$^*$-algebras that
contain unitally a subhomogeneous algebra without characters. This class is huge and contains a wide range of examples.

As for the first universal property, it has also been shown in recent work by Dadarlat, Hirshberg, Toms and Winter that there is a unital simple AH algebra that does not admit a unital embedding of
$\mathcal{Z}$.

\subsubsection{$\Z$-stability versus strict comparison}

The relationship between $\Z$-stability and strict comparison is somewhat mysterious, and it is not know yet whether they are equivalent, but in light of examples it might well be that they are, at least in the stably finite case.

What we do know is the following, proved by R\o rdam in \cite[Theorem 4.5]{roijm}. For its proof we will use a lemma:
\begin{lemma}
\label{tensor}
If $A$ and $B$ are {\rm C}$^*$-algebras and $n\langle a\rangle\leq m\langle a'\rangle$ in $\W(A)$, then $n\langle a\otimes b\rangle\leq m\langle a'\otimes b\rangle$ in $\W(A\otimes B)$. By symmetry, if $n\langle b\rangle\leq m\langle b'\rangle$ in $\W(B)$, then $n\langle a\otimes b\rangle\leq m\langle a\otimes b'\rangle$ in $\W(A\otimes B)$.
\end{lemma}

\begin{theorem}
\label{walmunp}
Let $A$ be a $\Z$-stable  {\rm C}$^*$-algebra. Then $\W(A)$ is almost unperforated.
\end{theorem}
\begin{proof}
We sketch the main argument, skipping some technical details.

One can construct, working within $\Z$, a sequence of elements $(e_n)$ such that $n\langle e_n\rangle\leq\langle 1_{\Z}\rangle\leq (n+1)\langle e_n\rangle$ (this is in \cite[Lemma 4.2]{roijm}).

Now, assume that $(n+1)\langle a\rangle\leq n\langle a'\rangle$ for $a$, $a'\in A$. Then we have
\[
\langle a\otimes 1_{\Z}\rangle\leq (n+1)\langle a\otimes e_n\rangle\leq n\langle a'\otimes e_n\rangle\leq \langle a'\otimes 1_{\Z}\rangle\,,
\]
in $\W(A\otimes\Z)$.

Use now that $\Z\cong\Z^{\otimes\infty}$ to construct a sequence of isomorphisms $\sigma_n\colon A\otimes\Z\to A$ such that
\[
\Vert \sigma_n(a\otimes 1)-a\Vert\stackrel{n\to\infty}{\to} 0\,.
\]
Given $\epsilon>0$, we can find $x$ with $\Vert x^*(a'\otimes 1_{\Z})x-a\otimes 1_{\Z}\Vert <\epsilon$.
Now let $x_k=\sigma_k(x)$ and we have $\Vert x_k^*\sigma_k(a'\otimes 1_{\Z})x_k-\sigma_k(a\otimes 1_{\Z})\Vert <\epsilon$, whence for $k$ large enough we get
\[
\Vert x_k^*a'x_k-a\Vert <\epsilon\,,
\]
so $a\precsim a'$, as desired.
\end{proof}

\begin{corollary}
If $A$ is simple and $\Z$-stable, then $A$ has strict comparison of positive elements.
\end{corollary}

The converse of the previous corollary has been conjectured to be true, but still remains a conjecture.

A similar argument to the one used above to show almost unperforation yields weak divisibility. We first need a lemma, that appears in \cite{pt}.
\begin{lemma}\label{specialrep}
Let $A$ be a unital and $\Z$-stable {\rm C}$^*$-algebra, with $a \in A_+$.
Then, $a$ is Cuntz equivalent to a positive element of the form $b \otimes 1_{\Z}
\in A \otimes \Z \cong A$.
\end{lemma}
\begin{proof}
Let $\psi\colon \Z \otimes \Z \to \Z$ be a $^*$-isomorphism, and
put
\[
\phi = (\mathrm{id}_{\Z} \otimes 1_{\Z}) \circ \psi\colon \Z\otimes\Z\to\Z\otimes\mathbf{1}_{\Z}\,.
\]
By \cite[Corollary 1.12]{TW1}, $\phi$ is approximately inner, and therefore so also is
\[
\mathrm{id}_A \otimes \phi\colon A \otimes \Z^{\otimes 2} \to A \otimes \Z \otimes \mathbf{1}_{\Z}.
\]
In particular, there is a sequence of unitaries $u_n$ in $A \cong A \otimes \Z^{\otimes 2}$
such that
\[
||u_n a u_n^* - (\mathrm{id}_A\otimes\phi)(a)|| \stackrel{n \to \infty}{\longrightarrow} 0.
\]
Approximate unitary equivalence
preserves Cuntz equivalence classes, whence $\langle a \rangle = \langle \phi(a) \rangle$.
The image of $\phi(a)$ is, by construction, of the form $b \otimes 1_{\Z}$
for some $b \in A \otimes \Z \cong A$.
\end{proof}

\begin{theorem}
If $A$ is a $\Z$-stable algebra, then $\W(A)$ is weakly divisible.
\end{theorem}
\begin{proof}
Construct elements $e_n$ in $\Z$ as in the proof of Theorem \ref{walmunp} so that $n\langle e_n \rangle\leq\langle 1_{\Z}\rangle\leq (n+1)\langle e_n\rangle$.

Given $a\in A$, use Lemma \ref{tensor} to conclude that $n\langle a\otimes e_n \rangle\leq\langle a\otimes 1_{\Z}\rangle\leq (n+1)\langle a\otimes e_n\rangle$.

Since by Lemma \ref{specialrep}, it is enough to consider elements of the form $a\otimes 1_{\Z}$, the result follows.
\end{proof}

It has been shown by M. R\o rdam that simple and $\Z$-stable algebras are either purely infinite simple or they have stable rank one (see \cite{roijm}). We shall make use of this fact whenever we need it.

\subsubsection{Finite decomposition rank}

Decomposition rank is a topological property, originally defined by Kirchberg and Winter, that in its lowest instance captures, like real rank, the covering dimension of the underlying space. The definition is somewhat involved, and we include it here mainly for completeness. Recall that a completely positive map (c.p. map) $f\colon A\to B$ (where $A$ and $B$ are C$^*$-algebras) is a linear map such that $f(x)\geq 0$ whenever $x\in A_+$ and so does every extension of $f$ to $M_n(A)$. We say that $f$ is a completely positive contraction (c.p.c.) if $f$ is c. p. and $f$ and its extensions to matrices are all contractions.

\begin{definition}
\begin{enumerate}[{\rm (i)}]
\item A c. p. map $f\colon F\to A$ has order zero if it maps orthogonal elements to orthogonal elements.
\item If $F$ is a finite dimensional algebra, a c.p. map $f\colon F\to A$ is $n$-decomposable if there is a decomposition $F=F_1\oplus\cdots\oplus F_n$ such that $f_{|F_i}$ has order zero for each $i$.
\item We say that $A$ has decomposition rank $n$, in symbols, $\mathrm{dr}(A)=n$, if $n$ is the smallest integer such that: for any finite subset $G\subset A$ and $\epsilon>0$, there are a finite dimensional algebra $F$ and c. p. c. maps $f\colon F\to A$ and $g\colon A\to F$ such that $f$ is $n$-decomposable and
\[
\Vert fg(b)-b\Vert<\epsilon
\]
for all $b\in G$.
\end{enumerate}
\end{definition}

A C$^*$-algebra has finite decomposition rank if $\mathrm{dr}(A)<\infty$. W. Winter has proved recently the following remarkable result (\cite[Theorem 5.1]{wdr}):

\begin{theorem}
Let $A$ be a finite, non elementary, simple, unital {\rm C}$^*$-algebra. If $\mathrm{dr}(A)<\infty$ then $A$ is $\Z$-stable.
\end{theorem}

This result has strong consequences, namely, it allows classification results with the Elliott invariant that were established (up to $\Z$-stability) for algebras with finite decomposition rank. The class to which these results apply includes the so-called C$^*$-algebras associated to uniquely ergodic, smooth, minimal dynamical systems, and in particular crossed products of the form $C(S^3)\times_\alpha\mathbb{Z}$ for minimal diffeomorphisms $\alpha$. It also contains UHF-algebras, Bunce Deddens algebras, irrational rotation algebras, among others.

\subsection{Back to the Elliott Conjecture -- ways forward}

\subsubsection{Success and failure of the Elliott Programme}

The classification programme had a great deal of success, and even provided with some spectacular results. If we start considering purely infinite simple algebras, we have

\begin{theorem}
{\rm (Kirchberg-Phillips, \cite{K},\cite{P})} If $A$ and $B$ are separable, nuclear {\rm C}$^*$-algebras, purely infinite and simple, and such that satisfy the so-called UCT, then if $\mathrm{Ell}(A)\cong \mathrm{Ell}(B)$, there is a $^*$-isomorphism $\varphi\colon A\cong B$ that induces the isomorphism at the level of the invariant.
\end{theorem}

We should note here that the Elliott invariant does not include traces as the tracial simplex in the purely infinite case is empty.

If we turn to the stably finite case, a lot of attention has been paid to algebras that admit an inductive limit decomposition. This may well be due to historical reasons and the fact that this line of research has proved so powerful over the years. The first result in this direction, although not phrased in this way, goes back to Glimm and his classification of UHF algebras by means of $\mathrm{K}_0$.

Recall that a UHF algebra is a C$^*$-algebra that appears as the inductive limit of a sequence of matrix algebras of the form $M_n$, together with unital $^*$-homomorphisms. The result is then

\begin{theorem} {\rm (Glimm, \cite{glimm})}
If $A$ anb $B$ are UHF algebras and $\mathrm{Ell}(A)\cong \mathrm{Ell}(B)$, then there is a $^*$-isomorphism $\varphi\colon A\cong B$ that induces the isomorphism at the level of the invariant.
\end{theorem}

Note that also in this case the invariant simplifies as $\mathrm{K}_1$ is trivial and there is only one trace.

The argument used by Glimm, an intertwining argument, was used later by Elliott (\cite{El1}) to classify in the same terms AF-algebras (more general than UHF), and has been used a number of times, with variations to allow approximate intertwinings rather than exact ones.

The AH class is where we find the most important classification result. Note that this class won't have real rank zero in general, whereas purely infinite simple algebras and AF algebras do.

An AH algebra is an inductive limit of a sequence $(A_i,\varphi_i)$ where each $A_i$ is a direct sum of algebras which look like finite matrices over $C(X_{i,j})$, for some compact metric spaces. The key condition for an AH algebra is that of slow dimension growth. Roughly speaking, this means that the dimension of the spaces compared to the sizes of the matrices tend to zero as we go along the limit decomposition (which, by the way, is not unique).

If, instead of comparing the dimensions of spaces to the sizes of matrices, we consider the dimensions to the third power, then we get the definition of \emph{very slow dimension growth}. The theorem is as follows:

\begin{theorem} {\rm (Elliott, Gong, Li, and Gong, \cite{EGL}, \cite{Gong})}
(EC) holds among simple unital AH algebras with very slow dimension growth.
\end{theorem}

The algebras in question mentioned in the result above are $\Z$-stable.

Villadsen exhibited (\cite{V}) in the mid 1990's examples of simple nuclear C$^*$-algebras that failed to satisfy strict comparison for projections. The techiques in those examples led to R\o rdam and Toms to produce counterexamples to Elliott's conjecture:

(i) M. R\o rdam (\cite{R2}) constructed a simple, nuclear C$^*$-algebra containing a finite and an infinite projection, and whose K-Theory was that of a purely infinite simple algebra, yet it cannot be purely infinite simple.

(ii) A. Toms (\cite{T2}) produced examples of two non-isomorphic simple, unital AH-algebras that agreed on their K-Theory, real rank, stable rank and other continuous and stable isomorphism invariants.

\subsubsection{Two equivalent conjectures}

In this section, we will make precise the meaning of (EC) and (WEC). Thus we shall define the categories in which the relevant invariants sit.

Let $\mathcal{E}$ denote the category whose objects are 4-tuples
\[
\left( (G_0,G_0^+,u),G_1,X,r \right),
\]
where $(G_0,G_0^+,u)$ is a simple
partially ordered Abelian group with distinguished order-unit
$u$ and state space $S(G_0,u)$, $G_1$ is a countable
Abelian group, $X$ is a metrizable Choquet simplex, and
$r\colon X \to S(G_0,u)$ is an affine map.  A morphism
\[
\Theta\colon \left( (G_0,G_0^+,u),G_1,X,r \right) \to
\left( (H_0,H_0^+,v),H_1,Y,s \right)
\]
in $\mathcal{E}$ is a 3-tuple
\[
\Theta = (\theta_{0}, \theta_{1}, \gamma)
\]
where
\[
\theta_{0}\colon (G_0,G_0^+,u) \to (H_0,H_0^+,v)
\]
is an order-unit-preserving positive homomorphism,
\[
\theta_{1}\colon G_1 \to H_1
\]
is any homomorphism, and
\[
\gamma\colon Y \to X
\]
is a continuous affine map that makes the diagram below commutative:
\[
\xymatrix{
{Y}\ar[r]^-{\gamma}\ar[d]^-{s} & {X}\ar[d]^-{r} \\
{S(H_0,v)}\ar[r]^-{\theta_0^*} & {S(G_0,u)\,.}}
\]
For a simple unital C$^{*}$-algebra $A$ the Elliott invariant
$\mathrm{Ell}(A)$ is an element of $\mathcal{E}$, where
$(G_0,G_0^+,u) = (\mathrm{K}_{0}(A), \mathrm{K}_{0}(A)^{+}, [1_{A}])$,
$G_1  =  \mathrm{K}_{1}(A)$, $X  =  \mathrm{T}(A)$,
and $r_A$ is given by evaluating a given trace at a $\mathrm{K}_0$-class.
Given a class $\mathcal{C}$ of simple unital $C^*$-algebras, let
$\mathcal{E}(\mathcal{C})$ denote the subcategory of $\mathcal{E}$ whose
objects can be realised as the Elliott invariant of a member of
$\mathcal{C}$, and whose morphisms
are all admissible maps between the now specified objects.

The definition of $\mathcal{E}$ removes an ambiguity from
the (on the other hand intuitive) statement of (EC), namely, what is meant
by an isomorphism of Elliott invariants.

In order to do the same for (WEC), we let $\mathcal{W}$ be the category whose objects are ordered pairs
\[
\left( (\W(A),\langle 1_A \rangle), \mathrm{Ell}(A) \right),
\]
where $A$ is a simple, unital, exact, and stably finite C$^*$-algebra, $(\W(A),\langle 1_A \rangle)$
is the Cuntz semigroup of $A$ together with the distinguished order-unit $\langle 1_A \rangle$,
and $\mathrm{Ell}(A)$ is the Elliott invariant of $A$. A morphism
\[
\Psi\colon \left( (\W(A),\langle 1_A \rangle), \mathrm{Ell}(A) \right) \to
\left( (\W(B),\langle 1_B \rangle), \mathrm{Ell}(B) \right)
\]
in $\mathcal{W}$ is an ordered pair
\[
\Psi = (\Lambda, \Theta),
\]
where $\Theta = (\theta_0,\theta_1,\gamma)$ is a morphism  in $\mathcal{E}$ and
$\Lambda\colon (\W(A),\langle 1_A\rangle)\to (\W(B),\langle 1_B\rangle)$ is an order- and
order-unit-preserving semigroup homomorphism satisfying two
compatibility conditions:  first,
\[
\xymatrix{
{ (\V(A),\langle 1_A \rangle)}\ar[r]^-{\Lambda|_{\V(A)}}\ar[d]^-{\rho}&
{(\V(B),\langle 1_B \rangle)}\ar[d]^-{\rho} \\
(\mathrm{K}_0(A),[1_A])\ar[r]^-{\theta_0}&(\mathrm{K}_0(B),[1_B])\,,}
\]
where $\rho$ is the usual Grothendieck map from $\V(\bullet)$ to $\mathrm{K}_0(\bullet)$ (viewing $\V(A)$ as a subsemigroup of $\W(A)$, which holds for stably finite algebras);
second,
\[
\xymatrix{
\mathrm{LDF}(B)\ar[r]^-{\Lambda^*}\ar[d]^-{\eta} & \mathrm{LDF}(A)\ar[d]^-{\eta} \\
\mathrm{T}(B)\ar[r]^-{\gamma}&\mathrm{T}(A)\,,}
\]
where $\eta$ is the affine bijection between
$\mathrm{LDF}(\bullet)$ and $\mathrm{T}(\bullet)$ given by
$\eta(\dt) = \tau$ (see~\cite[Theorem II.2.2]{bh}).  These
compatibility conditions are automatically satisfied if $\Psi$ is induced by
a $*$-homomorphism $\psi: A \to B$.

Recall the definition of $\widetilde{\W}(A)$ as an ordered semigroup (see \ref{wtilde}).

Let $\mathcal{\widetilde{\W}}$ be the category whose objects are
of the form $(\widetilde{\W}(A),[1_A])$ for some exact, unital, and stable rank
one C$^*$-algebra $A$, and whose morphisms are positive order-unit-preserving
homomorphisms
\[
\Gamma\colon (\widetilde{\W}(A),[1_A]) \to (\widetilde{\W}(B),[1_B])
\]
such that
\[
\Gamma(\V(A)) \subseteq \V(B)
\]
and
\[
\Gamma|_{\laff_b(\mathrm{T}(A))^{++}}\colon \laff_b(\mathrm{T}(A))^{++} \to \laff_b(\mathrm{T}(B))^{++}
\]
is induced by a continuous affine map from $\mathrm{T}(B)$ to $\mathrm{T}(A)$.

For the next definition, we remind the reader that $\V(A) \cong\mathrm{K}_0(A)^+$
for a C$^*$-algebra of stable rank one.

\begin{definition}\label{recovfunc} Let $\mathcal{C}$ denote the class of simple, unital, exact, and
stable rank one {\rm C}$^*$-algebras.
Let
\[
F\colon\mathbf{Obj}(\mathcal{E}(\mathcal{C})) \to \mathbf{Obj}(\mathcal{\widetilde{W}})
\]
be given by
\[
F\left( (\mathrm{K}_0(A),\mathrm{K}_0(A)^+,[1_A]),\mathrm{K}_1(A),\mathrm{T}(A),r_A \right) =
(\widetilde{\W}(A),[1_A]).
\]
Define
\[
F\colon\mathbf{Mor}(\mathcal{E}(\mathcal{C})) \to \mathbf{Mor}(\mathcal{\widetilde{W}})
\]
by sending $\Theta = (\theta_0,\theta_1,\gamma)$ to the morphism
\[
\Gamma\colon(\widetilde{\W}(A),[1_A]) \to (\widetilde{\W}(B),[1_B])
\]
given by $\theta_0$ on $\mathrm{K}_0(A)^+ = \V(A)$ and induced by $\gamma$ on $\laff_b(\mathrm{T}(A))^{++}$.
\end{definition}

The next proposition holds by definition.

\begin{proposition} With $\mathcal{C}$ as in Definition \ref{recovfunc}, the map
$F\colon\mathcal{E}(\mathcal{C}) \to \mathcal{\widetilde{W}}$ is a functor.
\end{proposition}

\begin{theorem}
(EC) implies (WEC) for the class of simple, unital, separable, and nuclear
{\rm C}$^*$-algebras with strict comparison of positive elements and $\mathrm{sr}
\in \{1,\infty\}$.
\end{theorem}

\begin{proof}
Algebras in the class under consideration are either purely infinite
or stably finite.
The theorem is trivial for the subclass of purely infinite algebras,
due to the degenerate nature of the Cuntz semigroup in this setting.
The remaining case is that of stable rank one.

Let $A$ and $B$ be simple, separable, unital, nuclear, and stably finite
C$^*$-algebras with strict comparison of positive elements, and suppose
that (EC) holds.  Let there be given an isomorphism
\[
\phi\colon \left(\W(A), \langle 1_A \rangle, \mathrm{Ell}(A)\right) \to
\left(\W(B), \langle 1_B \rangle, \mathrm{Ell}(B)\right).
\]
Then by restricting $\phi$ we have an isomorphism
\[
\phi|_{\mathrm{Ell}(A)}\colon \mathrm{Ell}(A) \to \mathrm{Ell}(B),
\]
and we may conclude by (EC) that there is a $*$-isomorphism $\Phi\colon A \to B$
such that $I(\Phi) = \phi|_{\mathrm{Ell}(A)}$.  Since $\Phi$ is unital it preserves the Cuntz class
of the unit.  The compatibility conditions imposed on $\phi$ together with Theorem
\ref{embedding} ensure that $\phi|_{\W(A)}$
is determined by $\phi|_{\V(A)}$ and $\phi^{\sharp}: \mathrm{T}(B) \to \mathrm{T}(A)$.
Thus, $\Phi$ induces $\phi$, and (WEC) holds.
\end{proof}

Note that the semigroup homomorphism $\phi$ in Theorem~\ref{embedding} is an isomorphism if and only if
$\iota$ is surjective.

Let $\mathrm{(EC)}^{'}$ and $\mathrm{(WEC)}^{'}$ denote the conjectures (EC) and (WEC),
respectively, but expanded to apply to all simple, unital, exact, and stably finite $C^*$-algebras.
Collecting the results of this section we have:

\begin{theorem}\label{ecwec}
Let $\mathcal{C}$ be the class of simple, unital, exact, finite, and $\mathcal{Z}$-stable
C$^*$-algebras.  Then, $(EC)^{'}$ and $(WEC)^{'}$ are equivalent in $\mathcal{C}$.  Moreover, there is a functor
$G\colon \mathcal{E}(\mathcal{C}) \to \mathcal{W}$ such that
\[
G(\mathrm{Ell}(A)) \stackrel{\mathrm{def}}{=} \left( F(\mathrm{Ell}(A)),\mathrm{Ell}(A) \right) = ((\W(A),\langle 1_A \rangle),\mathrm{Ell}(A))
\cong ((\widetilde{\W}(A),[1_A]),\mathrm{Ell}(A)).
\]
\end{theorem}

Even in situations where (EC) holds, there is
no inverse functor which reconstructs C$^*$-algebras from Elliott invariants.
Contrast this with Theorem \ref{ecwec}, where $G$ reconstructs the finer invariant
from the coarser one.

\subsubsection{Revising the invariant}

Since the Elliott invariant is incomplete, generally speaking, one tends to enlarge it. How to do this is not as straitghforward as one may think, as it depends on the amount of new information we throw in. The counterexamples to the Elliott Conjecture already mentioned before, in their most dramatic form (see \cite{T2}), are only distinguished by their Cuntz semigroups, but agree on the Elliott invariant and quite a few enlargements of it that include topologically non-commutative invariants (such as the real rank and the stable rank).

We have shown in the previous section that, for $\Z$-stable algebras, the addition of the Cuntz semigroup to the Elliott invariant is not an addition (see Theorem \ref{ecwec}). This adds philosophical grounds to the explanation of the somewhat mysterious success of the classification, using solely the Elliott invariant, of algebras that do not have real rank zero (hence they do not have a huge supply of projections), so traces and functions defined on them account for purely positive elements in that setting.

Although we have not mentioned it, the (functorial) recovery of the Cuntz semigroup from the Elliott invariant holds for AH-algebras with slow dimension growth. This class is not known to be $\Z$-stable yet, but there is (positive) progress in this direction (as shown recently in \cite{dpt}). We haven't included the proof as it is rather involved.

One may also hope that the Cuntz semigroup is a helpful tool when attempting a classification for non-simple algebras. This has been accomplished by Ciuperca and Elliott, to classify approximate interval algebras using their Cuntz semigroups.

Another argument in favour of adding only the Cuntz semigroup to the current invariant relies on the possibility that $\Z$-stability is equivalent to strict comparison (in the simple, nuclear case). There has also been progress on this front, mainly due to Toms and Winter (see \cite{tw}).


\providecommand{\bysame}{\leavevmode\hbox
to3em{\hrulefill}\thinspace}


\end{document}